\documentclass[10pt]{amsart}
\usepackage[top=1.875in, bottom=1.875in, left=1.875in, right=1.875in]{geometry}

\usepackage{graphicx}
\usepackage{mathrsfs}
\usepackage[shortlabels]{enumitem}
\usepackage[all,poly,necula]{xy} 
\usepackage{amsmath, mathtools}
\usepackage{amssymb}
\usepackage{wasysym}
\usepackage{amsfonts}
\usepackage{latexsym}
\usepackage{wrapfig}
\usepackage{float}
\usepackage{verbatim}
\usepackage{hyperref}
\usepackage{url}
\usepackage{color, caption, subcaption}
\usepackage[dvipsnames]{xcolor}
\usepackage[normalem]{ulem}

\usepackage{pgf,tikz,pgfplots, tikz-cd}
\pgfplotsset{compat=1.15}
\usepackage{mathrsfs}
\usetikzlibrary{arrows}
\usetikzlibrary{decorations.pathreplacing, decorations.markings}
\usepackage[export]{adjustbox}

\definecolor{grey}{rgb}{0.8,0.8,0.8}

\newcommand{\C}{\mathcal{C}}

\newcommand{\M}{\mathbf{M}}

\newcommand{\bbP}{\mathbb{P}}
\newcommand{\Q}{\mathbb{Q}}
\newcommand{\R}{\mathbb{R}}
\renewcommand{\S}{\mathbf{S}}

\newcommand{\Z}{\mathbb{Z}}

\DeclareMathOperator{\moduleCategory}{\mathsf{mod}} \renewcommand{\mod}{\moduleCategory}
\DeclareMathOperator{\proj}{\mathsf{proj}}
\DeclareMathOperator{\inj}{\mathsf{inj}}
\DeclareMathOperator{\ind}{\mathsf{ind}}
\DeclareMathOperator{\add}{\mathsf{add}}
\newcommand{\Ext}{\operatorname{Ext}\nolimits}
\newcommand{\Hom}{\operatorname{Hom}\nolimits}
\newcommand{\End}{\operatorname{End}\nolimits}

\newcommand{\op}{{\mathsf{op}}}
\newcommand{\RHom}{\mathbf{R}\strut\kern-.2em\operatorname{Hom}\nolimits}

\newcommand{\Image}{\operatorname{Im}\nolimits}
\newcommand{\Kernel}{\operatorname{Ker}\nolimits}

\mathchardef\hy="2D

\newcommand{\SM}{(\mathbf{S}, \mathbf{M})}
\newcommand{\bfA}{\mathbf{A}}
\newcommand{\bfC}{\mathbf{C}}
\newcommand{\Ctcc}{\mathbf{C}_{\mathrm{2si}}}
\newcommand{\Cocc}{\mathbf{C}_{\mathrm{1si}}}
\newcommand{\Ccc}{\mathbf{C}_{\mathrm{cc}}}
\newcommand{\Cnc}{\mathbf{C}_{\mathrm{nc}}}

\newcommand{\ctilt}{\mathrm{c}\hy\mathrm{tilt}}
\renewcommand{\epsilon}{\varepsilon}

\newtheorem{theorem}{Theorem}[section]
\newtheorem*{theorem*}{Theorem}
\newtheorem{lemma}[theorem]{Lemma}
\newtheorem{proposition}[theorem]{Proposition}
\newtheorem{corollary}[theorem]{Corollary}

\theoremstyle{definition}
\newtheorem{definition}[theorem]{Definition}
\newtheorem{example}[theorem]{Example}

\newtheorem*{conjecture*}{Conjecture}

\theoremstyle{remark}
\newtheorem{remark}[theorem]{Remark}

\numberwithin{equation}{section}

\newcommand{\wti}[1]{\widetilde{#1}}

\usepackage{pgf,tikz,pgfplots}
\pgfplotsset{compat=1.15}
\usepackage{mathrsfs}

\definecolor{grey}{rgb}{0.75,0.75,0.75}

\definecolor{darkgreen}{rgb}{0,0.7,0}
\definecolor{darkblue}{rgb}{0,0,0.7} 
\newcommand{\dfn}[1]{\textsl{\color{darkblue} #1}}

\newcommand{\mt}[1]{\mathtt{#1}}
\newcommand{\biinf}[1]{{}^\infty{#1}^\infty}
\newcommand{\crosswn}[7]{
\left\{\begin{array}{rr|}{#1} = & {#2}\\ {#5} = & {#6} \end{array}
{\phantom{a}\underline{#3}\phantom{a}} \begin{array}{|l} {#4} \\ {#7} \end{array}\right.
}

\usetikzlibrary{arrows}
\newcommand{\xcap}[2][1]{
\filldraw[thick, fill=white, draw=black, shift={(#2)},scale=#1] (0,0) circle (0.4);
\draw[shift={(#2)},scale=#1] (45:0.4) -- +(225:0.8) (-45:0.4) -- +(-225:0.8);
}

\newcommand{\Mtwo}[2][1]{
\draw[thick, shift={(#2)}]  (0,0) circle (2);
\xcap{#2}[#1]
\fill[darkgreen, shift={(#2)}] (0,-2) circle (4pt);
\fill[darkgreen, shift={(#2)}] (0,2) circle (4pt);
}
\newcommand{\mkpt}[1]{\fill[darkgreen,shift={(#1)}] (0,0) circle (3pt);}

\newcommand{\Mtwofan}[1]{
\draw[thick,white!70!black,shift={(#1)}] (0,-2) -- +(0,4);
\draw[thick,white!70!black,shift={(#1)}] (0,-2) .. controls (-1.5,-1) and (-1,0) .. (0,0.1) .. controls (1,0) and (1.5,-1) .. (0,-2);
\Mtwo{#1}
}

\newcommand{\Atwotwo}[1]{
    \draw[thick,shift={(#1)}]  (0,0) circle (2);
    \filldraw[thick, fill=white!70!black,shift={(#1)}]  (0,0) circle (0.75);
    \draw[thick, white!70!black,shift={(#1)}] (0,-2) --  (0,-0.75);
    \draw[thick, white!70!black,shift={(#1)}] (0,2) -- (0,0.75);
    \draw[thick, white!70!black,shift={(#1)}] (0,-2) .. controls (2,-1) and (1.5,2) .. (0,0.75);
    \draw[thick,white!70!black,shift={(#1)}] (0,-2) .. controls (-2,-1) and (-1.5,2) ..  (0,0.75);
    \fill[darkgreen,shift={(#1)}] (0,-2) circle (3pt);
    \fill[darkgreen,shift={(#1)}] (0,2) circle (3pt);
    \fill[darkgreen,shift={(#1)}] (0,-0.75) circle (3pt);
    \fill[darkgreen,shift={(#1)}] (0,0.75) circle (3pt);
}

\newcommand{\sbm}[1]{{\let\amp=&\left[\begin{smallmatrix}#1\end{smallmatrix}\right]}}

\begin{document}

\title[Non-orientable surfaces and cluster categories]{Marked non-orientable surfaces and cluster categories via symmetric representations}

\author{V\'eronique Bazier-Matte}
\address{Université Laval, Québec (Québec), G1V 0A6, Canada  }
\email{veronique.bazier-matte.1@ulaval.ca}

\author{Aaron Chan}
\address{Graduate School of Mathematics, Nagoya University, Furocho, Chikusaku, Nagoya 464-8602, Japan}
\email{aaron.kychan@gmail.com}

\author{Kayla Wright}
\address{University of Minnesota, Twin Cities, Minneapolis, MN}
\email{kaylaw@umn.edu}

\thanks{}

\begin{abstract}
We initiate the investigation of representation theory of non-orientable surfaces.
As a first step towards finding an additive categorification of Dupont and Palesi's quasi-cluster algebras associated marked non-orientable surfaces, we study a certain modification on the objects of the cluster category associated to the orientable double covers in the unpunctured case. More precisely, we consider symmetric representation theory studied by Derksen-Weyman and Boos-Cerulli Irelli, and lift it to the cluster category. This gives a way to consider `indecomposable orbits of objects' under a contravariant duality functor. Hence, we can assign curves on a non-orientable surface $\SM$ to indecomposable symmetric objects. Moreover, we define a new notion of symmetric extension, and show that the arcs and quasi-arcs on $\SM$ correspond to the indecomposable symmetric objects without symmetric self-extension. Consequently, we show that quasi-triangulations of $\SM$ correspond to a symmetric analogue of cluster tilting objects.
\end{abstract}

\maketitle

\tableofcontents

\section{Introduction}
\subsection{Background}
Cluster algebras are certain commutative algebras, first axiomatized by Fomin and Zelevinsky \cite{FZ02}, that comes with a distinguished set of generators called \emph{cluster variables} organised in overlapping sets called \emph{clusters}. 
Since the birth of cluster algebras, connections between cluster theory and many mathematical fields have been discovered, such as combinatorics, dynamical systems, knot theory, mirror symmetry, etc.  The motivation of this article originates from the connection between cluster algebras arising from surface topology and representation theory \cite{GSV05, FG06, FST08, Lab09, ABCP10, BZ11}.

More precisely, the cluster algebras in this setting - usually called \emph{surface cluster algebras} -  comes from the coordinate ring of a certain Teichm\"uller space associated to an \emph{orientable} marked surface with triangulation.  The cluster variables are in correspondence with \emph{arcs} on the surface, i.e. self non-crossing curves that connect marked points; clusters are in correspondence with maximal sets of pairwise non-crossing arcs, i.e. triangulations. The combinatorial nature of these correspondence have been  fruitful to reveal the algebraic structure of surface cluster algebras; see, for example, been studied in \cite{MSW11, MSW13}. 

It is natural to ask for an extension to \emph{non-orientable} surfaces since one still have a coordinate ring for a Teichm\"{u}ller space associated to the surface. Their idea is to enlarge the set of arcs to include what they call \emph{quasi-arcs}.
A quasi-arc is a self-non-intersecting 1-sided closed curves, meaning that its cylindrical neighbourhood forms a M\"{o}bius strip. The analogue of clusters is now given by \emph{quasi-triangulations}, which are maximal sets of pairwise non-crossing arcs and quasi-arcs.  Thus far, a few classical results from Fomin-Shapiro's surface cluster algebras have been extended; see \cite{DP15, Wil19, Wil20}.

We are interested in the role of representation theory in the quasi-cluster structure associated to marked non-orientable surfaces, in particular the \emph{additive categorification} \cite{BMRRT, CCS06, CC06} of quasi-cluster algebras.
In the classical case, this is given by a triangulated category $\C$ called Amiot's \emph{cluster category}. The category $\C=\C_{Q,W}$ is determined by a \emph{quiver with potential} (QP) $(Q,W)$ associated to a chosen initial cluster -- or triangulation in the surface case \cite{DWZ10,Lab09}. It turns out that the indecomposable \emph{rigid} objects -- i.e. those without self-extensions -- in $\C_{Q,W}$ correspond to cluster variables, and the so-called \emph{cluster tilting} objects correspond to clusters; see, for example, \cite{Kel10}. In practice, one can analyse $\C_{Q,W}$ through the module category $\mod J_{Q,W}$ \cite{KZ08} of the \emph{Jacobian algebra} associated to $(Q,W)$.

In the surface cluster algebra case, it was shown in \cite{ABCP10} that $J_{Q,W}$ falls under a prominent class of algebras in representation theory called \emph{gentle algebras} \cite{AS87}. The modules over gentle algebras are well-understood \cite{BR} and can be calculated combinatorially.  In particular, an indecomposable module is either a \emph{string} module, which is described by (Dynkin) `type $\mathbb{A}$' combinatorics, or a \emph{band} module, which is described by (extended Dynkin) type $\tilde{\mathbb{A}}$ combinatorics along with a parameter $\lambda \in \Bbbk^\times$ from the underlying field $\Bbbk$. Br\"{u}stle and Zhang \cite{BZ11} used this to study the cluster category $\C_{Q,W}$ by showing the following dictionary (see Section \ref{sec:cluster cat basic} for details). Note that by \cite{KY11} $\C_{Q,W}$ is independent of the choice of initial triangulation, and so we can denote it by $\C_{\SM}$ where $\SM$ is the marked surface of interest.

\newcommand{\toppad}{\rule[.5ex]{0pt}{2.5ex}}
\newcommand{\botpad}{\rule[-1.5ex]{0pt}{2.6ex}}
\begin{center}
\begin{tabular}{c|c|c}
\hline
\toppad
& {\bf Topology} & {\bf Cluster category} 
\botpad \\
\hline  \hline 
\toppad
\hypertarget{corresp.A}{(A)} & curve $\gamma$ connecting marked points & indecomposable string object $\gamma$
\botpad \\
\hline
\toppad
\hypertarget{corresp.B}{(B)} & closed curve $\omega$ & indecomposable band objects \\
&with a parameter $\lambda$ in $\Bbbk^\times$ & $(\omega,\lambda)$
\botpad \\
\hline
\toppad
\hypertarget{corresp.C}{(C)} & arcs & rigid indecomposable objects
\botpad \\
\hline 
\toppad
\hypertarget{corresp.D}{(D)} & triangulations & cluster tilting objects
\botpad\\
\hline 
\end{tabular}
\end{center}

\subsection{Our goal}
We aim to extends the above dictionary to the case of non-orientable surfaces. To do this, we rely on the fact that any non-orientable surface always has an orientable double cover. By lifting a triangulation $T$ of a non-orientable marked surface $\SM$ to a triangulation $\wti{T}$ on the orientable double cover $\wti{\SM}$, we can then canonically associate a QP $(Q,W)$ and hence a cluster category $\C_{\wti{\SM}}=\C_{Q,W}$.  The associated Deck transformation group is generated by an orientation-reversing automorphism $\sigma_\S$, which gives rise to an (arrow-reversing) \emph{involution} $\sigma$ on the QP. We show that this gives rise to a contravariant duality $\nabla$ on both $\C_{\wti{\SM}}$ and $\mod J_{Q,W}$ (Proposition \ref{prop:lift nabla}). 

Next, we want a categorical way - such as an `orbit category $\C_{\wti{\SM}}/\nabla$' - to treat a $\nabla$-orbit of objects as indecomposable.  We have to address two problems here. Firstly, the classical orbit construction requires a covariant autoequivalence; however, $\nabla$ is \emph{contravariant}. As a consequence, we need a different theory to deal with $\nabla$-orbits. Secondly, a primitive closed curve $\wti{\omega}$ on $\wti{\SM}$ with $\sigma_\S(\wti{\omega})=\wti{\omega}$ is a lift of two different closed curves, say $\omega,\omega'$, on $\SM$. Indeed, one of $\omega,\omega'$ is 1-sided (has a unique lift) and the other is 2-sided (admits a 2-sheeted cover); see discussion at the end of Section \ref{sec:topology}. In view of the dictionary \hyperlink{corresp.B}{(B)}, a good orbit construction should induce a partition of $\{(\wti{\omega},\lambda)\in \C_{\wti{\SM}}\mid \lambda\in\Bbbk^\times\}$ into two subsets, with one corresponds to $\omega$ and the other to $\omega'$.

To this end, we employ the \emph{symmetric representation theory}, developed by Derksen and Weyman \cite{DW02} as well as Boos and Cerulli Irelli \cite{BCI21}, associated to algebras with duality -- such as $(J_{Q,W},\sigma)$ in our setting. A \emph{symmetric representation} is an ordinary representation equipped with some extra data that forces each dual pair $(\alpha,\sigma(\alpha))$ of arrows of $Q$ to act adjointly, see Section \ref{sec:symm repn}. Although symmetric representations do not form an additive category, there is still a natural notion \emph{indecomposability}.  It was shown in \cite{DW02,BCI21} (see Proposition \ref{prop:e-indec}) that every indecomposable symmetric representation $X$ is uniquely determined by the $\nabla$-orbit of an indecomposable (ordinary) module $M$ in one of the following forms.

\begin{center}
\begin{tabular}{ll}
(1) Split type: & $X=M\oplus\nabla M$ for $M\ncong \nabla M$.  \\
(2) Ramified type: & $X=M\oplus \nabla M$ for $M\cong \nabla M$.   \\
(3) 1-sided type: & $X=M$ for $M\cong \nabla M$.
\end{tabular}
\end{center}

Conversely, every indecomposable module $M$ give rise to exactly one of these indecomposable symmetric representations. This trichotomy closely resembles the behaviour of curves on $\SM$, with split and ramified type corresponds to curves admiting a 2-sheeted cover whereas 1-sided type corresponds to 1-sided (closed) curves.
Indeed, by fully classifying indecomposable symmetric representations over $J_{Q,W}$ (Theorem \ref{thm:indec e-reps}), we have that for a closed curve $\wti{\omega}$ on $\wti{\SM}$ with $\sigma_\S(\wti{\omega})=\wti{\omega}$, there is a unique parameter $\lambda\in\Bbbk^\times$ so that $(\widetilde{\omega},\lambda)$ gives rise to a 1-sided indecomposable symmetric representation.  In other words, the notion of indecomposable symmetric representations does satisfy our desired criteria of good $\nabla$-orbit on the level of $\mod J_{Q,W}$.

We next define the notion of \emph{symmetric objects} (respectively \emph{indecomposable symmetric objects}) (Definition \ref{def:e-obj}) in the category $\C_{\wti{\SM}}$, which allows us to lift symmetric (respectively indecomposable symmetric) representations from $\mod J_{Q,W}$ to $\C_{\wti{\SM}}$. Combining with the dictionary \hyperlink{corresp.A}{(A)} and \hyperlink{corresp.B}{(B)} from \cite{BZ11}, this allows us to write down the following correspondences.

\begin{theorem}\label{thm:main1}{\rm (Theorem \ref{thm:curve corresp})}
There are the following bijections between curves on $\SM$ and indecomposable symmetric objects of the cluster category $\C_{\wti{\SM}}$.
\begin{align*}
\xymatrix@R=3pt{
{\left\{\begin{array}{c}\text{curves connecting}\\\text{marked points}\end{array}\right\}}
 \ar@{<->}[r]^(0.4){1:1} & 
{\left\{\begin{array}{c}\text{indecomposable symmetric objects}\\\text{of split string type}\end{array}\right\}},\\
{\left\{\begin{array}{c}\text{primitive one-sided}\\\text{closed curves}\end{array}\right\}}
\ar@{<->}[r]^(0.4){1:1} & 
{\left\{\begin{array}{c}\text{indecomposable symmetric objects}\\\text{of one-sided primitive band  type}\end{array}\right\}}.
}
\end{align*}
\end{theorem}
We note that the set on the right in the second row is discrete, as opposed to a continuous family in \hyperlink{corresp.B}{(B)}.  We omit the statement for the 2-sided closed curves due to technicalities; see Theorem \ref{thm:curve corresp} for details.

Our next goal is to find a categorical criteria for indecomposable symmetric objects that characterises self-non-crossing property of a curve on $\SM$ - these are encoded by vanishing of extensions of objects in the classical case.
As an analogue, we define \emph{symmetric extensions} (Definition \ref{def:e-ext}) for symmetric objects.  Using this, we further define the notion of \emph{symmetric rigid objects} and \emph{symmetric cluster tilting objects}, which yields the following correspondences extending \hyperlink{corresp.C}{(C)} and \hyperlink{corresp.D}{(D)}.

\begin{theorem}{\rm (Theorem \ref{thm:categorification})}
For a marked non-orientable unpunctured surface $\SM$ with orientable double cover $\wti{\SM}$, the correspondences in Theorem \ref{thm:main1} restricts to
\begin{align*}
\xymatrix@1{
{\left\{\begin{array}{c}\text{arcs and quasi-arcs}\\\text{of } \SM\end{array}\right\}}
 \ar@{<->}[r]^(0.4){1:1} & 
 {\left\{\begin{array}{c}\text{indecomposable symmetric}\\\text{symmetric rigid objects of }\C_{\wti{\SM}}\end{array}\right\}}.
}
\end{align*}
This induces the following correspondence
\begin{align*}
\xymatrix@1{
{\left\{\begin{array}{c}\text{quasi-triangulations}\\\text{of } \SM\end{array}\right\}}
 \ar@{<->}[r]^(0.47){1:1} & 
 {\left\{\begin{array}{c}\text{symmetric cluster tiling}\\\text{ objects of }\C_{\wti{\SM}}\end{array}\right\}}.
}
\end{align*}
\end{theorem}

\subsection{Future Direction}
A more precise additive categorification of a cluster algebra $\mathcal{A}$ asks not just correspondences between rigid/cluster tilting objects of a category $\C$ and cluster variables/clusters of $\mathcal{A}$, but also existence of a  \emph{cluster character} \cite{Palu08, Pla11, Pla18} (a.k.a. Caldero-Chapton map) $\chi: \mathrm{ob}\,\C \to \mathcal{A}$ that satisfies additional properties.
Having the correspondence is only a first step; in a future project, we will look at analogue of cluster characters for quasi-cluster algebras associated to non-orientable surfaces.

\subsection{Structure of the paper}
Our paper is structured as follows.
In Section \ref{sec:topology}, we review the necessary topological theory we use in this article; namely, non-orientable surfaces, quasi-arcs, quasi-triangulations, double cover. In Section \ref{sec:cluster cat basic}, we review the algebraic prerequisites.  
This includes QP's associated to orientable triangulated surfaces, the arising gentle Jacobian algebra, and cluster category. In Section \ref{sec:involutionandduality}, we describe a certain type of involutions on the QP's we use. 
This gives rise to a contravariant duality on the cluster categories and the module categories. 
In subsection \ref{subsec:duality functor}, we justify that such a duality categorifies the orientation-reversing automorphism on the double cover defining the non-orientable surface of interest. 
In Section \ref{sec:symm repn}, by thoroughly analysing the possible symmetric structure on the modules of the arising Jacobian algebras, we prove our first main result (Theorem \ref{thm:indec e-reps}) -- the classification of indecomposable symmetric representations. In the final Section \ref{sec:categorify}, we lift symmetric representation theory from the module category to the cluster category, which yields the correspondence (Theorem \ref{thm:curve corresp}) between curves and symmetric objects.  
Moreover, we define symmetric analogue of rigid objects and cluster tilting objects, which lead us to the final main result (Theorem \ref{thm:categorification}), namely, the categorification of quasi-triangulations of non-orientable surfaces.
\medskip

\subsubsection{Conventions}  Throughout, we assume any underlying field $\Bbbk$ is algebraically closed.  For any ($\Bbbk$-)algebra $\Lambda$ is assumed to be finite-dimensional unless otherwise specified.  The category of (finitely generated) $\Lambda$-modules is denoted by $\mod\Lambda$.  The full subcategory of finitely generated projective $\Lambda$-modules is denoted by $\proj\Lambda$.  For an additive category $\C$ and any object $X\in \C$, denote by $\add_{\C}(C)$, or simply $\add(C)$ if there is no confusion, the additive closure of $C$ in $\C$.

\subsubsection{Acknowledgements} The authors would like to thank Pierre-Guy Plamondon for connecting us together and inspiring us to investigate the algebraic side of quivers arising from non-orientable surfaces. We would also like to thank Osamu Iyama for the insightful comments towards the editing and presentation of the paper. Additionally, we thank Nancy Scherich, Yann Palu, and Toshiya Yurikusa for helpful conversations. The first named author was supported by the Postdoctoral Research Scholarship B3X from Fonds de Recherche du Québec - Nature et Technologies. The second named author is supported by JSPS Grant-in-Aid for Research Activity Start-up program 19K23401.

\section{Marked surfaces}\label{sec:topology}
In this section, we review the necessary topological notions that will be used throughout the paper, namely, the nuances of arcs and triangulations of non-orientable surfaces. 

By a \dfn{surface}, we mean a compact, connected, real $2$-dimensional, (possibly non-orientable) surface $\S$ with non-empty boundary $\partial \S$.  
Let $\M$ be a set of \dfn{marked points}, that is, a finite discrete set of points in $\partial\S$ such that $\M\cap \partial\neq\emptyset$ for each connected component $\partial$ of $\partial\S$.  Such a pair $\SM$ is called a \dfn{marked surface}.  We will assume marked surfaces can be non-trivially triangulated, i.e. $\SM$ is not a monogon, digon, or a triangle.

Suppose that $\SM$ is a marked \textit{non-orientable} surface.
Up to homeomorphism, $\S$ is the connected sum of $k$ projective planes $\R\bbP^2$.  The number $k$ is called the \dfn{non-orientable genus}, or simply just the genus if no confusion arises. Note that any non-orientable surface $N_k$ of genus $k$ admits an orientable double cover $\wti{N_k}$ of genus $k-1$. 

In order to visualize these surfaces, recall that $\R\bbP^2$ is the quotient of the 2-sphere $S^2$ under the antipodal map, which leads to a practical representation by drawing a \dfn{crosscap} $\bigotimes$.  Figure \ref{Fig::crosscap} shows how attaching a crosscap to a disk gives a surface homeomorphic to the M\"{o}bius strip.  More generally, $N_k$ can be drawn as a disk with $k$ crosscaps attached, \textit{or equivalently}, an orientable genus $g\geq 0$ surface with $k-2g \geq 1$ crosscaps attached.

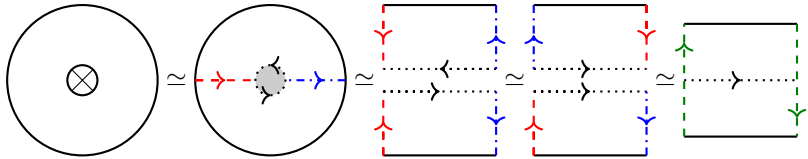
\begin{figure}[!htbp]\centering
\begin{tikzpicture}[scale=0.5]
\tikzset{
	capCut/.style={thick,dotted,->},
	redCut/.style={thick,red,dashed,->},
	blueCut/.style={thick,blue,dash dot,->},
}
\draw[thick]  (0,0) circle (2);\xcap{0,0}\node at (2.5,0) {$\simeq$};
\draw[thick]  (5,0) circle (2);
\fill[gray!40!white] (5,0) circle (0.4);
\draw[capCut] (5,-0.4) arc (-90:90:0.4);
\draw[capCut] (5,0.4) arc (90:270:0.4);
\draw[redCut] (3,0) -- +(1.6,0) (3,0) -- +(0.8,0);
\draw[blueCut] (7,0) -- +(-1.6,0)  (6.3,0) -- +(0.1,0);
\node at (7.5,0) {$\simeq$};

\draw[thick] (8,2) -- +(3,0) (8,-2) -- +(3,0);
\draw[capCut] (8,0.3) -- +(3,0) (11,0.3) -- +(-1.5,0);
\draw[redCut] (8,2) -- +(0,-1.7) (8,2) -- +(0,-0.9);
\draw[blueCut] (11,2) -- +(0,-1.7) (11,1.1) -- +(0,0.1);
\draw[capCut] (8,-0.3) -- +(3,0) (8,-0.3) -- +(1.5,0);
\draw[redCut] (8,-2) -- +(0,1.7) (8,-2) -- +(0,0.9);
\draw[blueCut] (11,-2) -- +(0,1.7) (11,-1.1) -- +(0,-0.1);
\node at (11.5,0) {$\simeq$};

\begin{scope}[shift={(4,0)}]
\draw[thick] (8,2) -- +(3,0) (8,-2) -- +(3,0);
\draw[capCut] (8,0.3) -- +(3,0) (8,0.3) -- +(1.5,0);
\draw[redCut] (11,2) -- +(0,-1.7) (11,2) -- +(0,-0.9);
\draw[blueCut] (8,2) -- +(0,-1.7) (8,1.1) -- +(0,0.1);
\draw[capCut] (8,-0.3) -- +(3,0) (8,-0.3) -- +(1.5,0);
\draw[redCut] (8,-2) -- +(0,1.7) (8,-2) -- +(0,0.9);
\draw[blueCut] (11,-2) -- +(0,1.7) (11,-1.1) -- +(0,-0.1);
\node at (11.5,0) {$\simeq$};
\end{scope}

\draw[thick] (16,1.5) -- +(3,0) (16,-1.5) -- +(3,0);
\draw[thick,black!50!green,dashed,->] (16,1) -- +(0,0.5) (16,-1.5) -- +(0,2.5);
\draw[thick,black!50!green,dashed,->] (19,-1) -- +(0,-0.5) (19,1.5) -- +(0,-2.5);
\draw[capCut] (17.5,0) -- +(1.5,0) (16,0) -- +(1.5,0);
\end{tikzpicture}
\caption{Identifying crosscap-attached disk with the M\"obius strip.} \label{Fig::crosscap}\end{figure}

Unless otherwise specified, a \dfn{curve} in $\SM$ is (the image of) a continuous function $\gamma:I\to \S$ such that
\begin{itemize}
\item either $\gamma$ is a \dfn{curve with endpoints} (in $\M$), i.e. $I\simeq [0,1]$ with $\gamma\cap \partial \S=\{I(0), I(1)\}\subset \M$ and $\gamma$ does not cut out a monogon containing no marked point when $I(0)=I(1)$,

\item or $\gamma$ is a \dfn{closed curve}, i.e. $I\simeq S^1$ with $\gamma\cap \partial \S=\emptyset$ and $\gamma$ non-contractible.
\end{itemize}
We always work with curves up to isotopy relative its endpoints (whenever `endpoints' make sense).  In particular, two curves bounding a marked point are considered distinct.We will use the following notation:
\begin{align*}
\Cnc\SM &:= \{ \text{(isotopy classes of ) curves with endpoints in }\M \} \\
\Ccc\SM &:= \{ \text{(isotopy classes of ) closed curves in }\SM \} \\
\bfC\SM &:= \Cnc\SM\sqcup\Ccc\SM
\end{align*}

A \dfn{boundary arc} is a curve that is isotopic to an interval on $\partial\S$ with endpoints in $\M$.  A \dfn{regular arc} $\gamma$ is a curve in $\SM$ with endpoints in $\M$ that is {\it not a boundary arc} and has no self-intersections (i.e. \dfn{simple}), except possibly at its endpoints.  Denote by $\bfA\SM$ the set of regular arcs in $\SM$.

Every curve on $\SM$ can be drawn as a curve on an orientable surface with crosscaps, such that when it hit a crosscap, it needs to come out from the antipodal point of this crosscap.  See, for example, how the `arc' given by the (concatenation of) red-and-blue (dashed, then dash-dotted) line in Figure \ref{Fig::crosscap} represents the vertical (green) oriented curve) on the M\"{o}bius strip on the far-right.  

\begin{example}\label{example: regulararc}
\begin{figure}[!htbp]
\centering
\begin{tikzpicture}[scale=0.6]
\begin{scope}[shift={(-5,0)}]
\draw[very thick,blue] (0,0) .. controls (0,1.5) and (1,0.5) .. (1,0) 
	.. controls (1,-1) and (-1,-1) .. (-1,0) node[black,left] {$\alpha$} 
	.. controls (-1,0.5) and (-1,1) .. (0,2);
\draw[very thick,blue] (0,-2) -- +(0,2);
\Mtwo{0,0}
\end{scope}

\begin{scope}
\draw[very thick,blue] (0,-2) .. controls (-1.5,-1) and (-1,0) .. (0,0) .. controls (1,0) and (1.5,-1) .. (0,-2);
\node at (-1,0) {$\beta$};
\Mtwo{0,0}
\end{scope}

\begin{scope}[shift={(10,0)}]
\draw[very thick,blue]  (0,0) ellipse (1.2 and 1.2);
\node at (-1.6,0) {$\kappa$};
\Mtwo{0,0}
\end{scope}

\begin{scope}[shift={(5,0)}]
\draw[very thick,blue] (0,-0.4) .. controls (0,-1) and (1,-0.5) .. (1,0)
	.. controls (1,0.5) and (0,1) .. (0,0.4);
\node at (1.6,0) {$\omega$};
\Mtwo{0,0}
\end{scope}
\end{tikzpicture}\caption{Examples of non-closed and closed curves.}\label{fig:eg1}
\end{figure}
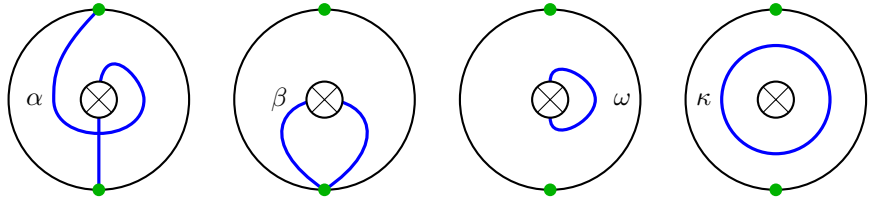
Consider the curves $\alpha,\beta,\gamma$ in Figure \ref{fig:eg1}.
We have $\alpha, \beta \in \Cnc\SM$ and $\omega,\omega' \in \Ccc\SM$.  Note that $\alpha$ is not a regular arc as it contains a self-intersection, whereas $\beta$ is a regular arc as the only self-intersection is at its endpoints.
\end{example}

\begin{example}\label{example: C(S,M)}
There are five regular arcs in the M\"obius strip with two marked points.  See Figure \ref{fig:flip-M2s}.
\end{example}

Recall that a closed curve on $\SM$ is \dfn{2-sided} if local orientation is preserved when traversing along itself; \dfn{1-sided}, otherwise.  This is equivalent to say that it has a 2-sheeted cover or a unique lift on the orientable double cover; see the discussion at the end of this section. A simple closed curve is 1-sided if and only if the number of crosscaps it goes through is odd \cite{oddcrosscapresult}.
Denote by $\Cocc\SM$ the set of 1-sided closed curves and by $\Ctcc\SM$ the set of 2-sided closed curves. So we have $\Ccc\SM = \Cocc\SM\sqcup\Ctcc\SM$.

The dotted curve in the far-right of Figure \ref{Fig::crosscap} shows the \dfn{core} of a M\"{o}bius strip, i.e. a 1-sided simple closed curve whose neighbourhood is homeomorphic to a M\"{o}bius strip.  We call such a curve a \dfn{quasi-arc}.  The simplest pictorial representation of a quasi-arc is given by drawing a closed curve through a crosscap once; see Figure \ref{fig:quasi-arc}.

\begin{figure}[!htbp]
\centering
\begin{tikzpicture}[scale=0.6]
\draw[thick]  (0,0) circle (2);\xcap{0,0}\node at (3,0) {$\simeq$};
\draw[thick,red] (0,-0.4) .. controls (0,-1) and (1,-0.5) .. (1,0)
	.. controls (1,0.5) and (0,1) .. (0,0.4);

\draw[thick] (4,1.5) -- +(3,0) (4,-1.5) -- +(3,0);
\draw[dashed,->] (4,1) -- +(0,0.5) (4,-1.5) -- +(0,2.5);\draw[dashed,->] (7,-1) -- +(0,-0.5) (7,1.5) -- +(0,-2.5);
\draw[thick,red] (7,0) .. controls (6.5,-0.5) and (5.5,-0.5) .. (5.5,0) .. controls (5.5,0.5) and (4.5,0.5) .. (4,0);
\node at (8,0) {$\simeq$};

\begin{scope}[shift={(5,0)}]
\draw[thick] (4,1.5) -- +(3,0) (4,-1.5) -- +(3,0);
\draw[dashed,->] (4,1) -- +(0,0.5) (4,-1.5) -- +(0,2.5);\draw[dashed,->] (7,-1) -- +(0,-0.5) (7,1.5) -- +(0,-2.5);
\draw[red,thick] (4,0) --+(3,0);
\end{scope}
\end{tikzpicture}
\caption{Quasi-arc on the M\"{o}bius strip}\label{fig:quasi-arc}
\end{figure}
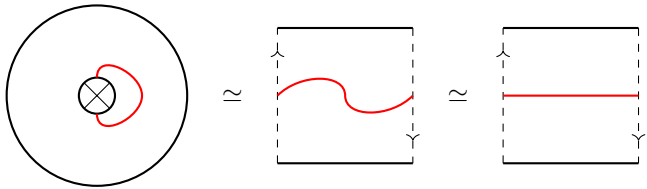

Quasi-arcs can take more complicating forms.  For example, Figure \ref{fig:3xcaparc} shows a quasi-arc that goes through three crosscaps on $(\R\bbP^2)^{\#3}$ with a disk removed. This can be turned back into the more familiar form by using a different representation of the same surface, as a sphere attached with 3 crosscaps is homeomorphic to a torus attached with 1 crosscap.
\begin{figure}[!htbp]
\centering
\begin{tikzpicture}[scale=0.55]
\draw[thick]  (0,0) circle (2);
\draw[very thick,blue] (0,0) circle (1.2);
\xcap{-40:1.2}\xcap{90:1.2}\xcap{-140:1.2}
\draw  (6,0) ellipse (3 and 2);
\draw (7.5,-0.05) arc (30:150:1.75 and 1);
\draw (4,0.2) arc (-150:-30:2.3 and 1);
\draw[very thick,blue] (4.47,-0.7) circle (0.4);
\draw[color=black, thick, fill=gray] (8,-0.7) circle (0.3);
\xcap[0.7]{4.2,-0.7}
\node at (2.5,0) {$=$};
\node at (-0.5,0.4) {$\alpha$};
\node at (5.2,-1) {$\alpha$};
\end{tikzpicture}\caption{Quasi-arc that goes through three crosscaps}\label{fig:3xcaparc}
\end{figure}
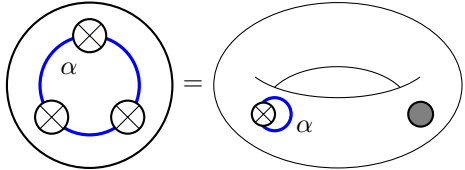

From now on, by an \dfn{arc}, or \dfn{internal arc} if we want to emphasis its nature, we mean a regular arc or a quasi-arc.  Denote by $\bfA^\otimes\SM$ the set of all arcs.  Next, we would like to talk about \emph{triangulations} (and their extended notion) formed by curves of $\SM$.  This requires the notion of \emph{non-crossing}.

Recall that two curves (in particular, two arcs) $\gamma, \gamma'$ in a (orientable or non-orientable) marked surface $\SM$ are \dfn{non-crossing} if, up to isotopy, they do not intersect one another except at their endpoints.  More formally, define
\begin{align*}
\text{int}(\gamma,\gamma') & := \min \{ \alpha\cap \alpha' \mid \alpha \simeq \gamma \text{ and } \alpha' \simeq \gamma'\}
\end{align*}
\noindent where $\alpha$ (respectively $\alpha'$) ranges over curves isotopic to $\gamma$ (respectively $\gamma'$), and define also
\begin{align*}
\text{cross}(\gamma, \gamma') & := \begin{cases}
\text{int}(\gamma,\gamma')\setminus\{\gamma(0),\gamma(1)\}, & \text{if both $\gamma, \gamma'$ are not closed};\\
\text{int}(\gamma,\gamma'), & \text{otherwise}.
\end{cases}   
\end{align*}
Then, $\gamma$ and $\gamma'$ being non-crossing is the same as saying  $\text{cross}(\gamma, \gamma')=\emptyset$.  

With the crosscap representation of non-orientable surfaces, two arcs may appear to have a crossing at the crosscap, but they actually do not cross in reality. 

\begin{definition}\label{definition :: quasitriangulation}
A \dfn{quasi-triangulation} (respectively \dfn{triangulation}) of $\SM$ is a maximal collection $T\subset \bfA^\otimes\SM$ (respectively $T\subset \bfA\SM$) of pairwise non-crossing arcs (respectively regular arcs).
\end{definition}

See Figure \ref{fig:flip-M2s} for all the triangulations and quasi-triangulations of the M\"{o}bius strip with 2 marked points $\mathcal{M}_2$.

\begin{figure}[!htbp]
\centering
\begin{tikzpicture}[scale=0.5]
\begin{scope}[shift={(0,0)}]
\draw[very thick,blue] (0,-2) .. controls (-2.5,-1) and (-1.5,1) .. (0,1) .. controls (1.5,1) and (2.5,-1) .. (0,-2);
\draw[very thick,orange]  (0,-0.3) ellipse (0.75 and 0.4);
\Mtwo{0,0}
\end{scope}
\begin{scope}[shift={(-6,0)}]
\draw[very thick,blue] (0,-2) .. controls (-2.5,-1) and (-1.5,1) .. (0,1) .. controls (1.5,1) and (2.5,-1) .. (0,-2);
\draw[very thick,orange] (0,-2) .. controls (-1.5,-1) and (-1,0) .. (0,0) .. controls (1,0) and (1.5,-1) .. (0,-2);
\Mtwo{0,0}
\end{scope}
\begin{scope}[shift={(-12,0)}]
\draw[very thick,blue] (0,-2) -- +(0,4);
\draw[very thick,orange] (0,-2) .. controls (-1.5,-1) and (-1,0) .. (0,0) .. controls (1,0) and (1.5,-1) .. (0,-2);
\Mtwo{0,0}
\end{scope}

\begin{scope}[shift={(0,-6)},yscale=-1]
\begin{scope}[shift={(0,0)}]
\draw[very thick,blue] (0,-2) .. controls (-2.5,-1) and (-1.5,1) .. (0,1) .. controls (1.5,1) and (2.5,-1) .. (0,-2);
\draw[very thick,orange]  (0,-0.3) ellipse (0.75 and 0.4);
\Mtwo{0,0}
\end{scope}
\begin{scope}[shift={(-6,0)}]
\draw[very thick,blue] (0,-2) .. controls (-2.5,-1) and (-1.5,1) .. (0,1) .. controls (1.5,1) and (2.5,-1) .. (0,-2);
\draw[very thick,orange] (0,-2) .. controls (-1.5,-1) and (-1,0) .. (0,0) .. controls (1,0) and (1.5,-1) .. (0,-2);
\Mtwo{0,0}
\end{scope}
\begin{scope}[shift={(-12,0)}]
\draw[very thick,blue] (0,-2) -- +(0,4);
\draw[very thick,orange] (0,-2) .. controls (-1.5,-1) and (-1,0) .. (0,0) .. controls (1,0) and (1.5,-1) .. (0,-2);
\Mtwo{0,0}
\end{scope}
\end{scope}

\draw (-9.5,0) -- +(1,0) (-3.5,0) -- +(1,0) (-9.5,-6) -- +(1,0) (-3.5,-6) -- +(1,0);
\draw (-12,-2.5) -- +(0,-1) (0,-2.5) -- +(0,-1);

\draw [decorate,decoration={brace,amplitude=10pt}]
(-14.5,2.5) -- (-3.5,2.5) node [midway,above,yshift={10}] {Triangulations};
\draw [decorate,decoration={brace,amplitude=10pt}]
(-2.5,2.5) -- (2.5,2.5) node [midway,above,yshift={10}] {Quasi-triangulations};
\end{tikzpicture}\caption{Flip graph of $\mathcal{M}_2$}\label{fig:flip-M2s}
\end{figure}
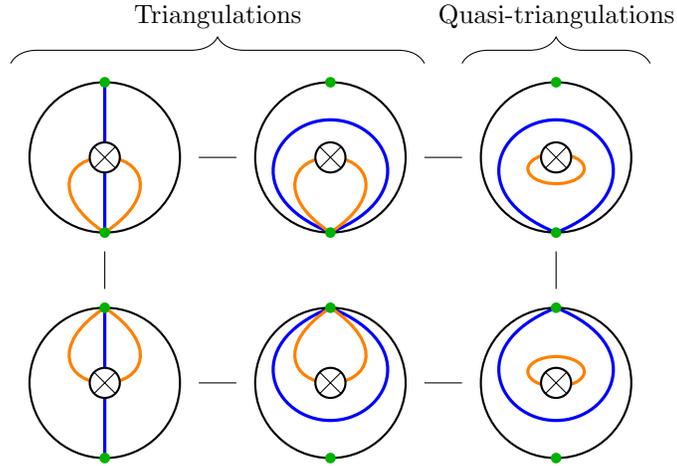

\begin{definition}\label{def:double cover}
Let $\SM$ be a marked unpunctured (not necessarily orientable) surface.
A(n orientable) \dfn{double cover} of $\SM$ is a double cover $p:\wti{\S}\to \S$ of surfaces with orientable $\wti{\S}$ and Deck transformation group $\{1,\sigma\}$ for some orientation-reversing automorphism $\sigma:\wti{\S}\to\wti{\S}$, such that $p^{-1}(\partial\S) = \partial \wti{\S}$ and  $\wti{\SM}=(\wti{\S},\wti{\M}:=p^{-1}(\M))$ is an orientable marked (unpunctured) surface.
In such a case, if $T$ is a triangulation of $\SM$, then we call the preimage $\wti{T}:= p^{-1}(T)$ of $T$ the double cover of $T$.
\end{definition}

We will omit the covering map $p$ and just say that $(\wti{\SM},\sigma)$, or even just $\wti{\SM}$, is an orientable double cover of $\S$. In Figure \ref{fig:fan}, we have $\S=\mathcal{M}_{2n}$ the M\"{o}bius strip with $2n$ marked points on the left-hand side, and its double cover -- an annulus with $2n$ marked points on each boundary component -- on the right-hand side.  We show also a triangulation $T=\{1,2\ldots, n\}$ on the M\"{o}bius strip and its double cover $\wti{T}=\{1,1', 2,2', \ldots, n,n'\}$.  We call this the \dfn{fan triangulation} of $\mathcal{M}_{2n}$.

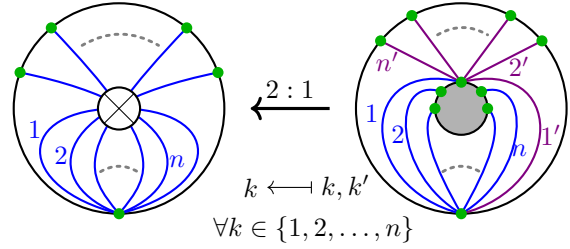
\begin{figure}[!htbp]
\centering
\begin{tikzpicture}[scale=0.7]
    \draw[thick]  (0,0) circle (2);
    \filldraw[thick, fill=white!70!black]  (0,0) circle (0.5);
    \draw[thick, blue!50!red] (0,0.5) -- (140:2); \mkpt{140:2}
    \draw[thick, blue!50!red] (0,0.5) -- (118:2); \mkpt{118:2}
    \draw[thick, blue!50!red,xscale=-1] (0,0.5) -- (140:2); \mkpt{40:2}
    \draw[thick, blue!50!red,xscale=-1] (0,0.5) -- (118:2); \mkpt{62:2}
\draw[line width= 1.1pt,
      dash pattern=on 1.1pt off 2.2pt,
      line cap=round, black!50!white]  (-0.45,1.4) to[bend left] (0.5,1.4);
    \draw[thick, blue] (0,-2) .. controls +(150:1.5) and (155:1.5) .. (140:0.5); \mkpt{140:0.5}
    \draw[thick, blue] (0,-2) .. controls +(120:1) and (-160:0.9) .. (180:0.5); \mkpt{180:0.5}
    \draw[thick, blue, xscale=-1] (0,-2) .. controls +(150:1.5) and (155:1.5) .. (140:0.5); \mkpt{40:0.5}
    \draw[thick, blue, xscale=-1] (0,-2) .. controls +(120:1) and (-160:0.9) .. (180:0.5); \mkpt{0:0.5}
\draw[line width= 1.1pt,
      dash pattern=on 1.1pt off 2.2pt,
      line cap=round, black!50!white] (-0.35,-1.2) to[bend left] (0.35,-1.2);
    
    \draw[thick, blue!50!red] (0,-2) .. controls (2,-1.5) and (2,1) ..  (0,0.5);
    \draw[thick, blue] (0,-2) .. controls (-2,-1.5) and (-2,1) ..  (0,0.5);
    \mkpt{0,-2};\mkpt{0,0.5};

    \draw[very thick, ->] (-2.5,0) -- node[midway,above] {$2:1$} +(-1.5,0);
\begin{scope}[shift={(-6.5,0)}]
\draw[thick]  (0,0) circle (2);
\draw[thick,blue] (-90:2) .. controls +(20:1.5) and (0:1.5) .. (160:2); \mkpt{160:2}
\draw[thick,blue] (-90:2) .. controls +(50:1) and (-40:1.2) .. (130:2); \mkpt{130:2}
\draw[thick,blue,xscale=-1] (-90:2) .. controls +(20:1.5) and (0:1.5) .. (160:2); \mkpt{20:2}
\draw[thick,blue,xscale=-1] (-90:2) .. controls +(50:1) and (-40:1.2) .. (130:2); \mkpt{50:2}
\draw[thick,blue] (0,-2) .. controls (-2,-1.5) and (-2,0) .. (0,0) .. controls (2,0) and (2,-1.5) .. (0,-2);
\xcap{(0,0)}
\draw[line width= 1.1pt,
      dash pattern=on 1.1pt off 2.2pt,
      line cap=round, black!50!white] (-0.35,-1.2) to[bend left] (0.35,-1.2);
\draw[line width= 1.1pt,
      dash pattern=on 1.1pt off 2.2pt,
      line cap=round, black!50!white] (-0.7,1.2) to[bend left] (0.75,1.2);
\mkpt{0,-2};
\end{scope}
\node[blue,thick] at (-8.1,-0.4) {$1$};
\node[blue,thick] at (-7.6,-0.9) {$2$};
\node[blue,thick] at (-5.4,-1) {$n$};

\node[blue,thick] at (-1.7,-0.1) {$1$};
\node[blue,thick] at (-1.2,-0.5) {$2$};
\node[blue,thick] at (1.1,-0.8) {$n$};
\node[blue!50!red,thick] at (-1.4,0.9) {$n'$};
\node[blue!50!red,thick] at (1.1,0.8) {$2'$};
\node[blue!50!red,thick] at (1.7,-0.5) {$1'$};

\node (src) at (-2.2,-1.5) {$k,k'$};\node (tgt) at (-4,-1.5) {$k$};
\draw[|->] (src) -- (tgt);
\node at (-2.8,-2.3) {$\forall k\in\{1,2,\ldots, n\}$};
\end{tikzpicture}
\caption{Fan triangulation of M\"{o}bius strip and its double cover}\label{fig:fan}
\end{figure}

Recall that a closed curve $\omega\in\Ccc\SM$ is \dfn{primitive} if  for it cannot be written as $\kappa^w$ for some $\kappa\in \Ccc\SM$ with $r>1$.
Therefore, we have a partition $\Ccc\SM=\sqcup_{r\geq 1}\Ccc^r\SM$ where $\Ccc^r\SM$ are closed curves of the form $\omega^r$ for a primitive $\omega\in \Ccc^1\SM$.
Likewise, we write $\Cocc^1\SM$ and $\Ctcc^1\SM$ for the sets of primitive 1-sided and 2-sided closed curves respectively.
We can characterises the whether a closed curve $\omega\in\Ccc^1\SM$ is 1- or 2-sided from its lift in $\wti{\SM}$ as follows:
\begin{align*}
\omega\in \Cocc^1\SM & \Leftrightarrow p^{-1}(\omega) \text{ is a single closed curve on $\wti{\SM}$,}\\
\omega\in \Ctcc^1\SM & \Leftrightarrow p^{-1}(\omega) = \wti{\omega}\sqcup \sigma(\wti{\omega})\text{ for some lift }\wti{\omega}\in\Ccc\wti{\SM}.
\end{align*}
In particular, if $\wti{\omega}:=p^{-1}(\omega)$ is the unique lift of $\omega\in\Cocc^1\SM$, then there is also a 2-sided closed curve $\kappa$ such that $p^{-1}(\kappa)$ is (isotopic to) the disjoint union of \emph{two copies of $\wti{\omega}$}.
An example of such a pair $(\omega,\kappa)$ is already shown in Example \ref{fig:eg1}.

\section{Quiver with potential, Jacobian algebra, and cluster category}\label{sec:cluster cat basic}

In this section, we review the notion of quivers with potential associated to an orientable surface. We discuss two algebraic objects one can associate to such a quiver with potential: the cluster category and the Jacobian algebra. In this, we review string and band modules arising from arcs and closed curves on a surface.

\subsection{QP and Jacobian algebra}
Suppose $\SM$ is a marked \emph{orientable} surface and $T$ is a triangulation on $\SM$.
Recall from \cite{Lab09} that one can associate to $T$ is \emph{quiver with potential} (QP) $(Q_T,W_T)$ given by 
\begin{itemize}
\item The quiver $Q_T$ has vertices being the internal arcs of $T$ and arrows \emph{clockwise} rotation of arcs around marked points.

\item Each internal triangle $\triangle$ yields an oriented cycle $a_\triangle b_\triangle c_\triangle$ in $Q_T$ that is unique up to cyclic permutation.  Then the potential $W_T$ is given by the sum of all these oriented cycles over all internal triangles of $T$.
\end{itemize}

\begin{example}\label{example:qpeasy}
Consider the triangulation of $\mathcal{M}_2$ in the middle of the top row of Figure \ref{fig:flip-M2s}. The double cover is given by the triangulation $T$ of the annulus with four arcs as shown in Figure \ref{fig:doublecoverqpeasy}. 

\begin{figure}[!htbp]
\centering
\begin{subfigure}[!htbp]{0.4\textwidth}
    \centering
    \begin{tikzpicture}[scale=0.8]
        \draw[thick]  (0,0) circle (2);
        \filldraw[thick,fill=white!80!black]  (0,0) circle (0.75);
        \draw[very thick, red] (0,-2) .. controls (2.5,-1) and (1.2,1.5) .. (0,0.75);
        \node[red] at (1.25,.75) {$4$};
        \draw[very thick,orange] (0,-2) .. controls (-2.5,-1) and (-1.2,1.5) .. (0,0.75);
        \node[orange] at (-1.25,.75) {$2$};
        \draw[very thick,cyan] (0,.75) .. controls (-1.5,0.75) and (-1.5,-1.25) .. (0,-1.25) .. controls (1.5,-1.25) and (1.5,.75) .. (0,.75);
        \node[cyan] at (0,-1.5) {$3$};
        \draw[very thick,blue] (0,-2) .. controls (-2.95,-1) and (-1.75,1.75) .. (0,1.75) .. controls (1.75,1.75) and (2.95,-1) .. (0,-2);
        \node[blue] at (0,1.5) {$1$};
        \fill[darkgreen] (0,-2) circle (3pt);
        \fill[darkgreen] (0,2) circle (3pt);
        \fill[darkgreen] (0,-0.75) circle (3pt);
        \fill[darkgreen] (0,0.75) circle (3pt);
    \end{tikzpicture}
    \caption{Triangulation $T$}
    \label{fig:doublecoverqpeasy}
    \end{subfigure}
    \begin{subfigure}[!htbp]{0.4\textwidth}
    \centering $Q_T = $
    \begin{tikzcd}
         & 1 \arrow[dl, swap, "\alpha_1"] \\
         2  \arrow[rr, "\alpha_2", shift left] \arrow[rr, "\beta_2", shift right, swap] & & 4 \arrow[ul, "\alpha_3", swap]  \arrow[dl, "\beta_1"] \\
         & 3 \arrow[ul, "\beta_3"]
       \end{tikzcd}
    
	$W_T = \alpha_1 \alpha_2 \alpha_3 + \beta_3\beta_2\beta_1$
    \caption{Quiver with potential $(Q_T,W_T)$}
    \label{fig:three sin x}
    \end{subfigure}
    \caption{A triangulation of annulus and its associated quiver with potential.}
    \label{fig:three graphs}
\end{figure}
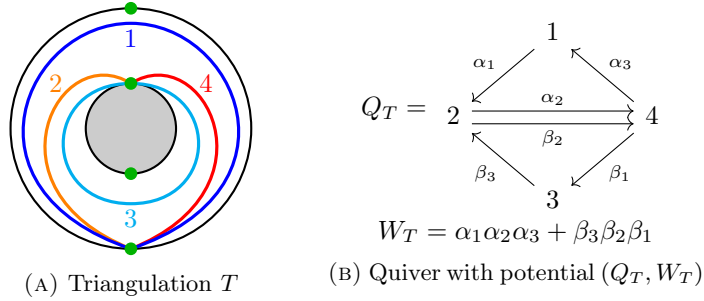
\end{example}

A QP gives rise to an algebra $J_{Q,W}$ called \dfn{Jacobian algebra}. In the case when $(Q,W)=(Q_T,W_T)$ for some triangulation $T$, then  \cite{ABCP10} showed that the arising Jacobian algebra $J_T:=J_{Q_T,W_T}$ belongs to a special class of finite-dimensional (basic) algebras called \dfn{gentle algebras} (see Definition \ref{def:gentle}); note that having $\M\subset \partial\S$ is crucial.  Its quiver-and-relations, i.e. the pair $(Q,R)$ of a quiver $Q$ and a set $R$ of linear combinations of paths in $Q$ such that $J_T\cong \Bbbk Q/(R)$, can be written in a more practical form:
\[Q=Q_T, \quad \text{ and }\quad R=\{ \text{length 2 paths in internal triangles of }T\}.
\]
\begin{example}\label{example:jacobianalg}
Using the QP from Example \ref{example:qpeasy}, we obtain the following ideal 
$$R= \langle \alpha_1\alpha_2, \alpha_2\alpha_3, \alpha_3\alpha_1, \beta_2\beta_1, \beta_3\beta_2, \beta_1\beta_3 \rangle.$$ This gives the Jacobian algebra $J_T = \Bbbk Q/R$ which has a basis given by 
$$\{\epsilon_1, \epsilon_2, \epsilon_3, \epsilon_4, \alpha_1, \alpha_2, \alpha_3, \beta_1, \beta_2, \beta_3, \alpha_1\beta_2, \alpha_2\beta_1,\beta_3\alpha_2,\beta_2\alpha_3\}.$$
\end{example}

\subsection{Cluster category associated to a QP}
By a result of Amiot \cite[Thm 3.6]{Ami09}, for any QP $(Q,W)$ with a finite-dimensional Jacobian algebra $J_{Q,W}$, then one can associate to it a Krull-Schmidt Hom-finite 2-Calabi-Yau $\Bbbk$-linear triangulated category, denoted by $\C_{(Q,W)}$, called the (Amiot's) \dfn{cluster category}.
For an introductory exposition on the cluster category, see \cite{Amilectureseries}.  When $(Q,W)=(Q_T,W_T)$ for some triangulation $T$, we will use a simpler notation $\C_T:=\C_{(Q,W)}$.

Let us now focus on the case when $(Q,W)=(Q_T,W_T)$ for some triangulation $T$ of marked orientable surface $\SM$.
Denote by $\ind\C$ the set of isomorphism classes of indecomposable objects in $\C$, whose relation with surface combinatorics is first explained in \cite{ABCP10} and \cite{BZ11}.  Then we have a bijection 
\begin{align}\label{eq:curve corresp}
 \big(\Ccc\SM\times \Bbbk^\times\big)\sqcup \Cnc\SM \leftrightarrow \ind\C.
\end{align}

\begin{example}
Let $\omega$ be the unique closed curve on the triangulation of the annulus given in Figure \ref{fig:doublecoverqpeasy}. For some $\lambda \in \Bbbk^{\times}$, this curve corresponds to the indecomposable module $V_\lambda = \xymatrix@1{\Bbbk \ar@<.5ex>[r]^{1}\ar@<-.5ex>[r]_{\lambda} & \Bbbk}$.
\end{example}

Henceforth, we use the same notation for curves with endpoints and the corresponding indecomposable objects in $\C_T$; likewise for the \dfn{coloured closed curves} $(\gamma,\lambda)\in \Ccc\SM\times\Bbbk^\times$.  

\begin{definition}
Let $\C$ be a triangulated category with shift functor $[1]$.  Denote by $\Ext_\C^1(X,Y):=\Hom_\C(X,Y[1])$ for any $X,Y\in\C$.

An object $T\in \C$ is called \dfn{cluster-tilting} if
\begin{itemize}
\item it is \dfn{rigid}, i.e. $\Ext_\C^1(T,T)=0$,
\item and $T^{\perp_1}:= \Kernel\Ext_\C^1(T,-)$ and ${}^{\perp_1} T:=\Kernel\Ext_\C^1(-,T)$ coincides with $\add(T)\subset \C$.
\end{itemize}

A cluster-tilting object is said to be \dfn{basic} if its indecomposable direct summands are pairwise distinct.  Denote by $\ctilt(\C)$ the set of (isomorphism classes of) basic cluster-tilting objects in $\C$.
\end{definition}

Suppose $T$ is a triangulation of an orientable marked surface $\SM$.  It turns out that 
\[
\Gamma:=\bigoplus_{\gamma\in T}\gamma[-1] \in \C_T
\] 
(and hence also $T$ itself) is a cluster-tilting object.  In fact, more generally, we have the following commutative diagram
\[
\xymatrix@C=60pt@R=12pt{ 
\bfA\SM \ar@{<->}[r]^{1:1\phantom{abcdefgabcdefg}} & \{\text{indecomposable rigid objects of }\C\} \\
\{\text{arcs of }T\} \ar@{<->}[r]^{1:1\phantom{abcdefgabcdefg}} \ar@{}[u]|{\cup} & \{\text{indecomposable direct summands of } \Gamma[1] \}.\ar@{}[u]|{\cup}
}
\]

By Koenig and Zhu's result \cite{KZ08}, the functor $\Hom_{\C_{(Q,W)}}(\Gamma,-):\C_{(Q,W)}\to \mod J_{Q,W}$ induces an equivalence 
\[
M(-):=\Hom_{\C_{(Q,W)}}(\Gamma,-) : \C_{(Q,W)}/[\Gamma[1]] \to \mod J_{Q,W}
\]
where $\C_{(Q,W)}/[\Gamma[1]]$ denotes the additive quotient of $\C$ by the ideal of morphisms factoring through $\Gamma[1]$.  
Practically, this means that indecomposable objects of $\C_{(Q,W)}$ are `given' by those of $\mod J_{Q,W}$ along with the indecomposable direct summands of $\Gamma[1]$.

Let us now describe the explicit structure of the indecomposable $J_T$-modules for a triangulation $T$ of a marked orientable unpunctured surface.
There are two types of such modules, called \emph{strings} and \emph{bands}.

\subsection{String modules vs curves with endpoints}\label{subsec:string intro}
Consider first the case of a curve with endpoints $\gamma:[0,1]\to \S$ \emph{that is not an arc of $T$}.  We have the following set of crossings between $\gamma$ and arcs of $T$
\[
\mathrm{cross}(\gamma,T) = \{\gamma(t_0), \gamma(t_1), \ldots, \gamma(t_c)\}, \text{ arranged so that }0< t_0< t_1 <\cdots t_{c}<1.
\]
Let $\gamma_i := \gamma|_{(t_{i-1},t_{i})}$ for $i=1,2,\ldots,c$.  Since each $\gamma_i$ is an angle of an (internal) triangle of $T$, this can be identified with an arrow $\alpha_i \in Q=Q_T$.  Note that $\alpha_i$ may not have the same orientation as the (oriented) segment $\gamma_i$, in which case, we write $\alpha_i\simeq\gamma_i^{-1}$; otherwise, $\alpha_i\simeq \gamma_i$.  The underlying vector space of the module $M(\gamma)$ is given by
\[
M(\gamma) \stackrel{\text{v.sp}}{=} \bigoplus_{i=0}^c \Bbbk \gamma(t_i) \text{ with $\alpha_i$ action: }
\begin{cases}
\Bbbk \gamma(t_{i-1}) \xrightarrow{1} \Bbbk\gamma(t_i), & \text{ if $\alpha_i\simeq\gamma_i$};\\
\Bbbk \gamma(t_{i}) \xrightarrow{1} \Bbbk \gamma(t_{i-1}), & \text{ if $\alpha_i\not\simeq\gamma_i$},
\end{cases}
\]
and with $e_\tau$-action, for primitive idempotent $e_\tau\in J_T$ corresponding to the arc $\tau$, on $\gamma(t_i)$ is given by identity if $\gamma(t_i)\in \tau$; by zero otherwise.

As we can see, a string module $M(\gamma)$ can be encoded purely by combinatorial means, namely, the sequence $\gamma_i$'s.  This is what people call \dfn{string combinatorics} in the representation theory of gentle (or generally, special biserial) algebras, which we will describe more properly in the following.

We consider the elements of $Q_0\sqcup Q_1\sqcup Q_1^{-1}$ as \dfn{letters}, where $Q_1^{-1}$ is the set of formal inverses $\alpha^{-1}$ of arrows $\alpha\in Q_1$.  We call a letter $\alpha$ a \dfn{directed arrow} if $\alpha\in Q_1$,  an \dfn{inverse arrow} if $\alpha\in Q_1^{-1}$, and \dfn{trivial} if $\alpha\in Q_0$.
It is customary to use the trivial path $e_x$ as letter for $x\in Q_0$.
For a directed arrow $\alpha\in Q_1$, its inverse $\alpha^{-1}$ has source $s(\alpha^{-1}):=t(\alpha)$ and target $t(\alpha^{-1})=s(\alpha)$.
The inverse $(\alpha^{-1})^{-1}$ of an inverse arrow $\alpha^{-1}$ is just $\alpha$, and the inverse of a trivial letter is itself.
A word $w$ is called a \dfn{walk} if it is either trivial, i.e. $w=e_i \in Q_0$, or $w=w_1\cdots w_\ell$ with letters $w_i$'s such that the following hold:
\begin{itemize}
\item $t(w_{i})=s(w_{i+1})$ for all $i$;
\item if both $w_i, w_{i+1}$ are directed (respectively inverse), then $w_iw_{i+1}\notin R$ (respectively $w_{i+1}^{-1}w_i^{-1}\notin R$);
\item if $w_i$ and $w_{i+1}$ are in different direction, then $w_i\neq w_{i+1}^{-1}$.
\end{itemize}
Inverting a letter extends to a reflection operation on the set of walks, and the induced equivalence classes are called \dfn{string}.  

Curves with endpoints can be identified with strings.  Indeed, for such a curve $\gamma$, the segment $\gamma_i=\gamma|_{(t_i,t_{i+1})}$ defines a letter $w_i$ in the corresponding walk; note that $\gamma_1,\ldots, \gamma_c$ suffices to determine $\gamma$ as the remaining starting interval of $\gamma$ are uniquely determine by going from the arc containing $\gamma(t_0)$ to the opposite marked point of the triangle, and likewise for the ending interval.
Note that the trivial strings $e_x$ for $x\in Q_0$ correspond to the curve that crosses $T$ only once at arc $x\in T$.
Now, the module $M(\gamma)$ for a string $\gamma$ (equivalently, curve with endpoints) is called a \dfn{string module}.

\begin{example}
Consider the triangulation $T$ of the surface $\SM$ in Example \ref{example:qpeasy}.
In Figure \ref{fig:non-closed curve}, we show a (non-closed) curve $\gamma$ on $\SM$.
\begin{figure}[!htbp]
\centering
\begin{tikzpicture}[scale=0.75]
        \draw[thick]  (0,0) circle (2);
        \filldraw[thick,fill=white!80!black]  (0,0) circle (0.75);
        \draw[thick, white!70!black] (0,-2) .. controls (2.5,-1) and (1.2,1.5) .. (0,0.75);
        \node[white!70!black] at (1.25,.85) {$4$};
        \draw[thick,white!70!black] (0,-2) .. controls (-2.5,-1) and (-1.2,1.5) .. (0,0.75);
        \node[white!70!black] at (-1.25,.85) {$2$};
        \draw[thick,white!70!black] (0,.75) .. controls (-1.5,0.75) and (-1.5,-1.25) .. (0,-1.25) .. controls (1.5,-1.25) and (1.5,.75) .. (0,.75);
        \node[white!70!black] at (0,-1.05) {$3$};
        \draw[thick,white!70!black] (0,-2) .. controls (-2.95,-1) and (-1.75,1.75) .. (0,1.75) .. controls (1.75,1.75) and (2.95,-1) .. (0,-2);
        \node[white!70!black] at (0,1.5) {$1$};
\draw[thick,blue]  plot[smooth, tension=.7] coordinates {(0,-0.75) (0.9,-0.5) (0.5,1.1) (-1.3,0.6) (-0.5,-1.5) (1.6,-0.3) (0,2)};
\node[thick, blue] at (-0.4,1.4) {$\gamma$};
        \fill[darkgreen] (0,-2) circle (3pt);
        \fill[darkgreen] (0,2) circle (3pt);
        \fill[darkgreen] (0,-0.75) circle (3pt);
        \fill[darkgreen] (0,0.75) circle (3pt);
\node at (4,0.2) {$M(\gamma)=$};
\node (v1) at (7.5,1.5) {$\Bbbk$};
\node (v2) at (5.5,0) {$\Bbbk$};
\node (v4) at (9.5,0) {$\Bbbk^{2}$};
\node (v3) at (7.5,-1.5) {$\Bbbk$};
\draw[->]  (v1) -- node[above]{\footnotesize$0$} (v2);
\draw[->] (v3) -- node[below]{\footnotesize$0$}(v2);
\draw[->]  (v4) -- node[right,pos=.6]{\footnotesize\phantom{a}$[0,1]$}(v3);
\draw[->] (v4) -- node[right,pos=.6]{\footnotesize\phantom{a}$[1,0]$}(v1);
\draw[->] (v2.25) --node[above]{\footnotesize $\begin{bsmallmatrix}
1\\0\end{bsmallmatrix}$} (v4.166);
\draw[->] (v2.-25) --node[below]{\footnotesize $\begin{bsmallmatrix}
0\\1\end{bsmallmatrix}$} (v4.-166);
\node at (10.5,0.2) {$=$};
\node at (10.5,0.2) {$=$};
\node[scale=0.7] at (11.5,0.2) {{$\arraycolsep=2.5pt\begin{array}{ccccc}
 & & 2 & & \\
 & 4 & & 4 &\\ 
1 & & & & 3\\
\end{array}$}};
\end{tikzpicture}\caption{A non-closed curve and its string module}\label{fig:non-closed curve}
\end{figure}
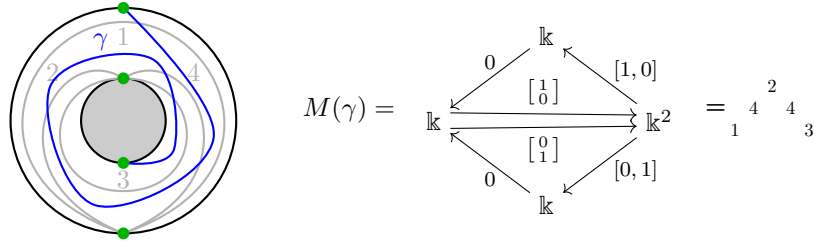
Orient $\gamma$ so it starts with the top marked point.
Then it crosses $T$ in the order of $1, 4 , 2 , 4, 3$ before reaching its other endpoint.
The string correspond to this is $\alpha_3^{-1}\beta_2^{-1}\alpha_2\beta_1$ which gives the  string module $M(\gamma)$ shown on the right of Figure \ref{fig:non-closed curve}.
\end{example}

\subsection{Band modules vs closed curves}\label{subsec:band intro}
Suppose now $\gamma:S^1\to \S$ is a (non-contractible) closed curve.
Write $\gamma = \omega^n$ for some \dfn{primitive} closed curve $\omega$ on $\S$, i.e. $\gamma$ is homotopic the concatenation of $n$ copies of a `shorter' closed curve $\omega$, and $\omega$ itself cannot be written as concatenation of a shorter closed curve.
Similar to the previous case, write
\begin{align}\label{eq:omega cross T}
&\mathrm{cross}(\omega,T) = \{\omega(z_0), \omega(z_1), \ldots, \omega(z_{c-1})\}\\
&\text{ with }z_i=\exp(2t_{i+1}\pi \sqrt{-1})\text{ and }0=t_1 <\cdots t_{c}<1.\notag 
\end{align}
Removing the intersections yields intervals $\omega_1,\ldots, \omega_c$ of $\omega$ which can be identified with arrows $\alpha_1,\ldots, \alpha_c \in Q$
By rotating the pieces if necessary, we assume that $\alpha_1\simeq \omega_1$ and $\alpha_c\not\simeq \omega_c$.
For $\lambda \in \Bbbk^\times$, denote by $J_n(\lambda)$ the Jordan block of size $n$ with eigenvalue $\lambda$.  Then we can define the indecomposable $J_T$-module $M_\lambda(\gamma):=M((\gamma,\lambda))$ associated to the indecomposable object $(\gamma=\omega^n,\lambda)\in \Ccc\SM\times \Bbbk^\times \subset \ind\C$ by 
\[
M_\lambda(\omega^n)  \stackrel{\text{v.sp}}{=} \bigoplus_{i=0}^{c-1} V_i, \text{ with }V_i=\bigoplus_{j=1}^{n} \Bbbk \omega(z_i)^{(j)},\]
\[\text{defining the $\alpha_i$ action by }
\begin{cases}
V_{i-1} \xrightarrow{1} V_i, & \text{ if $\alpha_i\simeq\gamma_i$};\\
V_{i} \xrightarrow{1} V_{i-1}, & \text{ if $\alpha_i\not\simeq\gamma_i$ and $i\neq c$};\\
V_{0} \xrightarrow{J_n(\lambda)} V_{c-1}, & \text{ if $i=c$}.\\
\end{cases}
\]
and having primitive idempotent $e_\tau\in J_T$ corresponding to $\tau\in T$ acts by identity on $V_i$ if $\omega(z_i)\in \tau$; by zero otherwise.

These indecomposable modules are called \dfn{band modules}.  Like string modules, they can be encoded completely by string combinatorics.  Consider a walk $w=w_1\cdots w_c$ with $s(w_1)=t(w_c)$, we can \dfn{rotate} it to form a new word $w_2\cdots w_cw_1$.  If this new word is also a walk, then the equivalence class of $w$ under compositions of reflections and rotations is called a \dfn{band}.  Similar to strings, bands correspond to closed curves on $\SM$.  
Unless otherwise specified, we will assume the representative $w=w_1\cdots w_c$ we take from the equivalence class has $w_1$ directed and $w_c$ inverse, which matches our convention of indexing the segments of closed curves in the previous paragraph.
Concatenation of strings is the natural operation inherited from concatenation of words; in particular, notations $w^n$ mean self-concatenating $n$ times.  A band $w$ is \dfn{primitive} if $w\neq u^n$ for any subword $u$ of $w$ that is also a band; hence, primitive bands correspond to a primitive closed curve.

\begin{example}\label{eg:band eg}
Consider the fan triangulation of the Möbius strip with two marked points (left of the top row of Figure \ref{fig:flip-M2s}).  It has an orientable double cover (see Definition \ref{def:double cover}) by an annulus as shown in Figure \ref{fig:doublecover} (this is a special case of Figure \ref{fig:fan}), where the preimage $\wti{T}:=\{1,2,1',2'\}$ of the triangulation $\{\overline{1},\overline{2}\}$ and the associated QP $(Q_{\wti{T}},\emptyset)$ are as shown.

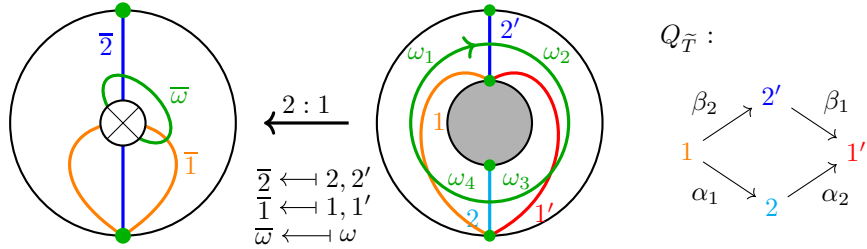
\begin{figure}[!htbp]
\centering
   \begin{tikzpicture}[scale=0.75]
    \draw[thick]  (0,0) circle (2);
    \filldraw[thick, fill=white!70!black]  (0,0) circle (0.75);
    \draw[very thick, cyan] (0,-2) -- node [near start, left] {$2$}  (0,-0.75);
    \draw[very thick, blue] (0,2) -- node [near start, right] {$2'$} (0,0.75);
    \draw[very thick, red] (0,-2) .. controls (2,-1) and (1.2,1.5) .. node [pos=0.12,right] {$1'$} (0,0.75);
    \draw[very thick,orange] (0,-2) .. controls (-2,-1) and (-1.2,1.5) .. node [midway, right] {$1$} (0,0.75);
    \fill[darkgreen] (0,-2) circle (3pt);
    \fill[darkgreen] (0,2) circle (3pt);
    \fill[darkgreen] (0,-0.75) circle (3pt);
    \fill[darkgreen] (0,0.75) circle (3pt);
    \draw[very thick, green!65!black, decoration={markings, mark=at position 0.3 with {\arrow{<}}}, postaction={decorate}] (0,0) circle (1.4)
    node at (115:-1.15){$\omega_3$}
    node at (-135:-1.7){$\omega_2$}
    node at (-45:-1.7){$\omega_1$}
    node at (65:-1.15){$\omega_4$};
    
\begin{scope}[shift={(-6.5,0)}]
    \draw[very thick,blue] (0,-2) -- node [pos=.85,left] {$\overline{2}$} +(0,4);
    \draw[very thick,orange] (0,-2) .. controls (-1.5,-1) and (-1,0) .. (0,0) .. controls (1,0) and (1.5,-1) .. (0,-2);
    \node[orange] at (1.2,-0.7) {$\overline{1}$};
    \draw[rotate=-45,very thick,  green!65!black] (0,0.35) ellipse (.75 and 0.4);
    \node[thick,  green!65!black] at (1,0.5) {$\overline{\omega}$};
    \Mtwo{0,0}
\end{scope}

    \draw[very thick, ->] (-2.5,0) -- node[midway,above] {$2:1$} +(-1.5,0);
    \node (v1) at (-2.5,-1) {$2,2'$}; \node (v2) at (-4,-1) {$\overline{2}$};
    \node (v3) at (-2.5,-1.5) {$1,1'$}; \node (v4) at (-4,-1.5) {$\overline{1}$};
    \node (v5) at (-2.5,-2) {$\omega$}; \node (v6) at (-4,-2) {$\overline{\omega}$};
    \draw[|->]  (v1) -- (v2);    \draw[|->]  (v3) -- (v4);    \draw[|->]  (v5) -- (v6);

\draw node[orange] (q1) at (3.5,-0.5) {$1$} node[cyan] (q2) at (5,-1.5) {$2$} 
	node[blue] (q2') at (5,0.5) {$2'$} node[red] (q1') at (6.5,-0.5) {$1'$};
\draw[->] (q1) -- node[midway,below left] {$\alpha_1$} (q2);
\draw[->] (q1) -- node[midway,above left] {$\beta_2$} (q2');
\draw[->] (q2) -- node[midway,below right] {$\alpha_2$} (q1');
\draw[->] (q2') -- node[midway,above right] {$\beta_1$} (q1');
\node at (3.5,1.5) {$Q_{\widetilde{T}}$ : };
\end{tikzpicture}
\caption{A primitive closed curve $\omega=\omega_1\omega_2\omega_3\omega_4$ in the double cover of the fan triangulation of the Möbius strip with two marked points and the quiver with (zero) potential associated to the orientable double cover.}
\label{fig:doublecover}
\end{figure}

There is a unique primitive closed curve $\omega = \omega_1\omega_2\omega_3\omega_4 = b_2b_1a_2^{-1}a_1^{-1}$ on the double cover that represent the preimage of the unique quasi-arc $\overline{\omega}$.  We have a band module over $J_{\wti{T}}$ given by  $M_{\lambda}(\omega^n)$ for each $n\geq 1$ and $\lambda\in \Bbbk^\times$.  Explicitly, the case when $n\in\{1,2\}$ can be written as follows.

\[ M_\lambda(\omega) =  \begin{tikzcd}[row sep= small, column sep= small]
     & \Bbbk \arrow[dr, "1"] \\
     \Bbbk  \arrow[ur, "1"] \arrow[dr, "{\lambda}", swap] & & \Bbbk  \\
     & \Bbbk \arrow[ur, "1", swap]
\end{tikzcd}, \quad M_\lambda(\omega^2) = \begin{tikzcd}[row sep=small, column sep=small]
     & \Bbbk^2 \arrow[dr, "1"] \\
     \Bbbk^2  \arrow[ur, "1"] \arrow[dr, "{\sbm{\lambda\amp 0\\1\amp\lambda}}", swap] & & \Bbbk^2  \\
     & \Bbbk^2 \arrow[ur, "1", swap]
\end{tikzcd}\,. \]
\end{example}

We recall the following description of Hom-spaces between indecomposable objects.

\begin{proposition}\label{prop::Homs}{\rm \cite[Lem 3.3]{Palu08}, \cite[Cor 5.4]{BZ11}, \cite[Prop 4.3]{AIR14}}
Let $\C$ be a Hom-finite Krull-Schmidt 2-CY triangulated category with cluster-tilting object $T$ with $\Lambda:=\End_\C(T)$, and $X,Y$ be objects of $\C$.
Write $X=X'\oplus U[1]$ and $Y=Y'\oplus V[1]$ so that $U,V\in \add (T)$ and $X',Y'$ has no direct summand in $\add (T[1])$.
Then the following hold.
\begin{enumerate}[(a)]
\item $M(X'[1])\cong \tau M(X')$ (and likewise for $Y'$).
\item There is an exact sequence \[
0\to D\Hom_{\Lambda}(M(Y'), \tau M(X')) \to  \Ext_{\C}^1(X',Y') \to \Hom_{\Lambda}(M(X'),\tau M(Y')) \to 0.
\]

\item There is a bifunctorial isomorphism 
\[\Ext_\C^1(X,Y) \simeq  \Ext_\C^1(X',Y') \oplus \Hom_\Lambda(M(U),M(Y'))\oplus\Hom_\Lambda(M(V), M(X')).\]
\end{enumerate}
\end{proposition}

\begin{proposition}\label{prop::Ext vanishing}
For the cluster category $\C=\C_T$ associated to surface triangulation, the following hold.
\begin{enumerate}[(a)]
\item For any $\gamma\in \Cnc\SM$, $\gamma$ is rigid if and only if $\gamma\in \bfA\SM$.

\item For any curves $\gamma,\delta$ that are self-non-crossing, any $X\in\C$ with underlying curve $\gamma$ and any $Y\in\C$ with underlying curve $\delta$, $\Ext_\C^1(X,Y)=0$ if and only if $\mathrm{cross}(\gamma,\delta)=\emptyset$.
\end{enumerate}
\end{proposition}
\begin{proof}
(b) is \cite[Prop 5.3]{BZ11} and (a) follows from (b).
\end{proof}

\section{Involution and duality}\label{sec:involutionandduality}

In this section, we associate to the orientable double cover of a surface to a QP with \emph{involution}.  We then further enhance this  to a contravariant duality functor $\nabla$ on the cluster category.  This provides the setup in the following sections, where we utilise the \emph{symmetric representation theory} first studied in \cite{DW02,BCI21} to write down a dictionary between the surface combinatorics and phenomena in the cluster categories.

\subsection{Symmetric QP and relation to non-orientable surfaces}\label{subsec:involution intro}
From now on, we assume $\Bbbk$ is an algebraically closed field with characteristic different from 2.

\begin{definition}
A \dfn{symmetric quiver} is a pair $(Q,\sigma)$ of a quiver $Q$ equipped with an involutive anti-automorphism (or simply \dfn{involution}) $\sigma$, i.e. $\sigma(i\xrightarrow{\alpha}j) = \sigma(i)\xleftarrow{\sigma(\alpha)}\sigma(j)$ for all arrows $\alpha\in Q_1$ with $\sigma^2=\mathrm{id}$.  This is equivalent to saying that $\sigma$ defines an algebra isomorphism, which we denote by $\sigma$ again by abusing notation, $\sigma:\Bbbk Q^\op \to \Bbbk Q$.
A \dfn{symmetric QP} is a tuple $(Q,W,\sigma)$ such that $(Q,\sigma$ is a symmetric quiver, $(Q,W)$ is a QP, and $\sigma(W)=W$.
\end{definition}

\begin{example}\label{eg:symmetric QP} 
\begin{enumerate}[(1)]
\item The type $\mathbb{A}_n$ quiver $1\xrightarrow{a_1}2 \xrightarrow{a_2} \cdots \xrightarrow{a_{n-1}} n$ with $\sigma(i)=n+1-i$ and $\sigma(a_i)=a_{n+1-i}$.
    
\item The Kronecker quiver $\xymatrix@1{1\ar@<.5ex>[r]^{a}\ar@<-.5ex>[r]_{a'} & 1'}$ has two choices of involutions - one of them fixes the arrows while the other swap $a$ and $a'$.
    
\item Consider the following QP $(Q,W)$ in Example \ref{example:qpeasy}. Take $\sigma(1) = 3$ and $\sigma(2) = 4$, then this defines an involution $\sigma$ on $(Q,W)$.  This induces an algebra isomorphism on the Jacobian algebras as $\sigma(W) = \sigma(\alpha_3)\sigma(\alpha_2)\sigma(\alpha_1) + \sigma(\beta_12)\sigma(\beta_2)\sigma(\beta_3) = \beta_3\beta_2\beta_1 + \alpha_1\alpha_2\alpha_3 = W$.

\item Suppose $(Q,W)$ is a QP.  Take $\wti{Q}:=Q\sqcup Q^\op$ and $\wti{W}:=W+W^\op$, where $W^\op$ is the linear combination of cycles on $Q^\op$ given by reversing the cycles in $W$.  Let $\sigma$ be the automorphism on $Q$ that swaps $i\in Q_0$ with the corresponding $i\in Q^\op_0$, and $\alpha\in Q_1$ with $\alpha^\op\in Q_1^\op$.  Then $(\wti{Q},\wti{W},\sigma)$ is a symmetric QP.
\end{enumerate}
\end{example}

\begin{definition}
A \dfn{fixed-point-free symmetric QP}, or \dfn{FF-symmetric QP} for short, is a symmetric QP $(Q,W,\sigma)$ such that both $\sigma|_{Q_0}$ and $\sigma|_{Q_1}$ are fixed-point-free.
\end{definition}

\begin{example}\label{eg:Kronecker0}
The smallest example of a FF-symmetric QP is the Kronecker quiver (Example \ref{eg:symmetric QP} (2)) with $\sigma$ swapping both vertices and swapping both arrows.  
More generally, for the quiver
\[
\mathbb{A}_{n,n}:= \quad \vcenter{\xymatrix@R=8pt@C=40pt{
 & n' \ar[r]^{b_{n-1}} & \cdots \ar[r]^{b_2} & 2' \ar[rd]^{b_1} & \\
1 \ar[ru]^{b_n} \ar[rd]_{a_1} & & & & 1', \\
 & 2 \ar[r]_{a_2} & \cdots \ar[r]_{a_{n-1}} & n  \ar[ru]_{a_n} & \\
}}
\]
we have $(\mathbb{A}_{n,n}, 0, \sigma)$, where $\sigma$ swaps the unprimed and the primed, defines a FF-symmetric QP.  This quiver corresponds to the fan triangulations of the M\"{o}bius strip shown in Figure \ref{fig:fan}.  

Both Example \ref{eg:symmetric QP} (3) and (4) are examples of FF-symmetric QP.
\end{example}

Let us now review the a characterisation of QP's arising from orientable marked unpunctured surfaces.

\begin{definition}\label{def:gentle}
Let $(Q,W)$ be a QP, and $R$ be the set of monomials that appear in the cyclic derivative $\partial_\alpha W$ with respect to some $\alpha\in Q_1$.
We say that $(Q,W)$ is \dfn{gentle} if \begin{enumerate}
\item $(Q,W)$ is \dfn{Jacobi-finite}, i.e. $J_{Q,W}$ is finite-dimensional;

\item for each $v\in Q_0$, there are at most 2 in-coming arrows and at most 2 out-going arrows;

\item For all $\alpha\in Q_1$, there is at most one $\beta\in Q_1$ such that $\beta\alpha \in R$ and at most one $\beta'\in \Q_1$ such that $\beta'\alpha\notin R$;

\item For all $\alpha\in Q_1$, there is at most one $\gamma\in Q_1$ such that $\alpha\gamma \in R$ and at most one $\gamma'\in \Q_1$ such that $\alpha\gamma'\notin R$.
\end{enumerate}
This is, by definition, equivalent to saying that the Jacobian algebra $J_{Q,W}$ is a (finite-dimensional) gentle algebra.
\end{definition}

\begin{proposition}\label{prop:ABCP corresp} {\rm \cite{ABCP10}}
For any marked unpunctured orientable surface $\SM$ equipped with a triangulation $T$, $(Q_T,W_T)$ is a gentle QP.
\end{proposition}

Keeping in the mind the relation between non-orientable surface and its double covering, we can extend Proposition \ref{prop:ABCP corresp} as follows.

\begin{proposition}\label{prop:FFS-QP corresp}
For a marked unpunctured (not necessarily orientable) surface $\SM$ equipped with a triangulation $T$.  Let $(\wti{\SM},\sigma)$ be an orientable double cover of $\SM$, and $\wti{T}$ be the associated double cover of $T$.
Then we have a fixed-point-free symmetric gentle QP $(Q_{\wti{T}},W_{\wti{T}}, \sigma)$ where $\sigma$ is induced by the restriction of $\sigma_{\wti{\S}}$ to $\wti{T}$.
\end{proposition}

\begin{remark}
One can consider relaxing surfaces equipped with `marked unpunctured points with triangulations' to surfaces equipped with `a pair of dual cellular dissections' in the sense of \cite{PPP19}.  Then `symmetric gentle QP' can be replaced by locally gentle algebra with involution (equivalently, locally gentle quiver with symmetric structure in the obvious sense).   The proof of this generalisation is analogous to the one presented below and we omit them for simplicity as these settings is beyond the scope of this script.
\end{remark}
\begin{proof}
Since $\sigma:\wti{\S}\to\wti{\S}$ only fixes some (or none if $\S$ is already orientable) of the closed curves on $\wti{\SM}$, it acts transitively on the set $\bfA\wti{\SM}$ of arcs, which means that the induced $\sigma$ on $Q_{\wti{T}}$ is a fixed-point free involution.  The claim then follows from Proposition \ref{prop:ABCP corresp}.
\end{proof}

\begin{example}
(1)  Let $\SM$ be the M\"{o}bius strip with 1 marked point.  This is a unique triangulation whose double cover $\wti{T}$ defines the Kronecker quiver with fixed-point free involution as described in Example \ref{eg:symmetric QP} (2).

(2) The $\mathbb{A}_{n,n}$ symmetric QP in Example \ref{eg:Kronecker0} is associated to the orientable double cover of `fan triangulation' on the M\"{o}bius strip as shown in Figure \ref{fig:fan}.
\end{example}

\subsection{Duality functor associated to involution}
\label{subsec:duality functor}
Suppose $(Q,W,\sigma)$ is a symmetric QP.  Let $\Lambda:=J_{Q,W}$ be the associated Jacobian algebra.
We denote by $-\sigma$ the algebra map $\Lambda^\op \to \Lambda$ given by $\alpha_1\alpha_2\cdots \alpha_\ell \mapsto (-1)^\ell \sigma(\alpha_\ell)\sigma(\alpha_{\ell-1})\cdots \sigma(\alpha_1)$ for all paths $\alpha_1\cdots \alpha_\ell$ of length $\ell$.
Consider the $\Lambda$-$\Lambda^\op$-bimodule ${}_1\Lambda_{-\sigma}$ whose underlying space is $\Lambda$ with the natural left $\Lambda$-action and with the right $\Lambda^\op$-action given by the $m\cdot a:= -\sigma(a)m$ for all $a\in\Lambda^\op$, $m\in {}_1\Lambda_{-\sigma}$.
Note that subtle choice of putting the minus twist on $\sigma$ comes from the use of `symmetric representations' (see next section) of these algebras.
Anyway, the bimodule is invertible of order 2 and defines an equivalence $-\otimes_{\Lambda}{}_1\Lambda_{-\sigma}:\mod \Lambda \to \mod \Lambda^\op$.  Composing with the $\Bbbk$-linear dual $(-)^*:=\Hom_{\Bbbk}(-,\Bbbk)$ yields a contravariant equivalence
\[ \nabla: \mod\Lambda \xrightarrow{\sim} \mod\Lambda \text{ such that } \nabla^2\cong\mathrm{Id}. \]
We call this the \dfn{duality} (associated to $\sigma$).
Note that if we denote by $M_\alpha$ the transformation on $M\in \mod\Lambda$ representing $\alpha\in Q_1$, then $(\nabla M)_{\sigma(\alpha)}=-M_{\alpha}^*$.

\begin{example}
    Consider the module $M = M_{\lambda}(\omega)$ from Example \ref{eg:band eg}.
    Then,
    \[ \nabla M =  \begin{tikzcd} [column sep=small, row sep = small]
     & \Bbbk \arrow[dr, "{-\lambda}"] \\
     \Bbbk  \arrow[ur, "-1"] \arrow[dr, "-1", swap] & & \Bbbk  \\
     & \Bbbk \arrow[ur, "-1", swap]  \end{tikzcd}. \]
\end{example}

We have the following observation.

\begin{proposition}\label{prop:nabla facts}
The following hold.
\begin{enumerate}[(a)]
\item We have an equivalence $\nabla:\proj\Lambda\xrightarrow{\sim} \inj\Lambda$ sends $e\Lambda$ to $D(\Lambda \sigma(e))$.  In particular, there is a natural isomorphism $\nabla\circ \nu^{\pm}\cong \nu^{\mp}\circ \nabla$, where $\nu := -\otimes_\Lambda D\Lambda$ is the \dfn{Nakayama functor}.

\item For any $M\in\mod \Lambda$, denote by $P_M^\bullet$ the minimal projective presentation of $M$, and $I_M^\bullet$ the minimal injective copresentation of $M$.
Then we have $P_{\nabla M}^\bullet = \nabla(I_{M}^\bullet)$ and $I_{\nabla M}^\bullet = \nabla(P_M^\bullet)$.

\item {\rm \cite[Prop 3.4]{DW02}} There is a natural isomorphism $\nabla\circ \tau^{\pm} \cong \tau^{\mp}\circ\nabla$, where $\tau$ denotes the \dfn{Auslander-Reiten translation} (AR-translation for short).
\end{enumerate}
\end{proposition}
\begin{proof}
(a) is just stating the algebra isomorphism $\sigma:\Lambda^\op \to \Lambda$ categorically.  (b) follows from (a).  (c) follows by combining (a) and (b), as $\tau M$ is the kernel of $\nu(P_M^\bullet)$ and $\tau^- M$ is the cokernel of $\nu(I_M^\bullet)$.
\end{proof}

We can lift this duality to the cluster category.  This is the first justification of $\nabla$ being a categorification of the defining involution of a non-orientable surface from its double cover.
Recall from Section \ref{sec:cluster cat basic} that there is a a canonical projection $M(-):\C_{(Q,W)}\to \mod J_{(Q,W)}$.

\begin{proposition} \label{prop:lift nabla}
Suppose $(Q,W,\sigma)$ is a Jacobi-finite symmetric QP.
Then $\sigma$ induces a contravariant exact duality $\nabla$ on $\C=\C_{Q,W}$ so that we have a commutative diagram
\[
\xymatrix{
\C \ar[r]^{\nabla} \ar[d]_{M(-)}& \C \ar[d]^{M(-)}\\
\mod J \ar[r]^{\nabla} & \mod J.
}
\]
\end{proposition}

\begin{proof}
Recall that $\C$ is defined as the quotient $\mathsf{per}(\Gamma)/\mathsf{D}_{\mathrm{fd}}(\Gamma)$, where $\Gamma$ is the Ginzburg dg-algebra associated to a certain quiver with potential $(Q,W)$, $\mathsf{per}(\Gamma)$ the perfect derived category, and $\mathsf{D}_{\mathrm{fd}}(\Gamma)$ the full subcategory of the derived category supported at totally finite-dimensional (dg) $\Gamma$-modules.
Here, $W$ is the sum of the form $\alpha\beta\gamma$, one for each internal triangle of $\wti{T}$ where $\alpha\beta\gamma$ is (a choice of) a length 3 path given by bouncing inside the triangle.
Recall also that the underlying graded algebra structure of $\Gamma$ is given by  $\Bbbk \overline{Q}$ where $\overline{Q}$ is the quiver with the same set of vertices as $Q$, and the set of arrows is $Q_1\sqcup Q_1^*\sqcup Q_0$ with $Q_0$ representing the set of loops $t_i$ of degree $-2$ for $i\in Q_0$ and $Q_1^*:=Q_1^\op$ is the set of `dual arrows' $\alpha^*$ of degree $-1$ (for each $\alpha\in Q_1$).  The differential of $\Gamma$ is given by $d(\alpha)=0$ and $d(\alpha^*)=\partial_\alpha W$ for all $\alpha\in Q_1$, and $d(t_i)=e_i(\sum_\alpha [\alpha,\alpha^*])e_i$ for all $i\in Q_0$.

By abuse of notation, define $\sigma:\Gamma\to \Gamma^\op$ by  extending that of the original one $\sigma:\Bbbk Q\to \Bbbk Q^\op$.  So we have $\sigma(\alpha^*)=\sigma(\alpha)^*$ for all $\alpha^*\in Q_1^*$ and $\sigma(t_i)=-t_{\sigma(i)}$.  This clearly defines an isomorphism of graded algebra.  We claim that $\sigma$ is also a chain map (and so $\sigma$ is a dga isomorphism).  Indeed, it is clear that $\sigma d(Q_1)=0=d(\sigma(Q_1))$.  For $\alpha^*\in Q_1^*$, we have
\[d(\sigma(a^*))=d(\sigma(a)^*)=\partial_{\sigma(\alpha)}(W)= \sigma(\partial_\alpha (W))=\sigma d(a^*)\]
by the assumption of $(Q,W,\sigma)$ being symmetric.  For $t_i$'s, we have
\begin{align*}
d(\sigma(t_i)) &=-d(t_{\sigma(i)}) \\
&= -e_{\sigma(i)}\left(\sum_{a} [a,a^*]\right)e_{\sigma(i)}\\
&= -\sigma(e_i)\sigma\left(-\sum_{a}[a^*,a]\right) \sigma(e_i) = \sigma d(t_i).
\end{align*}
Define $\nabla := D(-\otimes_\Gamma {}_1\Gamma_{-\sigma})$, where ${}_1\Gamma_{-\sigma}$ is the $\Gamma$-$\Gamma^\op$-dg-bimodule given by $\Gamma$ equipped with the natural left $\Gamma$-action and the right $\Gamma^\op$-action is given by the dga isomorphism $\Gamma^\op\to\Gamma$ sending $a$ to $-\sigma(a)$.  Both $(-\otimes_\Gamma {}_1\Gamma_{-\sigma})$ and $D(-)$ are exact equivalence that restricts to equivalence on the respective bounded derived categories.  Hence, this induces the require equivalence on $\C = \mathsf{per}(\Gamma)/\mathsf{D}_{\mathrm{fd}}(\Gamma)$.

Note also that restricting $\nabla$ to $\add(\Gamma[1])\subset\mathsf{per}(\Gamma)$ yields an equivalence $\nabla:\add(\Gamma[1])\xrightarrow{\sim}\add(\Gamma[-1])$.  Since $\C$ is 2-CY, $\nabla$ is an equivalence on the full subcategory $\add(\Gamma[1])\subset\C$, where, by abusing notation, $\Gamma$ here is the image of $\Gamma\in\mathsf{per}(\Gamma)$ in $\C$.  
Thus, we have an induced duality $\nabla:\C/[\Gamma[1]]\xrightarrow{\sim}\C/[\Gamma[1]]$.  
Note that $\Gamma$ is the cluster-tilting object of $\C$ with endomorphism ring $J$, so it remains to see that the induced duality coincide with the one naturally defined on $\mod J$ via $\sigma_J:J\to J$.  Indeed, as $\sigma_J = \sigma_\Gamma \circ\pi$ for $\pi$ the natural projection $\Gamma\to H^0(\Gamma)=J$, $\nabla$ agrees as functor on $\add(J)\xrightarrow{\sim}\add(DJ)$, which extends to the module category $\mod J$.
\end{proof}
\begin{remark}
Note that the exactness here means that $\nabla\circ[\pm 1] \simeq [\mp 1]\circ \nabla$ and a triangle $(X\xrightarrow{f}Y\xrightarrow{g}Z\xrightarrow{h}X[1])\in \C$ is sent to $(\nabla(X[1])\xrightarrow{\nabla(h)}\nabla(Z)\xrightarrow{\nabla(g)}\nabla(Y)\xrightarrow{\nabla(f)}\nabla(X))$.
\end{remark}

\section{Symmetric representations vs curves}\label{sec:symm repn}

We view the duality $\nabla$ as the categorification of orientation-reversing automorphism $\sigma=\sigma_\S$ on a triangulated surface $\S$.
To categorify curves - in particular, quasi-arcs - of a non-orientable surface arising as $\S/\sigma$, we need to makes sense of indecomposability for the $\nabla$-orbit of an indecomposable object.  Classical orbit category construction does not work well in this setting as $\nabla$ is a contravariant duality; fortunately, a resolution called \emph{symmetric representation} has recently been suggested in the literature \cite{DW02,BCI21}. As in the previous section, we assume the underlying field $\Bbbk$ is of characteristic not equal to 2.

\subsection{Symmetric representations}
\begin{definition}\label{def:e-rep}
Suppose now that the characteristic of the underlying field $\Bbbk$ is not 2 and $\Lambda \cong \Bbbk Q/I$ is an algebra equipped with an involution $\sigma$ such that $\sigma(I)=I$ (such as the Jacobian algebra of a symmetric QP).
Let $\epsilon\in\{+1,-1\}$.  By an \dfn{$\epsilon$-form} on a $\Bbbk$-vector space $V$ we mean a bilinear form $\langle-,-\rangle$ that is symmetric when $\epsilon=+1$, and skew-symmetric when $\epsilon=-1$. 
An \dfn{$\epsilon$-representation} $(M,\langle-,-\rangle)$ of $(\Lambda,\sigma)$ is a $\Lambda$-module equipped with a bilinear form on $M$ such that
\begin{enumerate}[(1)]
\item $\langle-,-\rangle$ is a non-degenerate $\epsilon$-form;
\item $\langle-,-\rangle|_{M_i\times M_j}\neq 0$ implies $j=\sigma(i)$, where $M_i:= Me_i$ for the primitive idempotent $e_i$ corresponding to $i\in Q_0$;
\item $\langle v\alpha ,w\rangle+\langle v,w\sigma(\alpha)\rangle=0$ for all arrow $(\alpha:i\to j)\in Q_1$ and all $v\in M_i, w\in M_{\sigma(j)}$.
\end{enumerate}
We also consider $\epsilon$-representations to be \dfn{symmetric representation} if we do not want to emphasise the parity of $\epsilon$.
\end{definition}

\begin{example} \label{ex::e-rep}
Consider the quiver from Example \ref{example:qpeasy} and the module $M$ given by
 \[ M = \begin{tikzcd}[column sep=small, row sep = small]
     & \Bbbk^2 \arrow[dl, swap, "\begin{bsmallmatrix}0\\1\end{bsmallmatrix}"] \\
     \Bbbk  \arrow[rr, "0", shift left] \arrow[rr, "0", shift right, swap] & & \Bbbk \arrow[ul, "\begin{bsmallmatrix}0~1\end{bsmallmatrix}", swap]  \arrow[dl, "\begin{bsmallmatrix}0~-1\end{bsmallmatrix}"] \\
     & \Bbbk^2 \arrow[ul, "\begin{bsmallmatrix}0\\-1\end{bsmallmatrix}"]
    \end{tikzcd} \cong 
    {\arraycolsep=2pt\begin{array}{ccc}
 1 & & 3 \\ & 2 & \end{array}} \;\oplus \; {\arraycolsep=2pt\begin{array}{ccc}
  & 4 &  \\ 1& &3 \end{array}}
  \]
Define a symmetric form $\langle-,-\rangle$ on $M$ as follows.

\[ \langle x,y\rangle |_{M_i \times M_j} = \begin{cases}
    xy, & \text{if $(i,j)=(2,4)$ or $(4,2)$;}\\
    ad+bc, & \text{if $(i,j)=(1,3)$ or $(3,1)$ with } x =\begin{pmatrix} a \\ b \end{pmatrix}, y = \begin{pmatrix} c \\ d \end{pmatrix}; \\
    0, & \text{otherwise}.
\end{cases}\]
 It is routine to check that $(M, \langle-,-\rangle)$ is a $\epsilon$-representation for $\epsilon=+1$.
\end{example}

\begin{remark}
Instead of the more compact terminology `$\epsilon$-representations', \cite{DW02} calls them orthogonal (when $\epsilon=+1$) and symplectic (when $\epsilon=-1$) representations since the underlying vector space equipped with the form is an orthogonal/symplectic vector space.  We follows \cite{BCI21} practice.  We also remark that \cite{BCI21} uses only the complex number instead of arbitrary algebraically closed field of non-2 characteristic; the results we need from them can be argued in the latter more general setting.
\end{remark}

The notion of direct sum of $\epsilon$-representations is well-defined by naturally extending that of ordinary representations, i.e. the collection of matrices $(M_\alpha)_{\alpha\in Q_1}$ cannot be block-decomposed in a uniform way.  This allows one to talk about the notion of \dfn{$\epsilon$-indecomposability}. Notably, the Krull-Schmidt theorem applies in this context too; that is, every $\epsilon$-representation can be written as direct sum of the $\epsilon$-indecomposables in a unique way.  As far as the application within this article is concerned, the following result from \cite[2.7]{DW02}, \cite[2.10]{BCI21} suffices to act as a substitute of the proper definition of indecomposable $\epsilon$-representation.

\begin{proposition}\label{prop:e-indec}
Let $M$ be an indecomposable $\epsilon$-representation.  Then precisely one of the following three cases occur:
\begin{enumerate}[(a)]
\item $M$ is \dfn{1-sided}, i.e. $M$ is indecomposable as a $\Lambda$-module.

\item $M$ is \dfn{ramified}, i.e. $M\cong L\oplus \nabla L$ as $\Lambda$-module for some indecomposable $\Lambda$-module $L\cong \nabla L$.

\item $M$ is \dfn{split}, i.e. $M\cong L\oplus \nabla L$ as $\Lambda$-module for some indecomposable $\Lambda$-module $L\ncong \nabla L$.
\end{enumerate}
Moreover, for an indecomposable $\Lambda$-module $L$, it gives rise to exactly one of the three types of indecomposable $\epsilon$-representation of the form $M$ above.
\end{proposition}

\begin{remark}
Note that what is called `1-sided' here is called `type I' in \cite{DW02,BCI21}.
\end{remark}

From now on, we will often omit $\langle-,-\rangle$ from the notation as Proposition \ref{prop:e-indec} implies that the $\epsilon$-representation structure is determined by underlying module structure. 

If we already know the classification of indecomposable ordinary modules, then Proposition \ref{prop:e-indec} gives us a way to classify all indecomposable $\epsilon$-representations.  Namely, for each indecomposable $L\in \mod \Lambda$, we first check whether or not that module is self-dual. If it is not self-dual, then we have a split $\epsilon$-indecomposable $M=L\oplus\nabla L$.  Otherwise, we check whether one can equip an $\epsilon$-form $\langle-,-\rangle$ so that $(L,\langle-,-\rangle)$ defines a $\epsilon$-indecomposable.  If this is the case, then we have a 1-sided $\epsilon$-indecomposable on the spot; otherwise, $L\oplus \nabla L$ can be given a structure of a ramified $\epsilon$-indecomposable.

\begin{example}
Let us take $\epsilon=1$.  Consider the quiver from Example \ref{example:qpeasy} and the $\epsilon$-representation $M$ from Example \ref{ex::e-rep}.  As ordinary module we have $M = M' \oplus \nabla M'$ with the ordinary indecomposable module.

\[  M'=  \begin{tikzcd}[column sep=small, row sep = small]
     & \Bbbk \arrow[dl, "1", swap] \\
     \Bbbk  \arrow[rr, shift left] \arrow[rr, shift right, swap] & & 0 \arrow[ul]  \arrow[dl] \\
     & \Bbbk \arrow[ul,"1"]
    \end{tikzcd} = {\arraycolsep=2pt\begin{array}{ccc}
 1 & & 3 \\ & 2 & \end{array}} \]
\end{example}

\begin{example}
Consider the quiver from Example \ref{example:qpeasy} again and also the band $\omega:=\alpha_2\beta_2^{-1}$, which defines a family of indecomposable band $\Lambda$-module $M_{\lambda}(\omega)$ with $\lambda\in \Bbbk^\times$.  We have
\[
M_\lambda(\omega) = \begin{tikzcd}[column sep=small, row sep = small]
     & 0 \arrow[dl] \\
     \Bbbk  \arrow[rr, "{[1]}", shift left] \arrow[rr, "{[\lambda]}", shift right, swap] & & \Bbbk \arrow[ul]  \arrow[dl] \\
     & 0 \arrow[ul]
    \end{tikzcd} \]
and
\[ 
\nabla M_\lambda(\omega) = \begin{tikzcd}[column sep=small, row sep = small]
     & 0 \arrow[dl] \\
     \Bbbk  \arrow[rr, "{[-\lambda]}", shift left] \arrow[rr, "{[-1]}", shift right, swap] & & \Bbbk \arrow[ul]  \arrow[dl] \\
     & 0 \arrow[ul]
    \end{tikzcd} \cong \begin{tikzcd}[column sep=small, row sep = small]
     & 0 \arrow[dl] \\
     \Bbbk  \arrow[rr, "{[1]}", shift left] \arrow[rr, "{[\lambda^{-1}]}", shift right, swap] & & \Bbbk \arrow[ul]  \arrow[dl] \\
     & 0 \arrow[ul]
    \end{tikzcd} = M_{\lambda^{-1}}(\omega),
\]
so 1-sided and ramified $\epsilon$-indecomposables appear only when $\lambda \in\{+1,-1\}$. Define a bilinear form $\langle-,-\rangle$ on the underlying vector space given by $\begin{pmatrix}0 & 1\\ -\lambda & 0 \end{pmatrix}$.  Then this satisfies Definition \ref{def:e-rep} (3) only if $\epsilon=-\lambda \in \{+1,-1\}$.

Let $N=M_{-\epsilon}(\omega)$, $L'=M_{\epsilon}(\omega)$, and $L=L'\oplus \nabla L$, i.e.
 \[  N= \begin{tikzcd}[column sep=small, row sep = small]
     & 0 \arrow[dl] \\
     \Bbbk  \arrow[rr, "{[1]}", shift left] \arrow[rr, "{[-\epsilon]}", shift right, swap] & & \Bbbk \arrow[ul]  \arrow[dl] \\
     & 0 \arrow[ul]
    \end{tikzcd}
    \text{, }\;\;
    L' = \begin{tikzcd}[column sep=small, row sep = small]
     & 0 \arrow[dl] \\
     \Bbbk  \arrow[rr, "{[1]}", shift left] \arrow[rr, "{[\epsilon]}", shift right, swap] & & \Bbbk \arrow[ul]  \arrow[dl] \\
     & 0 \arrow[ul]
    \end{tikzcd}
    \text{ and }
    L= \begin{tikzcd}
     & 0 \arrow[dl] \\
     \Bbbk^2  \arrow[rr, "{\sbm{1\amp 0\\0\amp-\epsilon}}", shift left] \arrow[rr, "{\sbm{\epsilon\amp 0\\0\amp-1}}", shift right, swap] & & \Bbbk^2 \arrow[ul]  \arrow[dl] \\
     & 0 \arrow[ul]
    \end{tikzcd}.
    \]
Then one can find an appropriate bilinear form $\langle-,-\rangle$ so that $N$ defines a 1-sided $\epsilon$-indecomposable, whereas $L$ defines a ramified $\epsilon$-indecomposable.
\end{example}

\subsection{Curves as indecomposable symmetric representations}

Throughout this subsection, we fix the following notation. Let $\pi:\wti{\SM}:=(\wti{\S},\wti{\M}) \xrightarrow{2:1} \SM$ denote the orientable double cover of a non-orientable marked unpunctured surface $\wti{\SM}$, and let $\sigma_\S$ denote the associated orientation-reversing automorphism of $\wti{\S}$.  Fix a triangulation $T$ on $\SM$ and let $\wti{T}$ be the triangulation on $\wti{\SM}$ so that $\pi(\wti{T})=T$.
Let $(Q,W,\sigma)$ be the gentle symmetric QP associated to $\wti{T}$, $J=J_{Q,W}$ the associated Jacobian algebra, and $\C=\C_{\wti{T}}$ the associated cluster category.

\begin{lemma}\label{lem:nabla on strings}
Let $\gamma\in \bfC_{nc}\wti{\SM}$ be a non-closed curve on $\wti{\SM}$.
Then $\nabla(\gamma)\cong\sigma_\S(\gamma)$ as object on $\C$.
In particular, $M(\gamma)\oplus M(\sigma_\S(\gamma))$ has a structure of an indecomposable split $\epsilon$-representation.
\end{lemma}
\begin{proof}
Suppose $\gamma\in \wti{T}$.  Then $\gamma=\pi(e\Gamma)$, where $\Gamma$ is the Ginzburg dga associated to $(Q,W)$, $e$ is the primitive idempotent corresponding to $\gamma\in Q_0$, and $\pi:\mathsf{per}(\Gamma)\to \C$ is the canonical projection.  It follows from the definition of $\nabla$ and $\sigma_\S$ that $\nabla(e\Gamma)\cong \sigma(e)\Gamma$ and $\sigma(e)$ is the primitive idempotent corresponding to $\sigma_\S(\gamma)\in Q_0$.
This immediate implies that $\nabla(\gamma)=\pi(\nabla(e\Gamma))=\sigma_\S(\gamma)$.

For any non-closed $\gamma\notin \wti{T}$, we show that $\nabla(M(\gamma))\cong M(\sigma_\S(\gamma))$ and the claim follows by using Proposition \ref{prop:lift nabla}.  Indeed, first recall that the underlying vector space of $M(\gamma)$ is $\bigoplus_{i=0}^c\Bbbk \gamma(t_i)$ where $\gamma(t_i)$'s are the crossings of $\gamma$ with $\wti{T}$.  Let $\tau_i\in \wti{T}$ be the arc containing $\gamma(t_i)$.  By definition of $\sigma_\S$, the underlying vector space of $M(\sigma_\S(\gamma))$ is given by $\bigoplus_{i=0}^c\Bbbk \sigma_\S(\gamma)(1-t_i)$ with $\sigma_\S(\gamma)(1-t_i)\in \sigma_\S(\tau_i)$.  Hence, as vector spaces $M(\sigma_\S(\gamma))$ agrees with $\nabla(M(\gamma))$.  Consider now the arrow $\alpha_i$ determined by $\gamma|_{(t_{i-1},t_i)}$ and denote by $M_{\alpha_i}$ the action of $\alpha_i$ on $\Bbbk\gamma(t_j)\subset M(\gamma)$ with an appropriate $j\in\{i,i-1\}$.  Then $\sigma(\alpha_i)$-action on $\nabla(M(\gamma))$ is given by $-M_{\alpha_i}^*$, which is just the negative of the identity map.  Hence $\nabla(M(\gamma))\cong M(\sigma_\S(\gamma)$ via the map that multiplies all basis vector by $-1$.
\end{proof}

We will now determine the $\epsilon$-indecomposables arising from $M_\lambda(\omega)$.
Let us fix some notations and terminologies first.
From now on until further notice, $\omega$ will always be a primitive closed curve on $\wti{\SM}$.
We will identify $\omega$ with its band form $\omega=\omega_1\omega_2\cdots \omega_c$.
As in Section \ref{sec:cluster cat basic}, we will always assume without loss of generality that $\omega_1=\alpha_1\in Q_1$ is directed and $\omega_c=\alpha_c^{-1}$ is inverse; otherwise, one can rotate the indices until this criteria is met.
For convenience, we will call the number $c$ the \dfn{length} of $\omega$.  In picture, we can display $\omega$ and $M_\lambda(\omega^n)$ as follows.
\begin{align}
\omega =& \,\,\large(\xymatrix@1@C=38pt{
0 \ar[r]^{\alpha_1} & 1 \ar@{-}[r]^{\alpha_2} & \cdots \ar@{-}[r]^{\alpha_{c-1}} & c-1 & 0 \ar[l]_{\alpha_c}}\large)\label{eq:band quiver}\\
M_\lambda(\omega^n)=&\large(\xymatrix@1@C=35pt{
V_0 \ar[r]^{1} & V_1 \ar@{-}[r]^{1} & \cdots \ar@{-}[r]^{1} & V_{c-1} & V_0 \ar[l]_{J_{n}(\lambda)}}\large).\label{eq:band mod struc}
\end{align}
For simplicity, we let $v_i^j$ be the basis vector $\omega(z_i)^{(j)}$ of $V_i$, and take also $v_i^0:=0$, so that $\alpha_c$-action on $V_0$ is given by $v_0^j\alpha_c = v_{c-1}^{j-1}+\lambda v_{c-1}^j$.  In the case when $n=1$, we will further omit the superscript index as long as there is no confusion.  All arithmetic operations on the subscript index $v_i^j$ will be taken modulo $c$.

\begin{lemma}\label{lem:sigma rotates 1-scc}
If $\omega=\sigma_\S(\omega)$, then the length $c$ of $\omega$ is even, say, $c=2r$, and for all $i\in\{1,2,\ldots,c\}$, we have $\sigma(\alpha_i)=\alpha_{i+r}$ and $\omega_{i+r} = \sigma(\alpha_r)^{-1}$.
\end{lemma}
\begin{proof}
Viewing the picture \eqref{eq:band quiver} as a $c$-gon by forgetting the orientation, $\sigma$ acts as a non-identity element of the dihedral group $\langle \rho,\beta\mid \rho^c=1,\beta^2=1, \beta\rho\beta=\rho^{-1}\rangle$.  Note that $\sigma$ cannot act as the reflection $\beta$ as it will fix at least a vertex of the $c$-gon.  

Suppose on the contrary that $c$ is odd, then $\sigma$ acts as $\beta\rho^k$ for some $k$ as it is of order 2.  But then $\sigma$ will fix an edge of the $c$-gon, contradicting the fixed-point-free property of $\sigma$.

It remains to show that $\sigma$ acts cyclically on the $c$-gon.  Suppose the contrary, i.e. the source and target of the arrow $\alpha_i:=\sigma(\alpha_c)$ are ${i}$ and $i-1$ respectively.  Since the source of $\sigma(\alpha_{c-1})$ is given by applying $\sigma$ on the target of $\alpha_{c-1}$ (and vice versa), this means that $\sigma(\alpha_{c-j})=\alpha_{i-j}$ for all $j$.  Hence, there is some $1\leq k\leq i$ such that $\sigma(\alpha_k)=\alpha_k$, which contradicts the fixed-point-free property of $\sigma$.
\end{proof}

\begin{lemma}\label{lem:cc dual}
Let $\omega$ be a primitive closed curve on $\wti{\SM}$ and $n\geq 1$ be a positive integer.
Then $\nabla M_\lambda(\omega^n) \cong M_\mu(\sigma_\S(\omega^n))$ for some $\mu\in\{\lambda, \lambda^{-1}\}$; the same holds for the corresponding object $(\omega^n,\lambda)\in \C$.  Moreover, the $J$-module $M_\lambda(\omega^n)$ (respectively the coloured closed curve object $(\omega^n,\lambda)\in\C$) is self-dual if and only if $\sigma_\S(\omega)=\omega$ and $\lambda\in\{\pm 1\}$.
\end{lemma}
\begin{proof}
By Proposition \ref{prop:lift nabla}, we only need to show for the $J$-module case.
Moreover, we only need to argue the case $n=1$ as there are exact sequences
\[
0\to M_\lambda(\omega) \to M_\lambda(\omega^{n+1})\to M_\lambda(\omega^n)\to 0
\]
for all $n$ that allows us to iteratively apply the exact equivalence $\nabla$ to get the desired result for the case when $n>1$.

Similar to the construction of $M_\lambda(\omega)$.  Let us consider a module $N_\lambda^j(\omega)$ of the form
\[\xymatrix@1{
u_{0} \ar[r]^{1} & u_{1} \ar@{-}[r]^{1} & \cdots \ar@{-}[r]^{1} & u_{j-1} \ar@{-}[r]^{\lambda} & u_{j}\ar@{-}[r]^1  & \cdots \ar@{-}[r]^{1}& u_{c-1} & u_0 \ar[l]_{1} }
\]
where $\alpha_j$ acts by multiplying $\lambda$.
Note that $M_\lambda(\omega)=N_\lambda^{c}(\omega)$.

We claim that, if both $\alpha_j$ and $\alpha_{j-1}$ points in the same direction (i.e. either both $\omega_j$ and $\omega_{j-1}$ are arrows or both are inverses of arrows), then $N_\lambda^j(\omega)\cong N_\lambda^{j-1}(\omega)$; otherwise, $N_\lambda^j(\omega)\cong N_{\lambda^{-1}}^{j-1}(\omega)$.
Indeed, it is enough to see locally that the commutative diagrams
\[
\vcenter{\xymatrix{
\Bbbk \ar[r]^{1}\ar[d]^{1} & \Bbbk\ar[r]^{\lambda}\ar[d]^{\lambda} & \Bbbk \ar[d]^{1} \\
\Bbbk \ar[r]^{\lambda} & \Bbbk\ar[r]^{1} & \Bbbk \\
}}\quad\text{ and }\quad \vcenter{\xymatrix{
\Bbbk \ar[r]^{1}\ar[d]^{1} & \Bbbk\ar[d]^{\lambda^{-1}} & \Bbbk \ar[l]_{\lambda}\ar[d]^{1} \\
\Bbbk \ar[r]^{\lambda^{-1}} & \Bbbk& \Bbbk \ar[l]_{1} \\
}}
\]
induce the required isomorphism.

Following the same argument in the proof of Lemma \ref{lem:nabla on strings} (for the $J$-module case), we have $M_\lambda(\omega)=N_\lambda^c(\omega)\cong N_{\lambda}^j(\sigma_\S(\omega))$ for some $j\in\{1,2,\ldots, c-1\}$.  Note that $j\neq c$ as $\sigma$ is fixed-point-free on arrows.  By the claim in the previous paragraph, we have $N_{\lambda}^j(\sigma_\S(\omega))\cong M_{\mu}(\sigma_\S(\omega))$ for some $\mu\in\{\lambda,\lambda^{-1}\}$.  This finishes the proof of the first assertion.

To show the second part, it remains to show that $M_{\pm 1}(\omega)$ is self-dual when $\sigma_\S(\omega)=\omega$.  By the argument of translating $\lambda$ to a neighbouring arrow above, we only need to show that $\sigma(\alpha_c)$ and $\alpha_c$ points in opposite direction, which is already shown in Lemma \ref{lem:sigma rotates 1-scc}.
\end{proof}

We are going to determine the exact values of $n$ and $\lambda$ so that $M_\lambda(\omega^n)$ has a structure of an $\epsilon$-representation; hence, a 1-sided indecomposable $\epsilon$-representation.
Note that by Proposition \ref{prop:e-indec}, this implies that any other values of $n,\lambda$, the module $M_\lambda(\omega^n)^{\oplus 2}$ has a structure of a ramified indecomposable $\epsilon$-representation.

From now on until the end of the section, unless otherwise specified, we will always assume the primitive closed $\omega=\sigma(\omega)$ and $\lambda\in\{\pm 1\}$; otherwise, as $M_\lambda(\omega)$ is not self-dual and so it must give rise to split $\epsilon$-indecomposable.
We keep the symbol $c$ for the length of $\omega$, and let $r$ be $c/2$, which is an integer as guaranteed by Lemma \ref{lem:sigma rotates 1-scc}.

We will achieve our goal in two steps: Lemma \ref{lem:1-scc e-indec} and Lemma \ref{lem:1scc higher case}.  The first one will tells us which $M_\lambda(\omega^n)$ is an $\epsilon$-representation and the second one tells us that all other cases cannot have $\epsilon$-representation structure.
Both steps will follow same strategy - we start by taking an $\epsilon$-form $\langle -,-\rangle$ (for arbitrary $\epsilon\in\{\pm 1\}$) and use conditions (2) and (3) of Definition \ref{def:e-rep} to obtain a list of equations on the values $b_{i,l}^{j,k} := \langle v_i^j, v_l^k \rangle$, then we use these equations to determine the value of $\epsilon$ and whether the form is non-degenerate form.
For convenience, we also define $b_{i,l}^{j,k}:=0$ whenever one of $j,k$ is zero.

First, condition (2) of Definition \ref{def:e-rep} says that it is sufficient to consider only $b_{i,r+i}^{j,k}$.
Let us now write out the list of equations that condition (3) requires.
Recall from Lemma \ref{lem:sigma rotates 1-scc} that  $\sigma(\alpha_i)=\alpha_{r+i}$ for all $i\in\{1,2\ldots, c\}$.
If we consider an arrow $\alpha_i=\omega_i$, then $\sigma(\alpha_i)=\omega_{r+i}^{-1}=\alpha_{r+i}$, and we have $\langle v_{i-1}^j\alpha_i, v_{r+i}^k\rangle = -\langle v_{i-1}^j, v_{r+i}^k\alpha_{r+i}\rangle$.  Writing down similar equations for the cases when $\alpha_i\neq\omega_i$, and evaluating both sides of all these equations yields the following set of conditions indexed over $i\in\{1,2\ldots, c\}$ and $j,k\in\{1,2,\ldots, n\}$:
\begin{align}
\langle v_{i}^j, v_{r+i}^k\rangle & = -\langle v_{i-1}^j, v_{r+i-1}^k\rangle\quad \forall i\in\{1,2,\ldots, c-1\}\setminus\{r\} \tag{Eq:$i,j,k$} \\
\langle v_{r}^j, v_{0}^k\rangle & = -\langle v_{r-1}^j, v_{c-1}^{k-1}\rangle - \lambda\langle v_{r-1}^j, v_{c-1}^{k}\rangle \tag{Eq:$r,j,k$}\\
\langle v_{0}^j, v_{r}^k\rangle & = -\langle v_{c-1}^{j-1}, v_{r-1}^{k}\rangle - \lambda\langle v_{c-1}^j, v_{r-1}^{k}\rangle \tag{Eq:$c,j,k$}
\end{align}

For each $j,k$, going through (Eq:$i,j,k$) from $i=1$ to $i=c-1$ says that $(b_{i,r+i}^{j,k})_{i\in \Z/c\Z}$ is dependent only on $b^{j,k}:=b_{0,r}^{j,k}$, namely, 
\begin{align}
b_{i,r+i}^{j,k} = \begin{cases}
(-1)^{i}b^{j,k}, & \text{if $0\leq i< r$;}\\
(-1)^{i}( b^{j,k}\lambda + b^{j,k-1} ), & \text{if $r\leq i< c$.}\\
\end{cases} \label{eq:only b_1}
\end{align}
Putting this into (Eq:$c,j,k$) and rearranging the equation (also using $\lambda^2=1$ and $c\in 2\Z$) yields
\begin{align}
b^{j-1,k-1} &= \lambda(b^{j-1,k}+b^{j,k-1}). \label{eq:jk} \tag{Eq:$j,k$}
\end{align}
Note that this is a null condition in the case when $j=k=1$, as (Eq:$c,1,1$) says only $b_{0,r}^{1,1}=-\lambda b_{c-1,r-1}^{1,1}$, which is already guaranteed by \eqref{eq:only b_1}.

Summarising what we have so far:
\begin{lemma}\label{lem:condition 3 renewed}
$M_\lambda(\omega^n)$ is a 1-sided indecomposable $\epsilon$-representation if and only if there is a non-degenerate $\epsilon$-form $\langle-,-\rangle$ such that  $b_{i,l}^{j,k}:=\langle v_i^j, v_l^k\rangle=0$ for all $l\neq r+i$, and the equations \eqref{eq:only b_1} and \eqref{eq:jk} hold for all $i,j,k$.
\end{lemma}

\begin{lemma}\label{lem:1-scc e-indec}
The following hold.
\begin{enumerate}[(a)]
\item The indecomposable module $M_{\lambda}(\omega)$ is an $\epsilon$-representation if and only if $\epsilon=(-1)^r\lambda$.
\item The indecomposable module $M_{\lambda}(\omega^2)$ is an $\epsilon$-representation if and only if $\epsilon=(-1)^{r+1}\lambda$.
\end{enumerate}
\end{lemma}
\begin{proof}
(a) We define a bilinear form so that its values on the basis is given by
\[
\langle v_i, v_j \rangle := \begin{cases}
(-1)^{i}, & \text{if $i\in 0\leq i<r$ and $j=r+i$;}\\
(-1)^{i}\lambda, & \text{if $r\leq i< c$ and $j=r+i$;}\\
0,  &\text{otherwise.}
\end{cases}
\]
It is straightforward to check that this is a non-degenerate $\epsilon$-form for $\epsilon=(-1)^r\lambda$.
Every bilinear form satisfying \eqref{eq:omega cross T} will be a scalar multiple of the one here (which has $b^{1,1}=1$).  The assertion follows.

(b) We first show the only-if direction.  
By \eqref{eq:only b_1}, we have $b_{r+i,i}^{2,1}=(-1)^{r+i}\lambda b^{2,1}$ and $b_{i,r+i}^{1,2}=(-1)^{i}b^{1,2}$ for all $0\leq i< r$.
If $j\neq k$, then \eqref{eq:jk} says that $\lambda b^{1,1}=0$.
Now substituting $b^{1,1}=0$ into \eqref{eq:jk} for the case when $j=k=2$, we get that $b^{2,1}=-b^{1,2}$.
Hence, we have $b_{r+i,i}^{2,1}=(-1)^{r+i+1}\lambda b^{1,2}$ for $0\leq i< r$.  

By definition, $b_{i,l}^{j,k}$ defines an $\epsilon$-form if and only if $b_{i,r+i}^{1,2}=\epsilon b_{r+i,i}^{2,1}$.  Substituting the equations from the previous paragraph yields
\[
(-1)^{i}b^{1,2} = b_{i,r+i}^{1,2} = \epsilon b_{r+i,i}^{2,1} = \epsilon (-1)^{r+i+1}\lambda b^{1,2},
\]
i.e. $\epsilon=(-1)^{r+1}\lambda$ as claimed.

For the converse, define a bilinear form given by
\[
\langle v_i^j, v_l^k \rangle = b_{i,l}^{j,k} :=\begin{cases}
(-1)^{i+j}, & \text{if $j\neq k$ and $1\leq i\leq r$ and $l=r+i$;}\\
(-1)^{i+j}\lambda, & \text{if $j\neq k$ and $r< i\leq c$ and $l=r+i$;}\\
(-1)^{i+1}\lambda/2, & \text{if $j=k=2$ and $1\leq i\leq r$ and $l=r+i$;}\\
(-1)^{i+1}/2, & \text{if $j=k=2$ and $r< i\leq c$ and $l=r+i$;}\\
0,  &\text{otherwise.}
\end{cases}
\]
Note that here we have $b^{1,1}=0$, $b^{1,2}=1$, $b^{2,1}=-1$, and $b^{2,2}=\lambda/2$.

Again, it is straightforward to check that \eqref{eq:only b_1} holds.
As argued before for the only-if direction, \eqref{eq:jk} yields $b^{1,1}=0$ and $b^{2,1}=-b^{1,2}$ - both of which are satisfied in our case.

Finally, as $\epsilon = (-1)^{r+1}\lambda$, we have $ b_{i,r+i}^{2,2} = (-1)^{r+1}\lambda b_{r+i,i}^{2,2}$ for all $0\leq i< r$.  Apply \eqref{eq:only b_1} to both sides yields
\[
(-1)^{i}b^{2,2} = b_{i,r+i}^{2,2} = (-1)^{r+1}\lambda b_{r+i,i}^{2,2} = (-1)^{r+1+r+i}\lambda(\lambda b^{2,2} + b^{2,1}).
\]
Since $\lambda^2=1$ and $b^{2,1}=-b^{1,2}$, this equation rearranges to $b^{2,2} = b^{1,2}\lambda/2$.  Hence, our bilinear form defines an $\epsilon$-representation on $M_\lambda(\omega^2)$.
\end{proof}

Now, we look at higher $n$.

\begin{lemma}\label{lem:bjk up to antidiag}
The equations \eqref{eq:eps j,k} forces $b^{j,k}=0$ for all $j+k\leq n$, and $b^{j,n+1-j}=(-1)^{j+1}b^{1,n}$ for all $1\leq j\leq n$.
\end{lemma}
\begin{proof}
For $j=1$, the equation \eqref{eq:jk} says that $b^{1,k-1}=0$ for all $1\leq k\leq n$; likewise, taking $k=1$ in equation \eqref{eq:jk} yields $b^{j-1,k}=0$ for all $1\leq j\leq n$.
Now we go through the equations \eqref{eq:jk} starting from the $j+k=4$ (the equation for $j+k=2$ is null and the equations for $j+k=3$ are included in the previous sentence) to $j+k=n+1$.
For each fixed $\ell:=j+k\in \{4,\ldots,n+1\}$, if we iterate the equations \eqref{eq:jk} from $j=1$ to $j=\ell-1$, then in each iteration, we obtain $b^{j,k-1} = 0$.  Hence we have the first part of the assertion.  In a similar way, taking $j+k=n+2$ in \eqref{eq:jk} yields $b^{j,n+1-j} = -b^{j-1,n+2-j}$, and the second part of the assertion follows.
\end{proof}

\begin{lemma}\label{lem:1scc higher case}
If $M_\lambda(\omega^n)$ is an $\epsilon$-representation, then the following hold.
\begin{enumerate}[(a)]
\item $b^{1,n}\neq 0$.
\item $\epsilon=\lambda (-1)^{n+r+1}$.
\item $n\leq 2$.
\end{enumerate}
In particular, $M_\lambda(\omega^n)\oplus \nabla(M_\lambda(\omega^n))$ is a ramified $\epsilon$-indecomposable for all $n\geq 3$ and $\lambda\in \{\pm 1\}$.
\end{lemma}
\begin{proof}
(a) Suppose the contrary.  Then combining Lemma \ref{lem:bjk up to antidiag} yields $b^{1,k}=0$ for all $k\in\{1,\ldots, n\}$.  By \eqref{eq:only b_1}, we then have $\langle v_{i}^{(1)}, -\rangle =0$ for all $i$.  This contradicts the non-degeneracy of $\langle-,-\rangle$.

(b) Since $\langle-,-\rangle$ is an $\epsilon$-form, we have $b_{i,r+i}^{j,k}=\epsilon b_{r+i,i}^{k,j}$.  By \eqref{eq:only b_1}, the case $0\leq i< r$ yields
\begin{align}
(-1)^{i}b^{j,k} &= \epsilon (-1)^{r+i}(\lambda b^{k,j}+b^{k,j-1}), \notag\\
\text{hence, }\quad b^{j,k} &= \epsilon (-1)^r(\lambda b^{k,j}+b^{k,j-1}).\label{eq:eps j,k}
\end{align}
Note that if $j+k=n+1$, then $b^{j,k-1}$ vanishes, so by Lemma \ref{lem:bjk up to antidiag} the equation \eqref{eq:eps j,k} in these cases become $(-1)^{j+1}b^{1,n} = \epsilon(-1)^r (\lambda (-1)^{n-j}b^{1,n})$, rearranging yields
\begin{align*}
b^{1,n} &= \epsilon \lambda (-1)^{n+r+1}b^{1,n}.
\end{align*}
Hence, by (a) and the assumption that $\epsilon,\lambda\in\{\pm 1\}$ we get $\epsilon=\lambda(-1)^{n+r+1}$ as required.

(c) By rearranging equation \eqref{eq:eps j,k} after substituting $\epsilon=\lambda(-1)^{n+r+1}$ from (b), we obtain 
\[
(-1)^{n+1}b^{j,k}+b^{k,j}+\lambda b^{k,j-1}=0.\]
Now consider the equation with $(j,k)=(2,n)$ and with $(j,k)=(n,2)$.\footnote{
Note that when $n=2$ the two equations are the same, which is why the remaining of the proof does not work in this case.}

When $n>2$, we can multiply the latter equation by $(-1)^{n+1}$ and subtract it from the first equation, which yields
\[
(-1)^nb^{2,n-1}+b^{n,1}=0.\]
On the other hand, by Lemma \ref{lem:bjk up to antidiag}, we can rewrite $b^{2,n-1}=-b^{1,n}$ and $b^{n,1}=(-1)^{n+1}b^{1,n}$ in terms of $b^{1,n}$, which means that we have $2(-1)^{n+1}b^{1,n}=0$.  Hence, we deduce that $b^{1,n}=0$ - a contradiction.
\end{proof}

Let us summarise our investigation so far. Recall (from the discussion after Lemma \ref{lem:nabla on strings}) that the length $\mathrm{len}(\omega)$ of a closed curve $\omega$ is the number of intersections of $\omega$ with the initial triangulation.

\begin{theorem}\label{thm:indec e-reps}
Let $(J=J_{Q,W}, \sigma)$ be a FF-symmetric gentle Jacobian algebra.
If $M$ is an indecomposable $\epsilon$-representation over $(J,\sigma)$, then exactly one of the following hold.
\begin{itemize}
\item $M$ is a split $\epsilon$-indecomposable with underlying $J$-module being either one of the following:
\begin{itemize}
\item[(SS)] $M(\gamma)\oplus M(\sigma(\gamma))$ for some curve $\gamma$ with endpoints.

\item[(SB)] $M_\lambda(\delta)$ for some (not necessarily primitive) closed curve $\delta$ with $\sigma(\omega)\neq\omega$ or $\lambda \neq\pm 1$.
\end{itemize}

\item $M$ is a 1-sided $\epsilon$-indecomposable with underlying $J$-module being $M_\lambda(\omega^n)$ with $\omega=\sigma(\omega)$ a primitive closed curve, $\lambda\in \{\pm 1\}$, $n\leq 2$, and $\epsilon=(-1)^{\mathrm{len}(\omega)/2+n-1}\lambda$.

\item $M$ is a ramified $\epsilon$-indecomposable.
\end{itemize}
\end{theorem}
\begin{proof}
Proposition \ref{prop:e-indec} that $M$ is a split $\epsilon$-indecomposable if and only if $M=N\oplus \nabla N$ for some indecomposable non-self-dual $J$-module $N\ncong \nabla N$.
Thus, the first case follows from Lemma \ref{lem:nabla on strings} and Lemma \ref{lem:cc dual}, as they combine to say that non-self-dual indecomposable $J$-modules are precisely those that satisfy the condition (SS) or (SB).  

On the other hand, Proposition \ref{prop:e-indec} says that 1-sided $\epsilon$-indecomposable are given by self-dual indecomposable $J$-module that can be equipped with an $\epsilon$-representation structure.  By Lemma \ref{lem:1scc higher case}, such an indecomposable $J$-module must be of the form given in Lemma \ref{lem:1-scc e-indec}.  Now the rest of the claim follows.
\end{proof}

\section{Categorifying quasi-triangulations}\label{sec:categorify}

In this section, we lift our work from Section \ref{sec:symm repn} to the cluster category.  Recall that when $\SM$ is orientable with triangulation $T$, triangulations are categorified by \emph{cluster-tilting objects} in the cluster category $\C_{T}$.  We will formulate a symmetric representation theoretic analogue of cluster-tilting to categorify quasi-triangulations for the case when $\SM$ is non-orientable.

We will use the same notations as in the previous section, that is:
\begin{itemize}
\item  $\SM$ is a non-orientable unpunctured marked surface with $T$ a triangulation on it.

\item  $\wti{\SM}$ is the double cover of $\SM$ with Deck transformation group $\langle \sigma_\S \rangle$ and $\wti{T}$ be the induced double cover of $T$.

\item $(Q,W,\sigma)$ is the symmetric QP associated to $(\wti{T},\sigma_\S)$, $\C:=\C_{(Q,W)}$ is the associated cluster category, and $M(\cdot):\C\to \mod J$ be the projection on the module category of $J=J_{Q,W}\cong \End_\C(\Gamma)$, where $\Gamma$ is the cluster-tilting object of $\C$ given by $\wti{T}[-1]:=\bigoplus_{\gamma\in \wti{T}}\gamma[-1]$.

\item $\epsilon\in\{\pm1\}$ and by $\epsilon$-representation we always mean $\epsilon$-representation of $J$ with respect to the involution $\sigma$.
\end{itemize}

\subsection{Lifting \texorpdfstring{$\epsilon$}{e}-representation theory to the cluster category}

Before going into the definitions, note that an $\epsilon$-representation $M$ comes with a canonical isomorphism $\psi_M:\nabla M\xrightarrow{\sim} M$ such that $\nabla(\psi_M)=\epsilon \psi_M$, and conversely, specifying such an isomorphism on an ordinary representation $M$ is equivalent to specifying an $\epsilon$-representation structure; see \cite[Sec 2.3]{BCI21}.  We can use this idea to lift symmetric representations to the cluster category $\C$ as follows.

\begin{definition}\label{def:e-obj}
Let $X\in \C=\C_{(Q,W)}$ and write $X=T'[1]\oplus Y$ a decomposition with $T'[1]$ the maximal direct summand of $X$ in $\add(\wti{T}[1])$.
\begin{enumerate}[(i)]
\item We say that $X$ is an \dfn{$\epsilon$-object} if there is an isomorphism $\psi_X: \nabla X\xrightarrow{\sim} X$ such that $\nabla \psi_X = \epsilon \psi_X$.  If, moreover, either one of the following cases occur:
\begin{itemize}
\item $T'=0$ and $M(Y)\in \mod J$ is an indecomposable $\epsilon$-representation of $J$, 
\item $Y=0$ and $T'=\nabla(\alpha)\oplus \alpha$ for some arc $\alpha\in \wti{T}$,
\end{itemize}
then we call $X$ an \dfn{indecomposable $\epsilon$-object} (or $\epsilon$-indecomposable for short).  To distinguish the two cases, we will call them \dfn{non-initial} and \dfn{initial} respectively.

\item Suppose $X$ is an indecomposable $\epsilon$-object.  The notion \dfn{split}, \dfn{1-sided}, \dfn{ramified} in Proposition \ref{prop:e-indec} extends naturally to $\epsilon$-indecomposable objects by regarding the case $X=T'=\nabla(\alpha)\oplus\alpha$ as split.  

\item For an $\epsilon$-indecomposable $X$, we define \dfn{$\epsilon$-factors} of $X$ as follows.
\begin{itemize}
\item If $X=\gamma\oplus\nabla(\gamma)$ is split or ramified, then an $\epsilon$-factor of $X$ is either $\gamma$ or $\nabla(\gamma)$.

\item If $X$ is one-sided with $X=(\omega^2,\lambda)$ for some primitive closed curve $\omega = \omega_1\cdots \omega_c$, then an $\epsilon$-factor of $X$ is $(\omega,\lambda)$.

\item If $X$ is one-sided with $X=(\omega,\lambda)$ for some primitive closed curve $\omega=\omega_1\cdots \omega_{2r}$, then an $\epsilon$-factor of $X$ is a non-closed curve $\alpha = \omega_{i+1}\cdots \omega_{i+r-1}$ so that $\omega_i$ is an inverse arrow (or equivalently $\omega_{i+r}$ is a direct arrow).
\end{itemize}
\end{enumerate}
\end{definition}

\begin{example}\label{eg:initial arc nabla}
Consider an initial arc $\alpha\in \wti{T}$.
Under the correspondence \eqref{eq:curve corresp}, we have an object $X:=\alpha\oplus \sigma_\S(\alpha)\in \C$, and the involution $\sigma$ on $(Q,W)$ induces an isomorphism $\psi_\alpha:\nabla(\alpha)\to \sigma_\S(\alpha)$.
Then we have
\[
\psi_X:= \begin{pmatrix}
0 & \psi_{\alpha}\\
\epsilon \psi_{\sigma\alpha} & 0 
\end{pmatrix} : \underbrace{\alpha\oplus \nabla\alpha}_{\nabla X} \to \underbrace{\sigma_\S(\alpha)\oplus\alpha}_{X}
\]
that satisfies $\nabla(\psi_X)=\epsilon\psi_X$.
Every initial $\epsilon$-object arise this way.
\end{example}

Let us give a more concrete example as well.
\begin{example}
Consider the quiver from Example \ref{example:qpeasy}.
Let $S_x$ denotes the simple module corresponding to vertex $x\in Q_0$.
Since $\nabla(S_1)\cong S_{1'}$, $S_1 \oplus S_{1'}$ is a split $\epsilon$-indecomposable representation, and we have an isomorphism $\psi_{M(X)}:S_1\oplus S_{1'}\to S_1\oplus S_{1'}$ that satisfies $\nabla(\psi_{M(X)})=\epsilon\psi_{M(X)}$.

Let $X\in \C$ be the lift of $S_1\oplus S_{1'}$, i.e. $M(X)=S_1\oplus S_{1'}$, and $\psi_X\in \C$ be the lift of $\psi_{M(X)}$.  Then $(X,\psi_X)$ defines a split $\epsilon$-indecomposable of string type.

Consider the band module $M_\lambda(\omega)$ in Example \ref{eg:band eg} associated to the primitive band $\omega$.
By Theorem \ref{thm:indec e-reps}, this is a 1-sided $\epsilon$-indecomposable if $\lambda=\epsilon$; ramified if $\lambda=-\epsilon$.  Hence, $(\omega,\epsilon)$ is a 1-sided $\epsilon$-indecomposable object in $\C$ and $(\omega, -\epsilon)\oplus (\omega,-\epsilon)$ is a ramified $\epsilon$-indecomposable object in $\C$.
\end{example}
As an application Theorem \ref{thm:indec e-reps}, we have the following correspondences between curves on $\SM$ with $\epsilon$-indecomposables in $\C$.

\begin{theorem}\label{thm:curve corresp} 
For every curve $\gamma\in \bfC\SM$, fix a lift $\wti{\gamma}\in \bfC\wti{\SM}$.
\begin{enumerate}
\item If $\gamma\in \Cnc\SM$, then correspondence \eqref{eq:curve corresp} induces the following bijection
\begin{center}
\begin{tikzpicture}
\begin{scope}[shift={(0,4.5)}]
\node[align=center] (L1)at (0,0) {$\left\{\begin{array}{c}\text{curves connecting}\\\text{marked points}\end{array}\right\}$};
\node[right] (R1) at (3.5,0) {$\left\{\begin{array}{c}\epsilon\text{-indecomposable objects} \\ \text{of split string type}\end{array}\right\}$};
\draw[<->] (L1) -- node[midway,above]{1:1} (R1);
\node (l1) at (1.3,-1) {$\gamma$};
\node[right] (r1) at (3.7,-1) {$\wti{\gamma}\oplus\nabla\wti{\gamma}$.};
\draw[|->] (l1)--(r1);
\end{scope}
\end{tikzpicture}
\end{center}

\item If $\omega \in \Ccc^1\SM = \Cocc^1\SM\sqcup\Ctcc^1\SM$ is a primitive closed curve, then there are the following bijections
\begin{center}
\begin{tikzpicture}
\begin{scope}[shift={(0,2.25)}]
\node[align=center] (L1)at (0,0) {$\left\{\begin{array}{c}\text{primitive 1-sided}\\\text{closed curves}\end{array}\right\}$};
\node[right] (R1) at (3.5,0) {$\left\{\begin{array}{c}\epsilon\text{-indecomposable objects} \\ \text{of 1-sided primitive band type}\end{array}\right\}$};
\draw[<->] (L1) -- node[midway,above]{1:1} (R1);
\node (l1) at (1.3,-1) {$\omega$};
\node[right] (r1) at (3.7,-1) {{$(\wti{\omega},\epsilon(-1)^{\mathrm{len}(\wti{\omega})/2})$,}};
\draw[|->] (l1)--(r1);
\end{scope}

\begin{scope}[shift={(0,0)}]
\node[align=center] (L1)at (0,0) {$\left\{\begin{array}{c}\text{primitive 2-sided}\\\text{closed curves}\end{array}\right\}$};
\node[right] (R1) at (3.5,0) {$\left\{\begin{array}{c}\epsilon\text{-indecomposable objects}\\ \text{of ramified or split}\\\text{primitive band type}\end{array}\right\}$};
\node at (8.7,-0.3) {$\Big/\sim$};
\draw[<->] (L1) -- node[midway,above]{1:1} (R1);
\node (l1) at (1.3,-1) {$\omega$};
\node[right] (r1) at (3.7,-1) {$[\, (\wti{\omega},\lambda)\oplus\nabla(\wti{\omega},\lambda) \,],$};
\draw[|->] (l1)--(r1);
\end{scope}
\end{tikzpicture}
\end{center}
where the equivalence $\sim $ in the last row is defined by $(\omega,\lambda)\oplus\nabla(\omega,\lambda) \sim (\omega',\lambda')\oplus\nabla(\omega',\lambda')$ if $\omega'$ is one of $\omega$ or $\nabla(\omega)$.
\end{enumerate}
\end{theorem}
Note that the equivalence is effectively forgetting the colouring on the closed curves.
\begin{proof}
Suppose $\overline{X}$ is an indecomposable object in $\C$.
If $\overline{X}$ is initial, then the $\overline{X}\oplus\nabla(\overline{X})$ is an $\epsilon$-object as discussed in Example \ref{eg:initial arc nabla}.
If $\overline{X}$ is non-initial, then $M(\overline{X})$ is an indecomposable $J$-module.  This gives rise to an indecomposable $\epsilon$-representation $M(X)$ for some $X\in \C$, and hence an isomorphism $\psi_{M(X)}$ in $\mod J$ satisfying $\nabla(\psi_{M(X)})=\epsilon \psi_{M(X)}$.  Thus, by Proposition \ref{prop:lift nabla}, this lifts to an isomorphism $\psi_X$ in $\C$ satisfying $\nabla(\psi_X)=\epsilon\psi_X$.
Hence, non-initial $\epsilon$-indecomposable objects in $\C$ are in bijection with indecomposable $\epsilon$-representations of $J$.

For indecomposable $\epsilon$-representations of $J$ that come from string modules,  Theorem \ref{thm:indec e-reps} says that they are always of the form $M(\gamma) \oplus M(\sigma(\gamma))$ for some non-initial $\gamma\in \Cnc\wti{\SM}\setminus\wti{T}$.
Since $\sigma$-orbits of $\Cnc\wti{\SM}\setminus\wti{T}$ is just $\Cnc\SM\setminus T$, we get the bijection in (1).

Using the classification of one-sided indecomposable $\epsilon$-representations in Theorem \ref{thm:indec e-reps} and restricts to primitive bands, and the fact that $\Cocc\SM$ is the set of fixed points under $\sigma$ in $\Ccc^1\wti{\SM}$, the bijection for $\Cocc^1\SM$ follows.

Finally, since $\Ccc^1\wti{\SM}/\sigma = \Cocc^1\SM\sqcup\Ctcc^1\SM$, the correspondence of the claim follows from the classification of ramified and split indecomposable $\epsilon$-representations of the form $M_\lambda(\omega)$ for primitive $\omega$ in Theorem \ref{thm:indec e-reps}.
\end{proof}

Recall that $\Ext_\C^1(X,Y)$ is defined as $\Hom_\C(X,Y[1])$.  Therefore, for each $f\in \Hom_\C(X,Y[1])$, we have a triangle
\[
Y \to C_f \to X \xrightarrow{f} Y[1]
\]
in $\C$.  So far as context is clear, we call the morphism $f$, the cocone $C_f$, as well as any isomorphic triangle an \emph{extension from $X$ to $Y$}.  Recall also that triangulations are in bijection with cluster-tilting (i.e. maximal rigid) objects of $\C$, and rigid objects are those with only split self-extension.

In the following, we give a `symmetric analogue' of these notions.

\begin{definition}\label{def:e-ext}
Let $X,Y$ be $\epsilon$-objects in $\C$ and $Y'$ be an $\epsilon$-factor of $Y$.
We call $f\in \Hom_\C(X,Y'[1])$ an \dfn{$\epsilon$-extension} from $X$ to $Y$ if $f\circ\psi_X\circ \nabla(f)=0$.  
The case when $f=0$ is called a \dfn{trivial $\epsilon$-extension}.

We call $X$ \dfn{$\epsilon$-rigid} if every $\epsilon$-extension from $X$ to itself is trivial, i.e. for any $\epsilon$-factor $X'$ of $X$, a morphism $f:X\to X'[1]$ satisfies $f\circ\psi_X\circ\nabla(f)=0$ implies that $f=0$.
\end{definition}
\begin{remark}
$\epsilon$-extensions are closed under scalar multiple, but unlike ordinary morphisms, they are not necessarily closed under addition.
\end{remark}

Note that, if we consider the triangle $C_f\to X\xrightarrow{f} Y'[1]\to $ with $C_f$ being the cocone of $f$, and consider also the long exact sequence induced by applying $\Hom_\C(\nabla(Y')[-1],-)$ to this triangle, then we can see that $f\circ\psi_X\circ \nabla(f)=0$ is equivalent to saying that $\psi_X\circ\nabla(f)$ factors through $C_f$.  In particular, we have the following commutative diagram from the octahedral axiom of triangulated categories
\[
\xymatrix{
& \nabla (Y')[-1]\ar[d] \ar@{=}[r] & \nabla (Y')[-1]\ar[d]^{\psi_X\nabla(f)} & \\
Y'\ar[r]\ar@{=}[d] & C_f\ar[r]\ar[d] & X \ar[r]^{f}\ar[d] & Y'[1]\ar@{=}[d] \\
Y'\ar[r]& E\ar[r] \ar[d] & D^f\ar[r]\ar[d] & Y'[1]\\
& \nabla(Y')\ar@{=}[r]& \nabla(Y')& \\
}
\]
where every row and every column is a triangle in $\C$.
We call this diagram the \dfn{$\epsilon$-extension diagram associated to $f$}, and $E$ the \dfn{$\epsilon$-extension from $X$ to $Y$} by abusing terminology (as in the classical case).

Since $M(-):\C\to \mod J$ is a cohomological functor, to find $\epsilon$-extensions from $X$ to $Y$ in $\C$, it suffices to find a commutative diagram 
\[\xymatrix{
 & & 0\ar[d] & 0\ar[d] & \\
0\ar[r] & M(Y')\ar[r]\ar@{=}[d] & M(C_f) \ar[d]\ar[r] & M(X)\ar[d]\ar[r] & 0 \\
0\ar[r] & M(Y')\ar[r] & M(E) \ar[d]\ar[r]& M(D^f)\ar[r]\ar[d] &0 \\
& & M(\nabla(Y'))\ar[d]\ar@{=}[r] & M(\nabla(Y'))\ar[d] & \\
 & & 0 & 0 & 
}
\]
in $\mod J$ where all rows and columns are exact, with the first row being equivalent (as a short exact sequence) to the $\nabla$-dual of the right hand column.  This is the strategy we will use to determine (non-)$\epsilon$-rigidity throughout.

\subsection{\texorpdfstring{$\epsilon$}{e}-rigidity for split and ramified \texorpdfstring{$\epsilon$}{e}-indecomposables}

\begin{lemma}\label{lem:split ramified rigid}
Consider an $\epsilon$-indecomposable object $X\in \C$.
\begin{enumerate}[(i)]
\item In the case when $X$ is a split $\epsilon$-indecomposable, we have that $X$ is $\epsilon$-rigid if and only if it is rigid.

\item If $X$ is a ramified $\epsilon$-indecomposable, then $X$ is never $\epsilon$-rigid.
\end{enumerate}
\end{lemma}
\begin{proof}
(1) Write $X=\gamma\oplus\nabla(\gamma)$ for an indecomposable object $\gamma$ and so $\Ext_\C^1(X,X)$ is a direct sum of $\Ext_\C^1(\delta,\delta')$ over all $\delta,\delta'\in\{\gamma,\nabla(\gamma)\}$.  If $X$ is rigid, then all of these individual spaces is necessarily zero, so $X$ is clearly $\epsilon$-rigid.
Conversely, we have a non-zero morphism $f:\delta\to \delta'[1]$ for some $\delta,\delta'\in\{\gamma,\nabla(\gamma)\}$ which yields $(f\pi)\psi_X\nabla(f\pi)=0$ for $\pi:X\to \delta$ the natural projection, and so $X$ cannot be $\epsilon$-rigid.

(2) If $X$ is ramified, then $X\cong(\omega,\lambda)^{\oplus 2}$ for some closed curve $\omega$ and some $\lambda\in\Bbbk^\times$.  But there is no rigid band object $(\omega,\lambda)$, so we must have a non-zero $f:(\omega,\lambda)\to (\omega,\lambda)[1]$, and apply the same argument as the converse part of (1) to see that $X$ is not $\epsilon$-rigid.
\end{proof}

\begin{lemma}\label{lem:1sided non-eps-rigid}
Suppose $\omega=\sigma(\omega)$ is a primitive closed curve such that $(\omega^2,\lambda)\in\C$ is a 1-sided $\epsilon$-indecomposable object.  Then $X$ is not $\epsilon$-rigid.
\end{lemma}
\begin{proof}
Recall that there is a functor $G_\omega: \mod\Bbbk[x,x^{-1}]\to \mod J$ so that the full subcategory of modules of the form $M_\lambda(\omega^n)$ can be understood from the representation theory of $\Bbbk[x,x^{-1}]$-module; see for example from \cite[II.3, II.4]{Erd}.
We have $M_\lambda(\omega^n)=G_{\omega}(V_\lambda^n)$ corresponds to the indecomposable $\Bbbk[x,x^{-1}]$-module $V_\lambda^n$ where $x$ acts as the $\lambda$-Jordan block of size $n$.
For any $n\geq 1$, let $v_n$ be the generator of $V_\lambda^n$, and so $V_\lambda^n$ have basis $\{(x-\lambda)^kv_n\}_{1\leq k\leq n}$.
Define $\iota_n:V_\lambda^n \to V_\lambda^{n+1}$ to be the map $v_n\mapsto (x-\lambda)v_{n+1}$ and $\pi_n:V_\lambda^{n+1} \to V_\lambda^n$ to be the map $v_{n+1}\mapsto v_n$.
Let $\iota_{n,1}:=\iota_n\iota_{n-1}\cdots\iota_1$ and $\pi_{1,n}:=\pi_1\pi_2\cdots \pi_n$.
Then we have short exact sequences
\[
\xymatrix@R=2pt@C=45pt{
\xi_{n}:\;\; 0 \ar[r] & V_\lambda^1 \ar[r]^{\iota_{n,1}} & V_\lambda^{n+1} \ar[r]^{\pi_{n}} & V_\lambda^{n} \ar[r] & 0, \\
\zeta_{n}:\;\; 0 \ar[r] & V_\lambda^n \ar[r]^{\iota_{n}} & V_\lambda^{n+1} \ar[r]^{\pi_{1,n}} & V_\lambda^{1} \ar[r] & 0, \\
\rho_n:\;\; 0 \ar[r] & V_\lambda^n \ar[r]^{(\pi_{n-1},\iota_n)^\top\quad} & V_\lambda^{n-1}\oplus V_\lambda^{n+1} \ar[r]^{\quad(\iota_n,\pi_{n+1})} & V_\lambda^n \ar[r] & 0;
}
\]
see, for example, \cite[Lemma II.4.2]{Erd}.
This induces the following commutative diagram
\[
\begin{tikzcd}
    & &  0 \arrow[d] & 0 \arrow[d] \\
    0 \arrow[r] 
      & V_\lambda^1 \arrow[r] \arrow[d, equal]
      & V_\lambda^3 \arrow[r] \arrow[d]
      & V_\lambda^2 \arrow[r] \arrow[d]
      & 0 \\
    0 \arrow[r]
      & V_\lambda^1 \arrow[r]
      & V_\lambda^4 \arrow[r] \arrow[d]
      & V_\lambda^3 \arrow[d] \arrow[r]
      & 0\\
    & & V_\lambda^1 \arrow[r, equal] \arrow[d] & V_\lambda^1 \arrow[d] \\
    & & 0 & 0,
\end{tikzcd}\]
where the first row is $\xi_2$, second row is $\xi_3$, third right-hand column is $\zeta_2$, and left-hand column is $\zeta_3$.  Note that the commutation on the top-right square comes from the sequence $\rho_3$.

Consider applying $G_\omega$ to $\xi_1=\zeta_1=\rho_1$, we have a non-split short exact sequence
\[
G(\xi_1):\quad 0\to M_\lambda(\omega) \to M_\lambda(\omega^2) \to M_\lambda(\omega)\to 0,
\]
which says that we can take $(\omega,\lambda)$ to be an $\epsilon$-factor of $(\omega^2,\lambda)$.

By applying $G_\omega$ to the commutative diagram of $\Bbbk[x,x^{-1}]$-modules above yields a commutative diagram
\[
\begin{tikzcd}
    & &  0 \arrow[d] & 0 \arrow[d] \\
    0 \arrow[r] 
      & M_\lambda(\omega) \arrow[r] \arrow[d, equal]
      & M_\lambda(\omega^3) \arrow[r] \arrow[d]
      & M_\lambda(\omega^2) \arrow[r] \arrow[d]
      & 0 \\
    0 \arrow[r]
      & M_\lambda(\omega) \arrow[r]
      & M_\lambda(\omega^4) \arrow[r] \arrow[d]
      & M_\lambda(\omega^3) \arrow[d] \arrow[r]
      & 0\\
    & & M_\lambda(\omega) \arrow[r, equal] \arrow[d] & M_\lambda(\omega) \arrow[d] \\
    & & 0 & 0
\end{tikzcd}\]
with all rows and columns exact, and the commutation of the top-right square comes from $G_\omega(\rho_3)$. Therefore, we can lift this diagram to an $\epsilon$-extension diagram (as in Definition \ref{def:e-ext}) with $C_f = (\omega^3,\lambda)\oplus t$, $D^f=(\omega^3,\lambda)\oplus t'$, $E = (\omega^6,\lambda)\oplus t''$ for some (possibly zero) $t,t',t''\in \add(\wti{T})$, we have that $(\omega^2,\lambda)$ is not $\epsilon$-rigid.
\end{proof}

\begin{example}
Let $J$ be the Kronecker algebra and $\omega$ be the unique closed curve of the annulus.
Then $M_\lambda(\omega^n) = \xymatrix@C=30pt{U_n \ar@<0.7ex>[r]^{1}\ar@<-0.3ex>[r]_{J_n(\lambda)} & V_n }$, where $U_n = \langle u_1, \dots, u_n\rangle$ and $V_n = \langle v_1, \dots, v_n\rangle$ over $\Bbbk$. We demonstrate that $M_\lambda(\omega^2)$ is not $\varepsilon$-rigid by showing the maps involved in the commutative diagram. 
The first row $0\to M_\lambda(\omega) \to M_\lambda(\omega^3) \to M_\lambda(\omega^{2})\to 0$ is given by
\[
\xymatrix@C=50pt{
0 \ar[r]\ar@<0.7ex>[d]\ar@<-0.7ex>[d] & U_1 \ar[r]^{{\begin{bsmallmatrix}0\\0\\1\end{bsmallmatrix}}} \ar@<0.7ex>[d]^{\lambda}\ar@<-0.7ex>[d]_{1} & U_3 \ar[r]^{{\begin{bsmallmatrix}1&0&0\\0&1&0\end{bsmallmatrix}}} \ar@<0.7ex>[d]^{J_3(\lambda)}\ar@<-0.7ex>[d]_{1} & U_2\ar@<0.7ex>[d]^{J_2(\lambda)}\ar@<-0.7ex>[d]_{1} \ar[r] & 0\ar@<0.7ex>[d]\ar@<-0.7ex>[d] \\
0 \ar[r] & V_1 \ar[r]_{{\begin{bsmallmatrix}0\\0\\1\end{bsmallmatrix}}} & V_3 \ar[r]_{{\begin{bsmallmatrix}1&0&0\\0&1&0\end{bsmallmatrix}}} & V_2 \ar[r] & 0. \\
}
\]
The second row is similar with the one-column matrices having an extra zero entry on top and the 2-by-3 matrices is enlarged to 3-by-4 matrices of the form $[I,0]$ for a 3x3 identity matrix $I$.

The rightmost column of the commutative diagram is given by the short exact sequence $0\to M_\lambda(\omega^2) \to M_\lambda(\omega^3) \to M_\lambda(\omega)\to 0$ which can be written explicitly in the following form: 
\[
\xymatrix@C=50pt{
0 \ar[r]\ar@<0.7ex>[d]\ar@<-0.7ex>[d] & U_2 \ar[r]^{{\begin{bsmallmatrix}0 & 0\\1&0\\0&1\end{bsmallmatrix}}} \ar@<0.7ex>[d]^{J_2(\lambda)}\ar@<-0.7ex>[d]_{1} & U_3 \ar[r]^{{\begin{bsmallmatrix}1 & 0 & 0\end{bsmallmatrix}}} \ar@<0.7ex>[d]^{J_3(\lambda)}\ar@<-0.7ex>[d]_{1} & U_1\ar@<0.7ex>[d]^{\lambda}\ar@<-0.7ex>[d]_{1} \ar[r] & 0\ar@<0.7ex>[d]\ar@<-0.7ex>[d] \\
0 \ar[r] & V_2 \ar[r]_{{\begin{bsmallmatrix}0 & 0\\1&0\\0&1\end{bsmallmatrix}}} & V_3 \ar[r]_{{\begin{bsmallmatrix}1 & 0 & 0\end{bsmallmatrix}}} & V_1 \ar[r] & 0. \\
}
\]
The left column can be described similarly by enlarging the matrices similar to the previous paragraph.
Then one can check the commutative of the symmetric extension diagram explicitly.
\end{example}

\begin{definition}\label{def:almost crossing}
An (oriented) \emph{almost crossing from $\gamma$ to $\delta$} is an overlap (underlined and encapsulated by the $|$-separators) of the form 
\[
\crosswn{\gamma}{[\gamma_L a_L^{-}]}{\kappa}{[a_R \gamma_R]}{\delta'}{[\delta_L b_L]}{[b_R^{-} \delta_R]}
\] 
where $\delta'\in\{\delta,\delta^-\}$, each of the bracketed parts can possibly be empty, $\kappa$ can be a trivial string, and $a_L, a_R, b_L, b_R$ are arrows. 
We denote by $ac(\gamma, \delta)$ the set of almost positive crossings from $\gamma$ to $\delta$. 
\end{definition}

If all $a_L, a_R, b_L, b_R$ exists, then we get a genuine topological crossing between the corresponding curves; see Figure \ref{fig:crossing}.
If, for example, $a_L$ (hence $\gamma_L$ as well) does not exist in the almost crossing, then we should modify Figure \ref{fig:crossing} so that the curve $\gamma$ starts from the marked point $m$ instead.
In particular, if both $a_L$ and $b_L$ do not exist, then we get only an intersection of the curves at the marked point $m$.

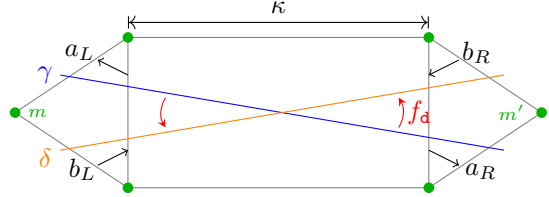
\begin{figure}[!htbp]
\centering
\begin{tikzpicture}
\draw[gray]  (0,2) rectangle (4,0);
\draw[gray] (0,2) -- (-1.5,1) -- (0,0);
\draw[gray] (4,2) -- (5.5,1) -- (4,0);
\fill[darkgreen] (-1.5,1) circle (2pt);\fill[darkgreen] (0,2) circle (2pt);\fill[darkgreen] (0,0) circle (2pt);
\fill[darkgreen] (5.5,1) circle (2pt);\fill[darkgreen] (4,0) circle (2pt);\fill[darkgreen] (4,2) circle (2pt);
\draw[blue] (-0.9,1.5) --(5,0.5);
\node[blue] at (-1.1,1.5) {$\gamma$};
\draw[orange] (-0.9,0.5) --(5,1.5);
\node[orange] at (-1.1,0.4) {$\delta$};
\node[darkgreen] at (-1.2,1) {\scriptsize $m$};
\node[darkgreen] at (5.1,1) {\scriptsize $m'$};
\draw[|<->|] (0,2.2) -- node[midway,above]{$\kappa$} (4,2.2);

\draw[-to] (0,1.5) -- node[left,pos=0.8,yshift={4}] {$a_L$} (-0.4,1.7);
\draw[-to] (-0.4,0.3) -- node[left,pos=0.25,yshift={-3}] {$b_L$} (0,0.5);

\draw[-to] (4,0.5) -- node[right,pos=0.9,yshift={-3}] {$a_R$} (4.4,0.3);
\draw[-to] (4.4,1.7) -- node[right,pos=0.2,yshift={4}] {$b_R$} (4,1.5);

\draw[-stealth,red] (3.6,0.8) to [bend right] node[midway,right,xshift={-2}]{$f_{\mathtt{d}}$} (3.6,1.2);
\draw[-stealth,red] (0.5,1.2) to [bend right] (0.5,0.8);
\end{tikzpicture}\caption{An oriented (almost) crossing $\mt{d}$ from $\gamma$ to $\delta$}\label{fig:crossing}
\end{figure}

We remark that not every intersection (up to isotopy) between $\gamma,\delta$ is encoded in $ac(\gamma,\delta)\cup ac(\delta,\gamma)$; the missing ones are called `crossing in a an arrow' and `crossing in a 3-cycle' in \cite{CS17}.

\subsubsection{Homomorphisms between string modules}

Let $\{x_i\}_{i=0,1,\ldots,c}$ be the canonical basis of $M(\gamma)$, i.e. $x_i=\gamma(t_i)$ in the notation of Section \ref{sec:cluster cat basic}.

Likewise, let $\{y_j\}_{j\in J}$ (for some finite indexing set $J$) be the basis elements of $M(\delta$).

If there is an overlap $\kappa=\gamma_a\cdots \gamma_{a+\ell}=\delta_b\cdots \delta_{b+\ell}$ between the (string form of) $\gamma$ and $\delta$, then we have $x_{a-1},y_{b-1}\in s(\gamma_a)=s(\delta_b), x_{a+i},y_{b+i}\in t(\gamma_{a+i})=t(\delta_{b+i})$.
If $\kappa$ is the overlap specified by $\mt{d}\in ac(\gamma,\delta)$, we write $x_i\sim_{\mt d}y_j$ when $x_i = x_{a+k}$ and $y_j = y_{b+k}$ for some $k\in \{-1,0,\ldots,\ell\}$.  In this case, we have a module homomorphism
\begin{align*}
f_{\mt d}: M(\gamma) &\to M(\delta)\\
x_i & \mapsto  \begin{cases}
y_{j} & \text{ if $x_i\sim_{\mt d} y_j$;}\\
0  & \text{ otherwise.}
\end{cases} 
\end{align*}
We can regard $f_{\mt d}$ as a counter-clockwise turning of subinterval $\omega$ of the curve $\gamma$; see Figure \ref{fig:crossing}.

The following result of Crawley-Boevey says that $f_{\mt d}$ defines a \dfn{canonical basis} of Hom-space between string modules.

\begin{theorem}{\rm\cite{CB89}} \label{thm:Crawley Boevey string map} 
For two (possibly the same) strings $\gamma$ and $\delta$, the set $\{f_{\mt{d}} \mid \mt{d} \in ac(\gamma, \delta)\}$ is a basis of  $\Hom_{J}(M(\gamma), M(\delta))$.
\end{theorem}

\subsubsection{ Homomorphism between a string and a band module}

A similar result exists for Hom-spaces involving band modules.  In the scope of this article, we only need maps involving primitive band $\omega$, where the associated Hom-spaces have much simpler descriptions than the non-primitive ones.  We refer the interested reader to the original paper of Krause \cite{Kra91}, or to \cite{CD20} for which the following formulation is based on.

Suppose $\omega=\omega_1\omega_2\cdots \omega_c$ is a primitive band.
Let $\{x_i\}_{i=0,2,\ldots,c-1}$ be the basis of $M_\lambda(\omega)$ where $x_i:=\omega(t_{i+1})^{(1)} \in s(\omega_{i+1})$ in the notation of Section \ref{sec:cluster cat basic}.
We repeatedly self-concatenate $\omega$ to form a bi-infinite string $\biinf{\omega} := \cdots \omega\omega\omega\cdots $.
This yields an infinite-dimension string module $M(\biinf{\omega})=\bigoplus_{i\in \Z}\Bbbk \hat{x}_i$ and a surjection
\begin{align*}
\pi_\omega: M(\biinf{\omega}) & \to  M_\lambda(\omega)\\
\hat{x}_{i-cm} & \mapsto \lambda^{m}x_i \quad \text{
for all $0\leq i< c$ and $m\in \Z$.}
\end{align*}

The map $p:\omega_i\mapsto \omega_{i-c}$ defines a $\langle p\rangle\cong\Z$-action called \dfn{period-shifting} on $\biinf{\omega}$.
The definition of almost crossing extends naturally to bi-infinite string, in which case, the $\Z$-action on $\biinf{\omega}$ induces a  $\Z$-action on $ac(\biinf{\omega},\delta)$ and on $ac(\delta,\biinf{\omega})$ for any (finite) string $\delta$.  

As before, we take $\{y_j\}_{j\in J}$ to be the basis of $M(\delta)$ for a (finite) string $\delta$.
For an almost crossings $\mt d\in ac(\delta,\biinf{\omega})$, we have a homomorphism of modules
\begin{align*}
g_{\mt d}^{\mathrm{bs}}: M(\delta)&\to M_\lambda(\omega)\\
y_j & \mapsto \pi_\omega f_{\mt d}(y_j)=\lambda^m x_i \text{ where } y_j\sim_{\mt d}\hat{x}_{i-cm} \text{ with } 1\leq i\leq c \text{ and }m\in\Z.
\end{align*}
We can write the $g_{\mt d}^{\mathrm{bs}}$ in matrix form with respect to the canonical bases:
\begin{align}\label{eq:g-bs matrix}
g_{\mt d}^{\mathrm{bs}}=\begin{pmatrix}
0 & \lambda^d f_{\mt{d},0} & \lambda^{d-1} f_{\mt{d},1} & \cdots & \lambda^{d-l} f_{\mt{d},l} & 0
\end{pmatrix}
\end{align}
where each $f_{\mt{d},m}$ for $0\leq m\leq l$ is (the $K$-linear map corresponding to) a $c$-by-$c$ matrix and $f_{\mt{d},m}$ for $0<m<l$ are the identity matrices.  Here the exponent $a$ is determined by the smallest index $i-dc$ with $0\leq i<c$ and $x_{i-dc}\in \Image(g_{\mt d}^{\mathrm{bs}})$.

One checks that $g_{p(\mt d)}^{\mathrm{bs}}=\lambda g_{\mt d}^{\mathrm{bs}}$ for all $k\in \Z$.  This means that the $\Z$-orbit $[\mt{d}]\in ac(\delta,\biinf{\omega})/\Z$ determines the same map up to scalar multiple.  Define $g_{[\mt{d}]}^{\mathrm{bs}}:=g_{\mt{d'}}^{\mathrm{bs}}$ where $\mt{d'}$ is determined by $y_j\sim_{\mt{d'}} x_i$ for some $0\leq i< c$. 


Dually, for an almost crossings $\mt d\in ac(\biinf{\omega},\delta)$, we have a homomorphism $g_{\mt d}^{\mathrm{sb}}: M_\lambda(\omega) \to M(\delta)$ of modules so that $g_{\mt d}^{\mathrm{sb}}\pi_\omega = \sum_{m\in \Z}\lambda^m f_{p^m(\mt d)}$, i.e. explicitly $g_{\mt d}^{\mathrm{sb}}$ is given by
\begin{align*}
g_{\mt d}^{\mathrm{sb}}: M_\lambda(\omega) &\to M(\delta)\\
x_i & \mapsto  \sum_{m\in \Z} \lambda^m f_{\mt d}(\hat{x}_{i+cm}) = \sum_{\substack{m\in \Z \\ \hat{x}_{i+cm}\sim_{\mt d}y_j}} \lambda^m y_j.
\end{align*}
Note that this is a finite sum as the overlap in $\mt{d}$ is of finite length.
We can also write $g_{\mt d}^{\mathrm{sb}}$ in matrix form:
\[
g_{\mt d}^{\mathrm{sb}}=\begin{pmatrix}
0 & \lambda^d f_{\mt{d},0} & \lambda^{d+1} f_{\mt{d},1} & \cdots & \lambda^{d+l} f_{\mt{d},l} & 0
\end{pmatrix}^\top,
\]
where each $f_{\mt{d},m}$ for $0\leq m\leq l$ are $c$-by-$c$ matrix and $f_{\mt{d},m}$ for $0<m<l$ are the identity matrices.

It is straightforward to check that $g_{p(\mt d)}^{\mathrm{sb}}=\lambda^{-1} g_{\mt d}^{\mathrm{sb}}$, so for a $\Z$-orbit $[\mt{d}]\in ac(\biinf{\omega},\delta)/\Z$, define $g_{[\mt{d}]}^{\mathrm{sb}}:=g_{\mt{d}' }^{\mathrm{sb}}$ where $\mt{d}'$ is determined by $y_j\sim_{\mt{d'}} x_i$ for some $0\leq i< c$.  

\begin{theorem}{\rm\cite{Kra91}} \label{thm:Krause band map} 
For a string $\delta$, a band $\omega$, and $\lambda\in\Bbbk^\times$, we have
\begin{enumerate}[(a)]
\item $\{g_{[\mt{d}]}^{\mathrm{sb}} \mid [\mt{d}] \in ac(\biinf{\omega}, \delta) / \Z\}$ is a basis of $\Hom_J(M_\lambda(\omega), M(\delta))$;

\item $\{g_{[\mt{d}]}^{\mathrm{bs}} \mid [\mt{d}] \in ac(\delta, \biinf{\omega}) / \Z\}$ is a basis of $\Hom_J(M(\delta), M_\lambda(\omega))$.
\end{enumerate}
\end{theorem}

\subsubsection{Composition of homomorphism}

Suppose we have three strings $\gamma,\delta,\eta$ and two almost crossings $\mt{d}\in ac(\gamma,\delta), \mt{e}\in ac(\delta,\eta)$.  If the intersection of the overlaps of $\mt{d}$ and $\mt{e}$ is non-empty, then we have an almost crossing $\mt{e}\cdot\mt{d}\in ac(\gamma,\delta)$; otherwise, we define $\mt{e}\cdot\mt{d}:=\emptyset$.  It is straightforward from the definition of the canonical maps that $f_{\mt e}f_{\mt d}=f_{\mt{e}\cdot\mt{d}}$, where $f_{\emptyset}:=0$.

Consider $g_{\mt{e}}^{\mathrm{sb}}:M_\lambda(\omega)\to M(\delta)$ and $g_{\mt{d}}^{\mathrm{bs}}:M(\gamma)\to M_\lambda(\omega)$.  Then we have
\begin{align}\label{eq:compose band string maps}
g_{\mt{e}}^{\mathrm{sb}}g_{\mt{d}}^{\mathrm{bs}} = g_{\mt{e}}^{\mathrm{sb}} \pi_\omega f_{\mt{d}} = \big(\sum_{m\in \Z}\lambda^m f_{p^m(\mt{e})}\big)f_{\mt{d}} = \sum_{m\in \Z} \lambda^m f_{p^m(\mt{e})\cdot\mt{d}}.
\end{align}

\subsubsection{Interaction with duality}

We explain the effect of applying the duality $\nabla$ to the canonical maps. 
We will use the following convention.  For a curve $\gamma$ with string form $\gamma = a_1^{\epsilon_1}a_2^{\epsilon_2}\cdots a_\ell^{\epsilon_\ell}$ where $a_i$'s are all arrows and $\epsilon_i\in\{+,-\}$ for all $i$, we define
\[
\nabla(\gamma):=\sigma(a_1)^{-\epsilon_1}\cdots \sigma(a_\ell)^{-\epsilon_\ell}.
\]
 
Let us first describe the isomorphism from $\nabla(M)$ of an indecomposable string-or-band module $M$ to the canonical basis of the resulting string-or-band module.  For a string module $M(\gamma)$, we have defined the combinatorial operation $\nabla$ on the string form of $\gamma$ in a way so that there is an isomorphism $\psi_{M(\gamma)}:\nabla(M(\gamma))\xrightarrow{\sim} M(\nabla(\gamma))$.
Recall from subsection \ref{subsec:string intro} that the canonical basis of $M(\gamma)$ is given by $\{u_i:=\gamma(t_i)\}_{0\leq i\leq c}$.  Therefore, the module $\nabla(M(\gamma))$ has a natural basis $\{u_i^*\}_{0\leq i\leq c}$ dual to $\{u_i\}_i$.  On the other hand, by design of $\nabla$ (on strings), the canonical basis of $M(\nabla(\gamma))$ is $\{v_i:=\sigma(\gamma)(t_i)\}$.  Under this setup, we have
\[
\psi_{M(\gamma)}: \nabla(M(\gamma))\xrightarrow{\sim} M(\nabla(\gamma)) \text{ given by } u_i^* \mapsto (-1)^iv_i\;\;\forall 0\leq i\leq c.
\]
Let $E(\gamma):=M(\gamma)\oplus M(\nabla(\gamma))$, then we have an isomorphism
\[
\psi_{E(\gamma)} = \begin{bmatrix} 0 &\psi_{M(\gamma)} \\ \epsilon \psi_{M(\nabla(\gamma))} & 0\end{bmatrix} : \nabla(M(\gamma))\oplus \nabla(M(\nabla(\gamma))) \xrightarrow{\sim} M(\gamma)\oplus M(\nabla(\gamma)).
\]
This isomorphism is associated to the bilinear map defining the $\epsilon$-representation structure on $E(\gamma)$ (see discussion in \cite[Sec 2.3]{BCI21}), that is, $\nabla(\psi_{E(\gamma)}) = \psi_{E(\gamma)}$.

For a 1-sided indecomposable $\epsilon$-representation $M_\lambda(\omega)$ where $\omega=\sigma(\omega)$ is a primitive closed curve (note that $\lambda=(-1)^{\mathrm{len}(\omega)/2}\epsilon$ by Lemma \ref{lem:1-scc e-indec}), we have a similar isomorphism $\psi_{M_\lambda(\omega)}: \nabla(M_\lambda(\omega))\to M_{\lambda^{-1}}(\omega)$ that satisfies $\nabla(\psi_{M_\lambda(\omega)})=\epsilon \psi_{M_\lambda(\omega)}$ given as follows.

\begin{lemma}\label{lem:self-dual band e-isom}
Suppose that $M_\lambda(\omega)$ is a 1-sided indecomposable $\epsilon$-representation for some primitive $\omega=\sigma(\omega)$.
Let $\{u_i\}_{0\leq i< 2r}$ and $\{v_i\}_{0\leq i< 2r}$ be the canonical basis of $M_{\lambda^{-1}}(\omega)$ and $M_\lambda(\omega)$ respectively, i.e. $u_i, v_i$ are symbols given by $\omega(z_i)$ in the notation of subsection \ref{subsec:band intro}.
Let $\{v_i^*\}_{0\leq i< c=2r}$ be the standard basis of $\nabla(M_\lambda(\omega))$ (whose underlying space is the $\Bbbk$-linear dual of $M_\lambda(\omega)$) dual to the canonical basis $\{v_i\}_{0\leq i< c=2r}$.
Then we have an isomorphism 
\[\psi_{M_\lambda(\omega)}: \nabla(M_\lambda(\omega)) \xrightarrow{\sim} M_{\lambda^{-1}}(\omega) \text{ given by } \begin{cases}
v_i^*\mapsto (-1)^{i}u_{r+i}\\
v_{r+i}^*\mapsto (-1)^{i}\lambda u_i
\end{cases}\text{ for all $0\leq i< r$}\]
such that $\nabla(\psi_{M_\lambda(\omega)})=\epsilon\psi_{M_\lambda(\omega)}$.
\end{lemma}
\begin{proof}
The isomorphism $\psi_{M_\lambda(\omega)}$ is the one naturally induced by the bilinear form defining the $\epsilon$-representation on $M_\lambda(\omega)$, using the data from the proof of Lemma \ref{lem:1-scc e-indec}.  The fact that it satisfies $\nabla(\psi_{M_\lambda(\omega)})=\epsilon \psi_{M_\lambda(\omega)}$ can be found in \cite[Sec 2.3]{BCI21}.
\end{proof}

For any almost crossing $\mt{d}\in ac(\gamma,\eta)$, say
\[
\crosswn{\gamma}{[\gamma_L a_L^{-}]}{\kappa}{[a_R \gamma_R]}{\eta^\pm}{[\eta_L b_L]}{[b_R^{-} \eta_R]},
\]
its dual $\nabla(\mt{d})\in ac(\nabla(\gamma),\nabla(\eta))$ is just the natural
\[
\crosswn{\nabla(\eta^\pm)}{[\nabla(\eta_L) \sigma(b_L)^-]}{\nabla(\kappa)}{[\sigma(b_R) \nabla(\eta_R)]}{\nabla(\gamma)}{[\nabla(\gamma_L) \sigma(a_L)]}{[\sigma(a_R)^- \nabla(\gamma_R)]}.
\]
In the case when $\eta=\biinf{\omega}=\cdots \omega_i\omega_{i+1}\cdots$ for a $\sigma$-stable closed curve $\sigma(\omega)=\omega$, we identify $\nabla(\biinf{\omega})$ with the `half-period shift' of $\biinf{\omega}$, i.e. if $\kappa\subset\biinf{\omega}$ starts at the $i$-th letter, then $\nabla(\kappa)\subset \nabla(\biinf{\omega})$ starts with the $(\mathrm{len}(\omega)/2+i)$-th letter.

Fix a $\mt{d}\in ac(\gamma,\eta)$ as above.
Let $a,b,\ell$ be integers such that 
\[
\kappa = \gamma_{a+1}\gamma_{a+2}\cdots \gamma_{a+\ell} = \eta_{b+1}\eta_{b+2}\cdots \eta_{b+\ell}.
\]
This means that, if we let $\{x_i\}_{i}$ and $\{y_j\}_j$ be the canonical basis of the indecomposable modules of $\gamma,\eta$ respectively, then the canonical map sends $x_i$ for $a\leq i\leq a+\ell$ to a linear combination of $y_j$ with $b\leq j\leq b+\ell$; here the indices are taken modulo $c:=\mathrm{len}(\omega)$ when any one of $\gamma,\eta$ is $\biinf{\omega}$.

\begin{proposition}\label{prop:dual map}
For strings $\gamma,\delta$, and primitive band whose underlying closed curve satisfies $\omega=\sigma(\omega)$, the following hold.

\begin{enumerate}[\rm(a)]
\item For any $\mt{d}\in ac(\gamma,\delta)$, we have $\nabla(f_{\mt d}) = (-1)^{a+b}\psi^{-1}_{M(\delta)}f_{\nabla(\mt{d})}\psi_{M(\gamma)}$.

\item For any $\mt{d}\in ac(\gamma, \biinf{\omega})$, we have $\nabla(g_{\mt d}^{\mathrm{bs}}) = (-1)^{a+b} \psi_{M(\delta)}^{-1}g_{\nabla(\mt{d})}^{\mathrm{sb}}\psi_{M_\lambda(\omega)}$.

\item For any $\mt{d}\in ac(\biinf{\omega}, \delta)$, we have $\nabla(g_{\mt d}^{\mathrm{sb}}) = (-1)^{a+b} \psi_{M_\lambda(\omega)}^{-1}g_{\nabla(\mt{d})}^{\mathrm{bs}}\psi_{M(\delta)}$.
\end{enumerate}
\end{proposition}
\begin{proof}
(a) This is straightforward to check from the definition of the $\psi$ maps and $\nabla(\mt{d})$.

(b) Consider $g_{\mt d}^{\mathrm{bs}}$ in its matrix form as in \eqref{eq:g-bs matrix} but with each submatrix $f_{\mt{d},i}$ decomposed further into a 2-by-2 block $\left(\begin{smallmatrix}
f_{\mt{d},i}^L & 0 \\ 0 & f_{\mt{d},i}^R\end{smallmatrix}\right)$, each of size $r$-by-$r$ for $r:=\mathrm{len}(\omega)/2$, that is, 
\[
g_{\mt d}^{\mathrm{bs}} = \begin{pmatrix}
0&  \lambda^d f_{\mt{d},0}^L & 0 &\lambda^{d-1} f_{\mt{d},1}^L & 0 & \cdots & \lambda^{d-l} f_{\mt{d},l}^L & 0 & 0  \\
0&0 & \lambda^d f_{\mt{d},0}^R & 0 & \lambda^{d-1} f_{\mt{d},1}^R & \cdots & 0& \lambda^{d-l} f_{\mt{d},l}^R & 0
\end{pmatrix}.
\]
Since the overlap of $\nabla(\mt{d})$ starts $r$ places later than that of $\mt{d}$ in $\biinf{\omega}$, the matrix form of $g_{\nabla(\mt{d})}^{\mathrm{sb}}$ will be of the form:
\[
g_{\nabla(\mt{d})}^{\mathrm{sb}} = \begin{pmatrix}
0 & 0\\
0 & \lambda^d f_{\nabla(\mt{d})_R,0} \\
\lambda^{d-1} f_{\nabla(\mt{d})_L,1} & 0\\
0 & \lambda^{d-1} f_{\nabla(\mt{d})_R,1} \\
\vdots & \vdots \\
0& \lambda^{d-l} f_{\nabla(\mt{d})_R,l} \\
\lambda^{d-l-1} f_{\nabla(\mt{d})_R,l+1} & 0\\
0 & 0
\end{pmatrix}.
\]
Note that here, for each $0\leq i\leq l$, $f_{\nabla(\mt{d}),i}^R$ is the transpose of $f_{\mt{d},i}^L$  and $f_{\nabla(\mt{d}),i+1}^L$ is the transpose of $f_{\nabla(\mt{d}),i}^RX$.  Also, the domain of $g_{\mt{d}}^{\mathrm{sb}}$ is $M_{\lambda^{-1}}(\omega)$, and so the exponent on $\lambda$ decreases as we go down the rows.

Now we compare the above matrix with $\psi_{M(\gamma)} \nabla(g_{\mt{d}}^{\mathrm{bs}})\psi_{M_\lambda(\omega)}^{-1}$.
As maps of vector spaces, $\nabla(g_{\mt{d}})^{\mathrm{bs}})$ is just $\Bbbk$-linear dual of $g_{\mt{d}}^{\mathrm{bs}}$, so the corresponding matrix is just taking the transpose of $g_{\mt{d}}^{\mathrm{bs}}$.  By the definition of $\psi_{M_\lambda(\omega)}^{-1}$, it swaps the two (block-)columns of the tranposed matrix, and then multiply $\lambda^{-1}$ to the first column, and multiply a further $(-1)^b$ to the resulting matrix.  The effect of $\psi_{M(\gamma)}$ further multiplies a factor of $(-1)^a$ but does not permute any entries of the matrix.  This results in the same matrix as $(-1)^{a+b}g_{\nabla(\mt{d})}^{\mathrm{sb}}$, as claimed.

(c) Follows from (b) by applying $\nabla$ on both sides.
\end{proof}

\begin{corollary}\label{cor:compose g with dual}
For any $\mt{d}\in ac(\biinf{\omega},\delta)$, we have $g_{[\mt{d}]}\psi_{M_\lambda(\omega)}\nabla(g_{[\mt {d}]}^{\mathrm{bs}})\neq 0$ if $p^m(\mt{d})\cdot\nabla(\mt{d})\neq \emptyset$ for some $m\in\Z$.
\end{corollary}
\begin{proof}
This follows from combining Proposition \ref{prop:dual map} with \eqref{eq:compose band string maps}.
\end{proof}

\begin{corollary}\label{lem:nabla composition commute}
For any strings $\gamma,\delta,\eta$, and almost crossings $\mt{d}\in ac(\gamma,\delta), \mt{e}\in ac(\delta,\eta)$, we have $\nabla(\mt{e}\cdot\mt{d})=\nabla\mt{d} \cdot \nabla\mt{e}$.
In particular, we have \[\nabla(f_{\mt{e}}\circ f_{\mt{d}})=\pm \psi_{M(\eta)}^{-1} f_{\nabla\mt{e}} f_{\nabla\mt{d}} \psi_{M(\gamma)}.\]
\end{corollary}
\begin{proof}
First part is straightforward from definition of composition and $\nabla$-operation on almost crossings.  The second part then follows by applying Proposition \ref{prop:dual map} (a).
\end{proof}

\subsection{\texorpdfstring{$\epsilon$}{e}-rigidity for 1-sided \texorpdfstring{$\epsilon$}{e}-indecomposables}

\begin{lemma}\label{lem:1-sided e-rigid}
Let $\omega=\sigma(\omega)$ be a primitive closed curve and $(\omega,\lambda)$ a 1-sided $\epsilon$-indecomposable object.
If $\omega$ is simple, then is $\epsilon$-rigid.
\end{lemma}
\begin{proof}
We need to show that every non-zero $f\in \Ext_\C^1((\omega,\lambda),\rho)$ has $f\psi_{(\omega,\lambda)}\nabla(f)$ for all $\epsilon$-factor $\rho$ of $(\omega,\lambda)$.  Note that by Proposition \ref{prop::Homs} and $\tau$-invariance of band modules, $\Ext_\C^1((\omega,\lambda),\rho) \cong \Hom_J(M_\lambda(\omega), \tau M(\rho)\oplus \Hom_J(M(\rho),M_\lambda(\omega))$ as vector space.

Write the underlying walk of $\omega$ as $\omega_1\cdots \omega_r\cdots \omega_{2r}$ and $\omega_i=\alpha_i^{\epsilon_i}$ with $\alpha_i\in Q_1$ and $\epsilon_i\in\{\pm1\}$ for all $1\leq i\leq 2r$.
For the indices of letters or arrows in $\omega$, we will also use arithmetic modulo $2r$ taking values in $\{1,\ldots, 2r\}$.
Recall from Lemma \ref{lem:sigma rotates 1-scc} that $\sigma(\alpha_i)=\alpha_{i+r}$ for all $1\leq i\leq 2r$.
Let us fix an $i\in\{1,\ldots,r\}$ so that $\rho = \omega_{i+1}\omega_{i+2}\cdots \omega_{i+r-1}$.  
The form of the curves $\rho, \sigma(\rho), \omega$ are shown in Figure \ref{fig:e-factor curves}, where the dashed lines are supposed to be identified, the shaded parts represent the `outside' of $\partial\S$, and two triangles are part of the triangulation.

\begin{figure}[!htbp]
\centering
\begin{tikzpicture}[
	bdry/.style={postaction={draw,decorate,decoration={border,angle=-45,amplitude=2.5mm,segment length=1.2mm}}}
	]
\draw[dashed,->] (0,-2) -- (0,2);\draw[dashed,->] (12,-2) -- (12,2);
\draw[blue] (0,0) -- node[pos=.1,above]{$\omega$} (12,0);
\draw[magenta] (0,0) -- node[pos=.2,below left]{$\sigma(\rho)$}  (4,-2) (12,0) -- (8,2);
\draw[orange] (2,-2) --  node[pos=.35,below right]{$\rho$} (10,2);
\draw[red] (1.4,-1.4) -- (10.6,1.4);

\draw[bdry] (1.2,-1.2) -- +(1.2,-1.2);
\draw[bdry] (3.5,-2.5) -- (4.5,-1.5);
\draw[gray] (2,-2) -- (4,-2) -- (3,2) -- cycle;\mkpt{2,-2}\mkpt{4,-2}\mkpt{3,2}
\draw (3.2,1) edge[->,bend left] node[below]{$\sigma(\alpha)$} (2.8,1);
\mkpt{1.4,-1.4}
\draw[-latex] (2,-1.8) -- node[right,pos=0.7]{$\tau$} +(-0.4, 0.4);
\begin{scope}[scale=-1,xshift=-12cm]
\draw[bdry] (1.2,-1.2) -- +(1.2,-1.2);
\draw[bdry] (3.5,-2.5) -- (4.5,-1.5);
\draw[gray] (2,-2) -- (4,-2) -- (3,2) -- cycle;\mkpt{2,-2}\mkpt{4,-2}\mkpt{3,2}
\draw (3.2,1) edge[->,bend left] node[midway,above]{$\alpha$} (2.8,1);
\mkpt{1.4,-1.4}
\draw[-latex] (2,-1.8) -- node[left,pos=0.7]{$\tau$} +(-0.4, 0.4);
\end{scope}
\end{tikzpicture}\caption{The curves $\omega, \rho, \sigma(\rho), \tau(\rho)$ on $\SM$}\label{fig:e-factor curves}
\end{figure}
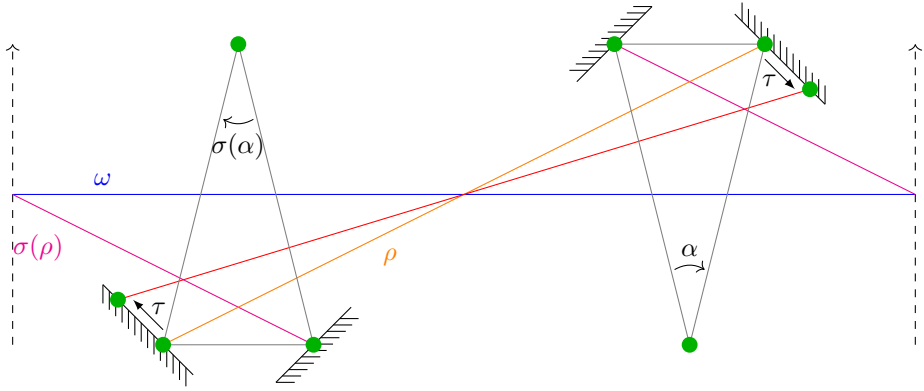

The first thing we claim is that $\dim_{\Bbbk} \Ext_\C^1((\omega,\lambda),\rho) = 1$.  It is a folklore that dimension of $\Ext_\C^1$ counts geometric intersection number (see \cite{CS17} for related result for the case of intersections between two curves with endpoints), so the claim can be easily seen from picture (Figure \ref{fig:doublecover}).  Since the exact statement is not shown in the literature, we will do it explicitly here.

\textbf{Claim 1}: $\dim_{\Bbbk}\Hom_J(M(\rho), M_\lambda(\omega))=0$ and $\dim_{\Bbbk}\Hom_J( M_\lambda(\omega), \tau M(\rho))=1$.

\textit{Proof of Claim}: 
By Theorem \ref{thm:Krause band map}, it suffices to show $ac(\rho,\biinf{\omega}) =\emptyset$ and $|ac(\biinf{\omega},\tau(\rho))/\Z|=1$, where $\tau(\rho)$ is the string so that $\tau M(\rho)\cong M(\tau(\rho))$.

Since $\alpha_{i}^-\rho\alpha_{r+i}$ is a subwalk of $\biinf{\omega}$, any almost crossing from $\rho$ to $\biinf{\omega}$ will induce an almost crossing from $\biinf{\omega}$ to itself, which then means that the closed curve $\omega$ has a self-crossing.  Thus, the set $ac(\rho,\biinf{\omega})$ is empty.

For $ac(\biinf{\omega}, \tau(\rho))$, first recall from \cite[Sec 3]{BZ11} that $\tau(\rho)$ can be described by moving the endpoints of $\rho$ to the `next' marked point on the boundary; see Figure \ref{fig:e-factor curves}.  In terms of strings, $\tau(\rho) = p_L \alpha_i^{-}\rho\alpha_{r+i} p_R^{-}$ for some (maximal or trivial) paths $p_L,p_R$; see Figure \ref{fig:tau rho}.  Hence, there are some $j,k\in\{1,\ldots, 2r\}$ so that we have an almost crossing
\[
\mt{d}=\crosswn{\biinf{\omega}}{\omega_L \alpha_j^- }{p_L \alpha_i^-\rho \alpha_{i+r}p_R^-}{\alpha_{k}\omega_R}{\tau(\rho)}{\phantom{.}}{\phantom{.}}
\]
in $ac(\biinf{\omega},\tau(\delta))$.
This is the only almost crossing (up to $\Z$-action) as otherwise we will have a almost self-crossing in $\omega$, which contradicts the simplicity assumption. $\blacksquare$

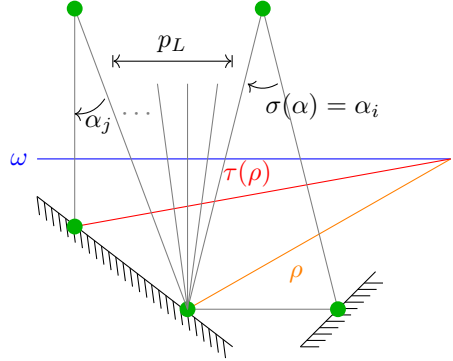
\begin{figure}[!htbp]
\centering
\begin{tikzpicture}[
	bdry/.style={postaction={draw,decorate,decoration={border,angle=-45,amplitude=2.5mm,segment length=1.2mm}}}
	]
\draw[blue] (0,0) -- node[pos=0,left]{$\omega$} (5.5,0);
\draw[orange] (2,-2) --  node[pos=0.35,below right]{$\rho$} (5.5,0);
\draw[red] (0.5,-0.9) -- node[pos=0.46,above]{$\tau(\rho)$} (5.5,0);
\draw[bdry] (0,-0.5) -- (2.6,-2.5);
\draw[bdry] (3.5,-2.5) -- (4.5,-1.5);
\draw[gray] (2,-2) -- (4,-2) -- (3,2) -- cycle;\mkpt{2,-2}\mkpt{4,-2}\mkpt{3,2}
\draw[gray] (2,-2) -- (2.4,1) (2,-2) -- (2,1) (2,-2) -- (1.6,1);\node[gray] at (1.35,0.6) {$\cdots$};
\draw[gray] (2,-2) -- (0.5,2) -- (0.5,-0.9);
\mkpt{0.5,2}\mkpt{0.5,-0.9}
\draw (3.2,1) edge[->,bend left]  (2.8,1);
\node at (3.8,0.7) {$\sigma(\alpha)=\alpha_{i}$};
\draw[|<->|] (1,1.3) -- node[midway,above] {$p_L$} (2.6,1.3);
\draw (0.9,0.8) edge[->,bend left] node[pos=0.25,below] {$\alpha_j$} (0.5,0.6);
\end{tikzpicture}\caption{Understanding the string form of $\tau(\rho)$}\label{fig:tau rho}
\end{figure}

By Claim 1 and Proposition \ref{prop::Homs}, $\Ext_\C^1((\omega,\lambda),\rho)$ is one-dimensional, and since the space of $\epsilon$-extensions is closed under scalar multiple, it is enough to show for an arbitrary non-zero morphism $f\in \Ext_\C^1((\omega,\lambda),\rho)$ satisfies $f\psi_{(\omega,\lambda)} \nabla(f)$.  Since there is an equivalence $M(-):\C/[\wti{T}] \xrightarrow{\sim} \mod J$, it suffices to show that the canonical basis element $g_{\mt d}\in \Hom_J(M_{\lambda}(\omega), \tau M(\rho))$ for $\mt{d}$ the unique (equivalence class of) almost crossing in $ac(\biinf{\omega},\tau(\rho))/\Z$ that $g_{\mt d}\psi_{M_\lambda(\omega)}\nabla(g_{\mt d})\neq 0$.  By Corollary \ref{cor:compose g with dual}, it suffices to show that for the unique almost crossing $\mt{d}\in ac(\biinf{\omega}, \tau(\rho))$, there is some $m\in \Z$ so that $\mt{d}\cdot \nabla(\mt{d})\neq \emptyset$.

If we index $\omega = (\omega_t)_{t\in \Z}$ so that $\omega_t = \alpha_{\overline{t}}^{\epsilon_{\overline{t}}}$, where $\overline{t}\in \{1,\ldots, 2r\}$ is $t$ modulo $2r$, then the effect of applying $\nabla$ to $\biinf{\omega}$ takes $\omega_t=\alpha_{\overline{t}}^{\epsilon_{\overline{t}}}$ to $\omega_{r+t} = \alpha_{\overline{r+t}}^{-\epsilon_{\overline{t}}}$.
Now the dual almost crossing is:
\[
\nabla(\mt{d})=\crosswn{\nabla(\tau(\rho))}{}{\nabla(p_L) \alpha_{r+i}\nabla(\rho) \alpha_{i}\nabla(p_R)}{}{\biinf{\omega}}{\nabla(\omega_L) \alpha_{r+j}}{\alpha_{r+k}^-\nabla(\omega_R).}
\]

Note that the $\alpha_i$ appearing in $\biinf{\omega}$ in the second line is $\omega_{2r+i}$, i.e. this $\alpha_i$ appears in the `next copy of $\omega$ in $\biinf{\omega}$' relative to the $\alpha_i$ in $\mt{d}$.
On the hand, $\alpha_{r+i}$ in both $\mt{d}$ and $\nabla(\mt{d})$ are the same letter in $\biinf{\omega}$ (namely, $\omega_{r+i}$).
Also, observe from Figure \ref{fig:tau rho} that $p_R^-$ is a prefix of $\nabla(\rho)$, and so $\nabla(\rho)=p_R^-\alpha_k q$ for some walk $q$.
Dually, we have $\rho=q'\alpha_{r+j}\nabla(p_L)$.

Now combine all the information we have
\[
\mt{d}\cdot\nabla(\mt{d})= \crosswn{\nabla(\tau(\rho))}{}{\nabla(p_L)\alpha_{r+i}p_R^-}{\alpha_k q \alpha_i\nabla(p_R^-)}{\tau(\rho)}{p_L\alpha_i^-q'\alpha_{r+j}}{},
\]
a well-defined almost crossing as required.
\end{proof}

The rest of this subsection is to show the following converse of Lemma \ref{lem:1-sided e-rigid}:

\begin{lemma}\label{lem:1-sided with self-intersection non-rigid}
Let $\omega=\sigma(\omega)$ be a primitive closed curve and $(\omega,\lambda)$ be a 1-sided $\epsilon$-indecomposable object.
If $\omega$ has self-intersection, then $(\omega,\lambda)$ is not $\epsilon$-rigid.
\end{lemma}

We write the $\omega = \omega_1\omega_2\cdots \omega_{2r}$ where $\omega_i$ are arrows or formal inverse of arrows.  The fact that the length $2r$ of $\omega$ is even comes from Lemma \ref{lem:sigma rotates 1-scc}, which also says that $\nabla(\omega_i)=\omega_{i+r}$.
Note that throughout, indices appearing in $\omega$ (and its rotations and reflection) are always taken modulo $2r$.

An self-intersection of $\omega$ is given by a(n almost) crossing:
\begin{align}\label{eq:self crossing}
\crosswn{\omega'}{\xi'_Lb_L^-}{\omega_C}{b_R\xi'_R}{\omega}{\xi_La_L}{a_R^-\xi_R}
\end{align}
where $\omega'$ is a primitive band obtained from some rotation of $\omega^{\kappa}$ for some $\kappa\in\{+1,-1\}$, the symbols $a_L, a_R, b_L, b_R$ are arrows, and all other symbols are (possibly trivial) subwords.
Note that we do not need to work $\biinf{\omega}$ and only focus on the primitive band since we are looking at self-intersection.

For all $k\in \Z$, write $\omega_k = a_k^{\epsilon_i}$ for an arrow $a_i$ and a sign $\epsilon_k\in\{+1,-1\}$.
Let $i,i',l$ be integers so that 
\begin{align*}
b_L^-\omega_C b_R &= \omega_{i'}\omega_{i'+\kappa}\cdots \omega_{i'+\kappa l} = a_{i'}^{\epsilon_j} a_{i'+\kappa}^{\epsilon_{i'+\kappa}} \cdots a_{i'+\kappa l}^{\epsilon_{i'+\kappa l}}, \\
 \text{and}\quad a_L\omega_C a_R^- &= \omega_i\omega_{i+1}\cdots \omega_{i+l} = a_i^{\epsilon_i} a_{i+1}^{\epsilon_{i+1}} \cdots a_{i+l}^{\epsilon_{i+l}}.
\end{align*}
In particular, we have $\epsilon_i = +1 = \epsilon_{i'+\kappa l}$, and $\epsilon_{i'}=-1=\epsilon_{i+l}$.

The following result says that we can always assume the overlap of the crossing comes from two different copies of the subinterval $\omega_C$ in $\omega$.

\begin{lemma}\label{lem:1-scc self-int, far apart}
Suppose that $\omega_0$ is of minimal possible length.

Then $I':=\{i'+\kappa, \ldots, i'+\kappa l-1\}$ and $I:=\{i+1, i+2, \ldots, i+l-1\}$ are distinct intervals of $\Z/2r\Z$.
\end{lemma}
\begin{proof}
Clear the two intervals cannot be identical; otherwise, $\omega_0$ is not an overlap of a crossing. This means that it is sufficient to show that $i',i'+\kappa l\notin I$ and $i,i+l\notin I'$.  We can further reduce to showing $i'\notin I$ and $i\notin I'$, as the other case follows by inverting both $\omega$ and $\omega'$.  We present the argument for $i'\notin I$; the argument for $i\notin I'$ is analogous.

Consider first the case when $\kappa=1$.
We have that $a_{i'}^-$ is a substring in $\omega$ aligned with $a_{2i'-i}$ in $\omega'$, and $a_{i+l}$ is a substring in $\omega'$ aligned with $a_{i+l-i'}$ in $\omega$.
This means that the crossing is either of the form
\[
\left\{\begin{array}{rr|ccccc|l}
\omega' = &  \cdots a_{i'}^- & \cdots & a_{i+l}^- &\cdots& a_{2i'-i}^- &\cdots  & a_{i'+l}^+\cdots \\
\omega = & \cdots a_i^+ &\cdots & a_{i+l-i'}^- &\cdots& a_{i'}^- & \cdots  & a_{i+l}^- \cdots ,
\end{array}\right. 
\]
or of a similar form where the column indexed by $(i+l,i+l-i')$ appears on the right of the column indexed by $(2i'-i,i')$.

This implies that the substring $a_{i'+1}^{\epsilon_{i'+1}}\cdots a_{i+l-1}^{\epsilon_{i+l-1}}$ of $\omega$ coincides with the substring $a_{i+1}^{\epsilon_{i+1}}\cdots a_{i+l-i'}^{\epsilon_{i+l-i'}}$, and so we can find a new crossing:
\begin{align*}
\crosswn{\omega'}{\cdots a_{2i'-i}^-}{\cdots}{a_{i'+l}^+\cdots}{\omega}{\cdots a_i^+}{a_{i+l-i'}^- \cdots,}
\end{align*}
where the overlap is of short length than the original one.  
Hence, this contradicts the minimality assumption on $\omega_0$.

Consider now the case when $\kappa=-1$. If on the contrary that $i'\in I$, then as indices decrease as we go right on $\omega'$, the substring $a_{i'}^{-\epsilon_{i'}}\subset \omega_C \subset \omega$ will align with $a_i^{-\epsilon_{i}}\subset \omega_C\subset \omega'$.  But this means that $\epsilon_i=\epsilon_{i'}$, a contradiction.
\end{proof}

Lemma \ref{lem:1-scc self-int, far apart} says that, possibly after inverting $\omega$, we can write
\begin{align}
\omega = b_U^- \omega_C^{\kappa} b_V \, \theta  \, a_L \omega_C a_R^-\, \theta', \label{eq:omega form2}
\end{align}
for some substring $\theta,\theta'$ such that $(U,V)=(L,R)$ if $\kappa=1$; otherwise, $(U,V)=(R,L)$.
Note that the part $b_V\theta a_L$ can possibly contract to a single letter $b_V=a_L$; likewise $a_R^-\theta'b_U$ may contract to a single letter, i.e. $\omega = b_U^-\omega_C^\kappa b_V \theta a_L\omega_C$.

Since $\omega$ is self-dual, applying $\nabla$ to \eqref{eq:self crossing} yields a new self-crossing
\begin{align}\label{eq:dual self crossing}
\crosswn{\omega}{\nabla(\xi_L)\sigma(a_L)^-}{\nabla(\omega_C)}{\sigma(a_R)\nabla(\xi_R)}{\omega'}{\nabla(\phi_L)\sigma(b_L)}{\sigma(b_R)^-\nabla(\phi_R).}
\end{align}
From this, one should expect that we can arrange the two copies of $\omega_C$ in only half of $\omega$ (and the other two copies of $\nabla(\omega_C)$) in the other half; in particular, we should only have at most one of $b_V\theta a_L$ and $a_R^-\theta ' b_U$ contracting to a single letter.

\begin{lemma}\label{lem:1-scc self-int, first half}
There is a walk $\omega''$ that is equivalent to $\omega$ as a band, so that we can arrange both copies of $\omega_C$ in $a_L\omega_C a_R^-$ and in $b_U^-\omega_C^\kappa b_V$ to lie in an $\epsilon$-factor $\rho$ of $\omega$, namely, that 
\begin{align}
\omega'' &= b_U^-\rho\sigma(b_U)\nabla(\rho) \notag \\
 &= b_U^- \underbrace{\omega_C^\kappa b_V \theta a_L \omega_C a_R^- \theta'}_{\rho} \sigma(b_U) \underbrace{\nabla(\omega_C^\kappa) \sigma(b_V)^- \nabla(\theta) \sigma(a_L)^- \nabla(\omega_C) \sigma(a_R) \nabla(\theta')}_{\nabla(\rho)}, \label{eq:omega rho}
\end{align}
for some substring $\theta,\theta'$ with the possibility that $b_V\theta a_L$ contracts to $b_V = a_L$.
\end{lemma}
\begin{proof}
Assume $\omega$ is in the form of \eqref{eq:omega form2}.  Recall that $a_L$ is positioned at the $i$-th letter in the walk $\omega$.  If $i<r$, then we can just take $\omega''=\omega$.

Suppose that $i\geq r$, by Lemma \ref{lem:sigma rotates 1-scc}, we have $\sigma(b_U)\nabla(\omega_C^{\kappa})\sigma(b_V)^- = \omega_{r}\cdots \omega_{r+l}$ and also $\sigma(a_L)^-\nabla(\omega_C)\sigma(a_R) = \omega_{i-r}\cdots \omega_{i-r+l}$.
Note that by Lemma \ref{lem:1-scc self-int, far apart}, $a_R$ must lie before the end of the walk $\omega$ (i.e. $i+l<2r$).

For ease of reading, let $\beta:=b_U^-\omega_C b_V$ and $\alpha:=a_L\omega_C a_R^-$, so $\omega$ is of the form
\[
\omega = \beta \phi \nabla(\alpha) \phi'\nabla(\beta)\nabla(\phi)\alpha\nabla(\phi')
\]
for some substring $\phi,\phi'$.
By rotate $\omega^{-}$ to a new primitive band that starts with $\beta^-$, we obtain
\[
\omega'' = \beta^- \nabla(\phi')^- \alpha^- \nabla(\phi)^- \nabla(\beta)^- (\phi')^-  \nabla(\alpha)^- \phi^-.
\]
The resulting band is of the form as claimed by swapping the role of $b_U, \omega_C, a_L$ with $b_V, \omega_C^{-}, a_R$ respectively.
\end{proof}

Let us rotate $\omega$ again so that it takes the form
\begin{align}\label{eq:omega final}
\omega = \begin{cases}
\omega_C a_R^- \theta' \nabla(b_L^-\rho)b_L^-\omega_C b_R\theta a_L, &\text{if }\kappa=+1;\\
a_R^- \theta' \nabla(b_R^-\rho) b_R^-\omega_C^-b_L\theta a_L\omega_C, &\text{if }\kappa=-1.
\end{cases}
\end{align}

Consider the following string
\begin{align}
\eta &:= \begin{cases}
\omega\rho, & \text{if }\kappa=+1;\\
\rho^-\omega & \text{if }\kappa=-1.
\end{cases} 
\end{align}

We define another string $\delta$ from $\nabla(\eta)$ as follows
\begin{align}
\nabla(\eta) &= \begin{cases}
\nabla(\omega_Ca_R^-\theta')b_L^-\omega_Cb_R\theta a_L \overbrace{\omega_Ca_R^-\theta'\nabla(b_L^-\omega_C b_R\theta a_L\rho)}^{\delta} & \text{if }\kappa=+1;\\
& \\
\underbrace{\nabla(\rho^-a_R^-\theta')b_R^-\omega_C b_L\theta a_L\omega_C}_{\delta} a_R^- \theta'\nabla(b_R^-\omega_C^-b_L\theta a_L\omega_C) & \text{if }\kappa=-1.
\end{cases}
\end{align}

We have now completed the setup needed.

\begin{proof}[Proof of Lemma \ref{lem:1-sided with self-intersection non-rigid}]

We prove the claim for the case when $\kappa=+1$ by constructing the following commutative diagram where all rows and columns are exact sequences; for the case when $\kappa=-1$ one just needs to modify the crossings $\mt{a}, \mt{b}, \mt{c}, \mt{d}, \mt{h}$.
\[
\xymatrix@C=45pt@R=35pt{
 & & 0\ar[d] & 0\ar[d] & \\
0\ar[r] & M(\rho)\ar[r]^{f_{\mt{a}}-\lambda f_{\mt{b}}}\ar@{=}[d] & M(\eta) \ar[d]_{{\left(\begin{smallmatrix}
f_{\mt c}\\ f_{\mt d}\end{smallmatrix}\right)}}\ar[r]^{g_{\mt h}^{\mathrm{bs}}} & M_\lambda(\omega)\ar[d]^{ g_{\nabla\mt{h}}^{\mathrm{sb}}} \ar[r] & 0 \\
0\ar[r] & M(\rho)\ar[r]^{{\left(\begin{smallmatrix}
-\lambda f_{\mt x} \\ f_{\mt y}
\end{smallmatrix}\right)}\phantom{aaa}} & {\begin{array}{c}
M(\delta)\\ \oplus \\ M(\nabla(\delta))
\end{array}} \ar[d]_{\lambda(f_{\nabla\mt{y}}, f_{\nabla\mt{x}})}\ar[r]^{(\lambda f_{\nabla\mt{d}},f_{\nabla\mt{c}})}& M(\nabla(\eta)) \ar[r]\ar[d]^{f_{\nabla\mt{a}}-\lambda f_{\nabla\mt{b}}} &0 \\
& &  M(\nabla(\rho))\ar[d]\ar@{=}[r] & M(\nabla(\rho))\ar[d] & \\
 & & 0 & 0 & 
}
\]

For ease of reading, we use $\cdots$ instead of writing the full walks as long as there is no confusion, and use a cross $\times$ to denote an empty entry.  To further reduce complication, we take  
\[\phi_A:=\omega_C a_R^-\theta' \;\;\text{ and }\;\; \phi_B:=\omega_Cb_R\theta a_L,
\]
which means that 
\[\rho = \phi_B\phi_A,\;\;\; \omega = \phi_A \sigma(b_L) \nabla(\phi_B\phi_A) b_L^- \phi_B,\;\;\; \text{ and } \delta = \phi_A\sigma(b_L)\nabla(\phi_B\rho).
\]
The almost crossings we needed are
\begin{align*}
\mt{a} &:= \crosswn{\phantom{\nabla()}\rho}{\times}{\rho}{\times}{\eta}{\cdots b_R\theta a_L}{\cdots\phantom{a},} \\
\mt{b} &:= 
\crosswn{\phantom{\nabla()}\rho}{\times}{\omega_C}{b_R\theta a_L\phi_A }{\eta}{\times}{a_R^-\cdots\rho,}\\
\mt{c} &:= \crosswn{\phantom{\nabla()}\eta}{\times}{\phi_A\sigma(b_L)\nabla(\phi_B)\nabla(\omega_C)}{\sigma(a_R)\cdots}{\delta}{\times}{\sigma(b_R)^-\cdots,}\\
\mt{d} &:= \crosswn{\eta}{\cdots \sigma(a_L)^-}{\nabla(\delta)}{\times}{\nabla(\delta)}{\times}{\times}\\
\mt{h} &:= \crosswn{\eta}{\times}{\omega\omega_C}{b_R\cdots}{\biinf{\omega}}{{}^\infty\omega \omega_C \cdots a_L}{a_R^-\cdots a_L \omega^\infty,}\\
\end{align*}

These yield the following compositions:
\begin{align*}
\mt{c}\cdot\mt{a} &=  0  = \mt{d}\cdot\mt{b},\\
\mt{c}\cdot\mt{b} =: \mt{x} &= \crosswn{\phantom{\nabla()}\rho}{\times}{\omega_C}{b_R\cdots}{\delta}{\times}{a_R^-\cdots,}\\
\mt{d}\cdot\mt{a} =: \mt{y} &= \crosswn{\rho}{\times}{\rho}{\times}{\nabla(\delta)}{\cdots a_L}{\times,}\\
\nabla\mt{c}\cdot\mt{y} =: \mt{w} &= \crosswn{\rho}{\times}{\omega_C}{b_R^-\cdots}{\nabla(\eta)}{\nabla(\phi_A)b_L^-\omega_Cb_R\theta a_L}{a_R^-\theta'\nabla(b_L^-\phi_B \rho)}\\
&= \nabla\mt{d}\cdot\mt{x},\\
\mt{h}\cdot\mt{a} &= \crosswn{\rho}{\times}{\omega_C}{b_R\theta a_L \phi_A}{\biinf{\omega}}{{}^\infty\omega\omega_C\cdots a_L}{a_R^-\cdots a_L\omega^\infty}\\
&=p^{-1}(\mt{h}\cdot\mt{b}),\\
\nabla\mt{h}\cdot \mt{h}=:\mt{u} &= \crosswn{\eta}{\phi_A\sigma(b_L)\nabla(\phi')}{\nabla(\phi_A)b_L^-\phi_B\omega_C}{b_R\cdots}{\nabla(\eta)}{\times}{a_R^-\cdots}\\
&= \nabla\mt{c}\cdot\mt{d},
\end{align*}
\begin{align*}
\mt{v}&:=p^{-1}(\nabla\mt{h})\cdot \mt{h}\\
&=\crosswn{\eta}{\times}{\phi_A\sigma(b_L)\nabla(\phi_B)\nabla(\omega_C)}{\sigma(a_R)\cdots}{\nabla(\eta)}{\cdots a_L}{\sigma(b_R)^-\nabla(\theta a_L)\nabla(\phi_A)}\\
&=\nabla\mt{d}\cdot\mt{c}.
\end{align*}

From these datum (and that $\lambda=\lambda^{-1}$), it is easy to verify the commutation of the left-hand square and the bottom square, as well as the exactness of the second row and the left-hand column.

For the exactness of the first row, since $p(\mt{h}\cdot\mt{a})=\mt{h}\cdot\mt{b}$, we have \[
g_{\mt{h}}^{\mathrm{bs}}f_{\mt{b}} = g_{\mt{h}\cdot\mt{b}}^{\mathrm{bs}} = \lambda g_{\mt{h}\cdot\mt{a}}^{\mathrm{bs}} = \lambda g_{\mt{h}}^{\mathrm{bs}}f_{\mt{a}}
\]
which yields the required exactness.  The exactness of the right-hand column then follows by applying $\nabla$.

Finally, for the commutation of the upper right squares, by \eqref{eq:compose band string maps} we have $g_{\nabla\mt{h}}^{\mathrm{sb}}g_{\mt{h}}^{\mathrm{bs}} = \lambda^{-1} f_{\mt{v}}+f_{\mt{u}}$, whereas the composition through $M(\delta)\oplus M(\nabla(\delta))$ yields $\lambda f_{\mt{v}}+ f_{\mt{u}}$.  This completes the proof.
\end{proof}

\subsection{\texorpdfstring{$\epsilon$}{e}-cluster-tilting object}

\begin{definition}
Let $X\in\C$ be an $\epsilon$-object with $\epsilon$-indecomposable decomposition $X=\bigoplus_{i=1}^n X_i$.
We say that $X$ is an \dfn{$\epsilon$-cluster-tilting object} in $\C$ if 
\begin{itemize}
\item each $X_i$ is an $\epsilon$-rigid object, 
\item $\Ext_\C^1(X_i,X_j)=0$,
\item $X$ is maximal with respect to the above properties, i.e. if $Y$ is an indecomposable $\epsilon$-rigid object $Y$ satisfying $\Ext_\C^1(Y,X_i)=0=\Ext^1(X_i,Y)$ for all $X_i\ncong Y$, then there must be some $j\in\{1,\ldots,n\}$ so that $Y\cong X_j$.
\end{itemize}
\end{definition}

Now we can collect everything we have in the section to obtain our main result.

\begin{theorem}\label{thm:categorification}
The bijections in Theorem \ref{thm:curve corresp} restricts to a correspondence
\begin{align*}
\bfA^\otimes\SM & \leftrightarrow \{\epsilon\text{-rigid objects of }\C_{\wti{\SM}}\}.
\end{align*}
This induces a correspondence
\begin{align*}
\{\text{quasi-triangulations of }\SM\} & \leftrightarrow \{\epsilon\text{-cluster-tilting objects of }\C_{\wti{\SM}}\},
\end{align*}
which restricts to a correspondence
\[
\left\{\begin{matrix}
\sigma\text{-stable triangulations}\\\text{of }\widetilde{\SM}
\end{matrix}\right\} \leftrightarrow \left\{\begin{matrix}
\text{triangulations}\\\text{of }\SM
\end{matrix}\right\} \leftrightarrow \left\{\begin{matrix}
\nabla\text{-stable cluster-tilting}\\\text{objects of }\C_{\wti{\SM}}\end{matrix}\right\}.
\]
\end{theorem}
\begin{proof}
The first bijection follows from combining Lemma \ref{lem:split ramified rigid}, Lemma \ref{lem:1sided non-eps-rigid}, Lemma \ref{lem:1-sided with self-intersection non-rigid}, and Lemma \ref{lem:1-sided e-rigid}.
Now the second one follows from Proposition \ref{prop::Ext vanishing}.
In particular, when restricting to triangulations, we only get split $\epsilon$-indecomposables appearing and so $\epsilon$-rigid is just usual rigidity and $\epsilon$-cluster-tilting is the usual cluster-tilting; these yields the last correspondence.
\end{proof}

\begin{example} \label{ex::cluscat}
Consider the fan triangulation $T$ of $\mathcal{M}_2$ from Example \ref{eg:band eg} which is associated to $\mathbb{A}_{2,2}$-quiver.  Fix $\epsilon\in\{\pm 1\}$.
It is well-known that $\C_{\wti{T}}$ has infinitely many rigid objects (corresponding to arcs on $\wti{\SM}$) and cluster tilting objects (corresponding to triangulations on $\wti{\SM}$), but not many of them are give rise to indecomposable $\epsilon$-rigid object.  For example, if we consider the arc $\alpha\in \bfA\wti{\SM}$ corresponding to the string $b_2$, then $\alpha\oplus \nabla(\alpha)$ non-rigid in $\C$.
Note that $\alpha,\sigma(\alpha)$ are the lifts of the self-intersecting curve shown in Figure \ref{fig:eg1}; see Figure \ref{fig:cluster cat eg crossing} where we also display the structure of $M(\alpha)$ and $M(\sigma\alpha)$.

\begin{figure}[!htbp]
\centering
\scalebox{0.8}{
\begin{tikzpicture}[scale=0.7]
    \draw[blue, very thick]  plot[smooth, tension=.7] coordinates {(0,2) (-1.5,0.2) (0,-1.5) (1,-0.5) (1,0.5) (0,0.75)};
    \draw[red, very thick]  plot[smooth, tension=.7] coordinates {(0,-2) (-1.5,-1) (-1.5,1) (0,1.5) (1.5,.8) (1.5,-1) (0,-0.75)};
\Atwotwo{0,0}
\node[red,thick] at (2.5,0.5) {$\sigma(\alpha)$};
\node[blue,thick] at (-1,0) {$\alpha$};

\begin{scope}[shift={(-5,0)}]
\draw[very thick,blue] (0,-2) .. controls (-0.2,-1.5) and (-1,-0.5) .. (0,0.1) 
	.. controls (1,0.5) and (1,0) .. (1,-0.5) 
	.. controls (1,-1.5) and (-1,-1.5) .. (-1,0) node[blue,left] {$\overline{\alpha}$} 
	.. controls (-1,0.5) and (-1,1) .. (0,2);
\Mtwofan{0,0}
\end{scope}
\begin{scope}[shift={(1.5,-0.25)},scale=.9]
\node at (3.5,2) {$M(\alpha)=$};\node at (3.5,-1) {$M(\sigma\alpha)=$};
\node at (8,2) {$=\;\begin{matrix}1\\2\end{matrix}$};
\node at (8,-1) {$=\;\begin{matrix}2'\\1'\end{matrix}$};

\node (v1) at (5,2) {$\Bbbk$};\node (v2) at (6,1) {$\Bbbk$};\node (v3) at (6,3) {$0$};\node (v4) at (7,2) {$0$};
\draw[->] (v1) --node[pos=.05,below] {\footnotesize id} (v2); \draw[->] (v1) -- (v3); \draw[->] (v2) -- (v4); \draw[->] (v3) -- (v4);

\node (v1) at (5,-1) {$0$};\node (v2) at (6,-2) {$0$};\node (v3) at (6,0) {$\Bbbk$};\node (v4) at (7,-1) {$\Bbbk$};
\draw[->] (v1) -- (v2); \draw[->] (v1) -- (v3); \draw[->] (v2) -- (v4); \draw[->] (v3) --node[pos=.7,above] {\footnotesize\; id} (v4);
\end{scope}
\end{tikzpicture}}
\caption{Self-crossing arc, its lift, and corresponding non-$\epsilon$-rigid object}\label{fig:cluster cat eg crossing}
\end{figure}
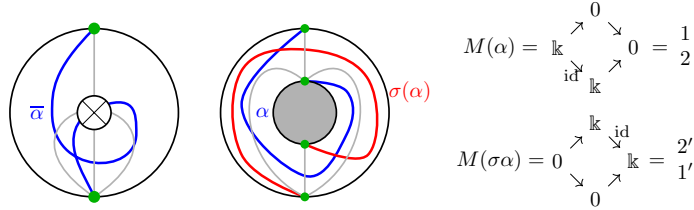

There are only finitely many indecomposable object $X\in \C_{\wti{T}}$ such that $X\oplus\nabla(X)$ is rigid, which means that there are only finitely many $\epsilon$-rigid objects of split string type.  The two obvious one are the initial arcs in $\wti{T}\subset \wti{\SM}$, which correspond to the initial arcs in $T\subset \SM$.  The remaining ones are given in Figure \ref{fig:M2fan eps-obj}, where we display from top to bottom their structures as indecomposable $\epsilon$-representations over $J$, the corresponding arcs in $\wti{\SM}$, and the corresponding arcs in $\SM$.
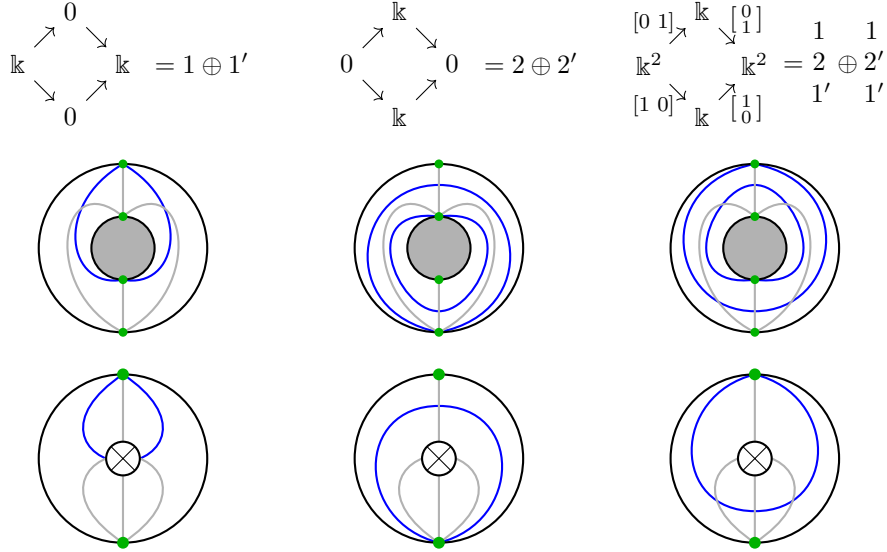
\begin{figure}[!htbp]
\centering
\begin{tikzpicture}[scale=0.7]
\begin{scope}[shift={(-12,0.5)},scale=0.8]
\draw[thick, blue] (0,2) .. controls (1.7,1) and (1.3,-1) .. (0,-0.75);
\draw[thick, blue] (0,2) .. controls (-1.7,1) and (-1.3,-1) .. (0,-0.75);
\Atwotwo{0,0}
\draw[thick,blue,scale=-1,shift={(0,5)}] (0,-2) .. controls (-1.5,-1) and (-1,0) .. (0,0.1) .. controls (1,0) and (1.5,-1) .. (0,-2);
\Mtwofan{0,-5}
\end{scope}

\begin{scope}[shift={(-6,0.5)},scale=0.8]
\draw[thick,blue] (0,-2) .. controls (-2.5,-1.5) and (-2,1.5) .. (0,1.5) .. controls (2,1.5) and (2.5,-1.5) .. (0,-2);
\draw[thick,blue] (0,.75) .. controls (-2,1) and (-1,-1.5) .. (0,-1.5) 
						.. controls (1,-1.5) and (2,1) .. (0,.75);
\Atwotwo{0,0}
\draw[thick,blue,shift={(0,-5)}] (0,-2) .. controls (-2,-1.5) and (-2,1.25) .. (0,1.25) .. controls (2,1.25) and (2,-1.5) .. (0,-2);
\Mtwofan{0,-5}
\end{scope}

\begin{scope}[shift={(0,0.5)},scale=0.8]
\draw[thick,blue,scale=-1] (0,-2) .. controls (-2.5,-1.5) and (-2,1.5) .. (0,1.5) .. controls (2,1.5) and (2.5,-1.5) .. (0,-2);
\draw[thick,blue,scale=-1] (0,.75) .. controls (-2,1) and (-1,-1.5) .. (0,-1.5) 
						.. controls (1,-1.5) and (2,1) .. (0,.75);
\Atwotwo{0,0}
\draw[thick,blue,scale=-1,shift={(0,5)}] (0,-2) .. controls (-2,-1.5) and (-2,1.25) .. (0,1.25) .. controls (2,1.25) and (2,-1.5) .. (0,-2);
\Mtwofan{0,-5}
\end{scope}

\node (u1) at (-14,4) {$\Bbbk$};\node (u2) at (-13,3) {$0$};\node (u3) at (-13,5) {$0$};\node (u4) at (-12,4) {$\Bbbk$};
\draw[->] (u1) -- (u2); \draw[->] (u1) -- (u3); \draw[->] (u2) -- (u4); \draw[->] (u3) -- (u4);
\node at (-10.5,4) {$=1\oplus 1'$};

\node (v1) at (-7.75,4) {$0$};\node (v2) at (-6.75,3) {$\Bbbk$};\node (v3) at (-6.75,5) {$\Bbbk$};\node (v4) at (-5.75,4) {$0$};
\draw[->] (v1) -- (v2); \draw[->] (v1) -- (v3); \draw[->] (v2) -- (v4); \draw[->] (v3) -- (v4);
\node at (-4.25,4) {$=2\oplus 2'$};

\node (w1) at (-2,4) {$\Bbbk^2$};\node (w2) at (-1,3) {$\Bbbk$};\node (w3) at (-1,5) {$\Bbbk$};\node (w4) at (0,4) {$\Bbbk^2$};
\draw[->] (w1) --node[pos=.1,below]{\footnotesize $[1~0]\quad\;\;$} (w2); 
\draw[->] (w1) --node[pos=.1,above]{\footnotesize $[0~1]\quad\;\;$} (w3); 
\draw[->] (w2) --node[pos=.9,below]{\quad$\begin{bsmallmatrix}1\\0\end{bsmallmatrix}$} (w4); 
\draw[->] (w3) --node[pos=.9,above]{\quad$\begin{bsmallmatrix}0\\1\end{bsmallmatrix}$} (w4);
\node at (1.5,4) {$={\begin{matrix}1\,\\2\,\\1'\end{matrix}}\oplus {\begin{matrix}1\,\\2'\\1'\end{matrix}}$};
\end{tikzpicture}\caption{Non-initial $\epsilon$-rigid objects of split string types on $\mathcal{M}_2$}\label{fig:M2fan eps-obj}
\end{figure}

There is one more indecomposable $\epsilon$-rigid object, namely the unique 1-sided $\epsilon$-indecomposable $(\omega,\epsilon)$ corresponding to the quasi-arc; see Figure \ref{fig:doublecover}.
\end{example}

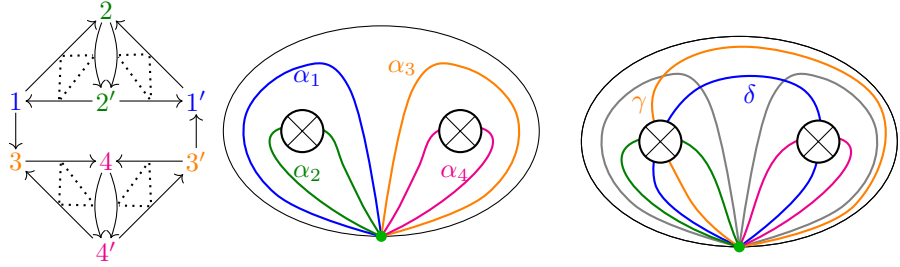
\begin{figure}[!htbp]
\centering
\begin{tikzpicture}[
	thicDotted/.style={dotted, line width=.8pt},
	midArrow/.style={midway, inner sep=0}
]
\begin{scope}[scale=0.7, shift={(-2.2,0)}]
\draw  (0,0) ellipse (3 and 2);

\draw[blue,thick]  plot[smooth, tension=.7] coordinates {(0,-2) (-2,-1.2) (-2.6,0.2) (-1.6,1.2) (-0.6,0.8) (0,-2)};
\node[blue] at (-1.4,1) {$\alpha_1$};
\draw[green!50!black,thick] (0,-2) .. controls (-1.4,-1.4) and (-3,0.1) .. (-1.55,0.1)
					.. controls (-0.8,0.1) and +(-.2,.4) .. (-0.7,-0.7) 
					.. controls (-0.5,-1.1) and (-0.4,-1.4) .. (0,-2);
\node[green!50!black] at (-1.4,-0.8) {$\alpha_2$};
\xcap{-1.5,0}

\begin{scope}[xscale=-1]
\draw[orange,thick]  plot[smooth, tension=.7] coordinates {(0,-2) (-2,-1.2) (-2.6,0.2) (-1.6,1.2) (-0.6,0.8) (0,-2)};
\node[orange] at (-0.4,1.2) {$\alpha_3$};
\draw[magenta,thick] (0,-2) .. controls (-1.4,-1.4) and (-3,0.1) .. (-1.55,0.1)
					.. controls (-0.8,0.1) and +(-.2,.4) .. (-0.7,-0.7) 
					.. controls (-0.5,-1.1) and (-0.4,-1.4) .. (0,-2);
\node[magenta] at (-1.4,-0.8) {$\alpha_4$};
\end{scope}
\xcap{1.5,0}
\fill[darkgreen] (0,-2) circle (3pt);
\end{scope}

\node[inner sep=1pt,blue] (v1) at (-6.4,0.4) {$1$};\node[inner sep=1pt,blue] (v1') at (-4,0.4) {$1'$};
\node[inner sep=1pt,green!50!black] (v2') at (-5.2,0.4) {$2'$};\node[inner sep=1pt,green!50!black] (v2) at (-5.2,1.6) {$2$};
\node[inner sep=1pt,orange] (v3) at (-6.4,-0.4) {$3$};\node[inner sep=1pt,orange] (v3') at (-4,-0.4) {$3'$};
\node[inner sep=1pt,magenta] (v4) at (-5.2,-0.4) {$4$};\node[inner sep=1pt,magenta] (v4') at (-5.2,-1.6) {$4'$};
\draw[->]  (v1) edge node[midArrow](a12){} (v2); 
\draw[->]  (v2') edge node[midArrow](a2'1'){} (v1');
\draw[->]  (v2) edge[bend right=20] node[midArrow](a22'l){} (v2'); 
\draw[->]  (v2) edge[bend left=20]  node[midArrow](a22'r) {} (v2');
\draw[->] (v2') edge node[midArrow](a2'1){} (v1); 
\draw[->] (v1')edge node[midArrow](a1'2){}(v2);
\draw[thicDotted] (a12) -- (a22'l) -- (a2'1) -- (a12) (a2'1')--(a22'r)--(a1'2)--(a2'1');
\draw[->]  (v1) edge (v3); \draw[->] (v3') edge (v1');
\draw[->]  (v3) edge node[midArrow](a34){} (v4); \draw[->] (v4') edge node[midArrow](a4'3'){} (v3');
\draw[->]  (v4') edge node[midArrow](a4'3){} (v3); \draw[->]  (v3') edge node[midArrow](a3'4){} (v4); 
\draw[->]  (v4) edge[bend right=20] node[midArrow](a44'r){} (v4'); \draw[->]  (v4) edge[bend left=20] node[midArrow](a44'l){} (v4');
\draw[thicDotted] (a4'3) -- (a34) -- (a44'r) -- (a4'3) (a3'4) -- (a4'3') -- (a44'l) -- (a3'4);

\begin{scope}[scale=0.7, shift={(4.6,-0.2)}]
\draw  (0,0) ellipse (3 and 2);

\draw[gray,thick]  plot[smooth, tension=0.7] coordinates {(0,-2) (-2,-1.2) (-2.6,0.2) (-1.6,1.2) (-0.6,0.8) (0,-2)};
\begin{scope}[xscale=-1]
\draw[gray,thick]  plot[smooth, tension=0.7] coordinates {(0,-2) (-2,-1.2) (-2.6,0.2) (-1.6,1.2) (-0.6,0.8) (0,-2)};
\end{scope}

\draw  (0,0) ellipse (3 and 2);
\draw[thick,blue]  plot[smooth, tension=0.7] coordinates {(0,-2) (-1.6,-0.8) (-0.8,1) (1.2,1) (1.4,-0.6) (0,-2)};

\draw[orange,thick]  plot[smooth, tension=0.7] coordinates {(0,-2) (-1,-1) (-1.6,1) (0.2,1.8) (2.6,0.8) (2.2,-1.2) (0,-2)};
\draw[green!50!black,thick] (0,-2) .. controls (-2,-1.4) and (-3,0) .. (-1.55,0.1)
					.. controls (-0.4,0.1)  and (-0.4,-1.2) .. (0,-2);
\xcap{-1.5,0}

\begin{scope}[xscale=-1]
\draw[magenta,thick] (0,-2) .. controls (-1.4,-1.4) and (-3,0.1) .. (-1.55,0.1)
					.. controls (-0.4,0.1) and (-0.4,-1.2) .. (0,-2);
\end{scope}
\xcap{1.5,0}
\fill[darkgreen] (0,-2) {} circle (3pt);

\node[orange] at (-1.9,0.7) {$\gamma$};
\node[blue] at (0.2,.9) {$\delta$};
\end{scope}
\end{tikzpicture}
\caption{Triangulations involved in Example \ref{eg:N2eg}}\label{fig:N2eg}
\end{figure}
\begin{example}\label{eg:N2eg}
We consider $\SM$ a genus 2 non-orientable surface with 1 boundary component and 1 marked point.  
Let $T=\{\alpha_1,\alpha_2,\alpha_3,\alpha_4\}$  be a triangulation on $\SM$ as shown in the middle of Figure \ref{fig:N2eg}.
We display the gentle algebra as a quiver with relation on the left.
Here $i,i'$ for each $1\leq i\leq 4$ are the vertices corresponding to the lifts $\widetilde{\alpha_i},\sigma(\widetilde{\alpha_i})$ of $\alpha_i$ to $\widetilde{\SM}$, and the dotted line connecting two arrows represents a monomial quadratic relation given by the composition of the two arrows.

Consider the curves $\gamma, \delta$ on $\SM$ as shown on the right.  Let $\widetilde{\gamma},\widetilde{\delta}$ be a lift to $\widetilde{\SM}$.  
Then \[
M(\widetilde{\gamma})\oplus \nabla M(\widetilde{\gamma}) = 1\oplus 1',\text{ and } M(\widetilde{\delta})\oplus \nabla M(\widetilde{\gamma}) = \begin{matrix}
1 \\ 3\end{matrix}\oplus \begin{matrix}
3' \\ 1'\end{matrix}.\]
The set $\{\gamma,\delta,\alpha_2,\alpha_4\}$ form a triangulation of $\SM$, and so the object
\[
\widetilde{\gamma}\oplus \widetilde{\delta}\oplus\widetilde{\alpha_2}\oplus\widetilde{\alpha_4} \oplus \nabla(\widetilde{\gamma}\oplus \widetilde{\delta}\oplus\widetilde{\alpha_2}\oplus\widetilde{\alpha_4})
\]
is an $\epsilon$-cluster tilting object in $\C_{\widetilde{\SM}}$.
\end{example}

\bibliographystyle{unsrt} 
\bibliography{biblio}

\end{document}